\documentclass[12pt]{amsart}
\usepackage{graphicx}
\usepackage{amsmath,bm,amssymb,harvard}
\usepackage{color}
\usepackage{epstopdf}
\usepackage[svgnames]{xcolor}
\usepackage{txfonts}
\usepackage{bm}%
\theoremstyle{plain}%
\newtheorem{theorem}{Theorem}[section]

\newtheorem{prop}[theorem]{Proposition}
\newtheorem{cor}[theorem]{Corollary}
\newtheorem{lem}[theorem]{Lemma}
\theoremstyle{definition}
\newtheorem{defn}[theorem]{Definition}
\newtheorem{example}[theorem]{Example}
\newtheorem{thm}{Theorem}[section]
\theoremstyle{remark}
\newtheorem{rem}[theorem]{Remark}

\numberwithin{equation}{section}

\renewcommand{\Im}{{\mathcal I}}

\newcommand{\rr}{\mathbb R}
\newcommand{\E}{\mathbb E}

\newcommand{\bbR}{{\Bbb R}}
\newcommand{\bbN}{{\Bbb N}}
\newcommand{\bbC}{{\Bbb C}}

\newcommand{\bbE}{{\Bbb E}}

\newcommand{\imag}{{\mathbf i}}
\newcommand{\tr}{{\textnormal{tr}}}

\renewcommand{\cite}{\citeyear}

\begin{document}
\sloppy

\title[On multivariate fractional random fields: tempering and operator-stable laws]{On multivariate fractional random fields: tempering and operator-stable laws}

\author{Gustavo Didier}
\address{Gustavo Didier, Mathematics Department,
Tulane University, New Orleans, LA, U.S.A.}
\email{gdidier@tulane.edu}

\author{Shigeki Kanamori}
\address{Shigeki Kanamori, Department of Applied Mathematics,
Waseda University, Japan}
\email{shigeki.005@ruri.waseda.jp}

\author{Farzad Sabzikar}
\address{Farzad Sabzikar, Department of Statistics,
Iowa State University, Ames, IA 50011}
\email{sabzikar@iastate.edu}
\urladdr{http://sabzikar.public.iastate.edu/}

\begin{abstract}
In this paper, we define a new and broad family of vector-valued random fields called \textit{tempered operator fractional operator-stable random fields} (TRF, for short). TRF is typically non-Gaussian and generalizes tempered fractional stable stochastic processes. TRF comprises moving average and harmonizable-type subclasses that are constructed by tempering (matrix-) homogeneous, matrix-valued kernels in time- and Fourier-domain stochastic integrals with respect to vector-valued, strictly operator-stable random measures. We establish the existence and fundamental properties of TRF. Assuming both Gaussianity and isotropy, we show the equivalence between certain moving average and harmonizable subclasses of TRF. In addition, we establish sample path properties in the scalar-valued case for several Gaussian instances.
\end{abstract}

\maketitle

\section{Introduction}\label{sec:inro}


In this paper, we define a new and broad family of vector-valued random fields called \textit{tempered operator fractional operator-stable random fields} (TRF, for short). TRF is typically non-Gaussian and generalizes tempered fractional stable stochastic processes. It ties together the research literatures on stable laws, anisotropic operator scaling, semi-long range dependence as well as transient anomalous diffusion in physics. TRF comprises moving average and harmonizable-type subclasses that are constructed by tempering (matrix-) homogeneous, matrix-valued kernels in time- and Fourier-domain stochastic integrals with respect to vector-valued, strictly operator-stable random measures. We establish the existence and fundamental properties of TRF. Moving average and harmonizable-type instances are generally non-equivalent; however, assuming both Gaussianity and isotropy, we show the equivalence between certain subclasses. In addition, we establish sample path properties in the scalar-valued case for several Gaussian instances.

Fractional, or non-Markovian, constructs provide the mathematical framework for \textit{scale invariant} systems. These systems lack a characteristic time or space scale, and their behavior across scales is related by means of \textit{scaling exponents} (Mandelbrot and Van Ness \cite{mandelbrot:vanness:1968}, Flandrin \cite{flandrin:1992}, Wornell and Oppenheim \cite{wornell:oppenheim:1992}). A cornerstone class of scale invariant stochastic processes is fractional Brownian motion (FBM), i.e., the only Gaussian, self-similar, stationary increment process (Embrechts and Maejima \cite{embrechts:maejima:2002}). The literature on fractional probability theory and its applications is now extensive (e.g., Dobrushin and Major \cite{dobrushin:major:1979}, Granger and Joyeux \cite{granger:joyeux:1980}, Dahlhaus \cite{dahlhaus:1989}, Robinson \cite{robinson:1995-gaussian}, Moulines et al.\ \cite{moulines:roueff:taqqu:2008}, Beran et al.\ \cite{beran:feng:ghosh:kulik:2013}, Bardet and Tudor \cite{bardet:tudor:2014}, Pipiras and Taqqu \cite{pipiras:taqqu:2017}).

Recall that a distribution is called \textit{stable}, in the scalar case, or \textit{operator-stable}, in the vector case, when it can be reexpressed as the sum of independent copies of itself, up to dilation and shift factors (Jurek and Mason \cite{jurek:mason:1993}, Meerschaert and Scheffler \cite{meerschaert:scheffler:2001}). Non-Gaussian stable laws display heavy tails, a property observed in a number of areas such as in finance (Meerschaert and Scheffler \cite{meerschaert:scheffler:2003}) and network traffic (Taqqu et al.\ \cite{taqqu:willinger:sherman:1997}, Willinger et al.\ \cite{willinger:paxson:riedi:taqqu:2003}). On the other hand, in the context of random fields, a (fractional) system is called \textit{anisotropic} when its behavior may be affected by rotations or reflections (e.g., Bonami and Estrade \cite{bonami:estrade:2003}, Didier et al.\ \cite{didier:meerschaert:pipiras:2018:symmetries}). Anisotropy is encountered in several fields of research such as in radiology (Brunet-Imbault et al.\ \cite{brunet-imbault:lemineur:chappard:harba:benhamou:2005}) and texture analysis (Arneodo et al.\ \cite{arneodo:decoster:roux:2000}, Roux et al.\ \cite{roux:abry:vedel:jaffard:wendt:2016}). An important framework of fractional, heavy-tailed anisotropic models that we call \textit{operator scaling} or \textit{operator fractional operator-stable random fields} (ORF) was constructed over the years by multiple authors (see, for instance, Maejima and Mason \cite{maejima:mason:1994} on random processes, Bierm\'{e} et al.\ \cite{bierme:meerschaert:scheffler:2007} on scalar-valued random fields, and Kremer and Scheffler \cite{kremer:scheffler:2019} on vector-valued random fields). We can generally express ORF by means of two non-equivalent classes of random fields defined in the multidimensional time and Fourier domains, respectively, as
\begin{equation}\label{moving_RF_aniso}
\bbR^n \ni X_{H}({\bm x}) = \int_{{\bm y}\in\rr^d} \Big[ \varphi({\bm x}- {\bm y})^{ H- q B } - \varphi(-{\bm y})^{ H - q B } \Big]\ M(d{\bm y}), \quad {\bm x} \in \bbR^d,
\end{equation}
and
\begin{equation}\label{harmo_RF_aniso}
\bbR^n \ni \widetilde{X}_{H}({\bm x}) = {\rm Re} \int_{ {\bm y} \in\rr^d} ( e^{\imag \langle {\bm x},{\bm \xi}\rangle} -1  )\varphi( {\bm \xi} )^{-H}\varphi({\bm \xi})^{-qB}\ M(d{\bm \xi}), \quad {\bm x} \in \bbR^d.
\end{equation}
In \eqref{moving_RF_aniso} and \eqref{harmo_RF_aniso}, $H$ is the so-named Hurst (matrix) exponent, $\varphi > 0$ is a (real-valued) homogeneous function in the sense of matrix scaling (see \eqref{e:E-homogeneous}), for any matrix $A$ the expression $\varphi(\cdot)^{A}$ denotes the matrix exponential  (see \eqref{e:c^E}), $q$ is an appropriate constant, and $M(d{\bm y})$ and $M(d{\bm \xi})$ are operator-stable random measures associated with the matrix exponent $B$ (see \eqref{e:operator-stable_chf} and \eqref{e:M(dx)}). In the scalar case ($n=1$), we can set $B^{-1} = \alpha \in (0,2]$ and the random measures are symmetric $\alpha$-stable ($ S \alpha S $; see Example \ref{ex:SalphaS}). In dimensions $d = n= 1$, \eqref{moving_RF_aniso} and \eqref{harmo_RF_aniso} encompass linear fractional stable motion (LFSM; Stoev and Taqqu \cite{stoev:taqqu:2004}) and FBM ($\alpha = 2$). Fundamental properties have been established for many Gaussian and non-Gaussian instances of ORF, including existence, stochastic continuity, critical H\"older exponents and the Hausdorff dimension of sample paths (see, for instance, Maejima and Mason \cite{maejima:mason:1994}, Bierm\'{e} et al.\ \cite{bierme:meerschaert:scheffler:2007}, S\"{o}nmez \cite{sonmez:2016,sonmez:2018}).

In many empirical settings, power law behavior -- as parametrized by some fixed scaling, or Hurst, exponent -- is expected to hold only within a range of scales. Outside this range, the observed dynamics may qualitatively change. In the anomalous diffusion literature, for example, this appears in the form of \textit{transience}. A particle's position $X = \{X(t)\}$ is said to undergo \textit{transient anomalous diffusion} when its mean squared displacement $\E X^2(t)$ satisfies a scaling relation of the form $\E X^2(t) \approx C t^{\zeta}$, $C \geq 0$, where the exponent $\zeta = \zeta(t) \geq 0$ may itself change as a function of time (Piryatinska et al.\ \cite{piryatinska:sanchev:woyczynski:2005}, Stanislavsky et al.\ \cite{stanislavsky:weron:weron:2008}, Sandev et al.\ \cite{sandev:chechkin:kantz:metzler:2015}, Wu et al.\ \cite{wu:deng:barkai:2016}, Chen et al.\ \cite{chen:wang:deng:2017}, Liemert et al.\ \cite{liemert:sandev:kantz:2017}, Chen et al.\ \cite{chen:wang:deng:2018}, Molina et al.\ \cite{molina:sandev:safdari:pagnini:chechkin:metzler:2018}). Transience may also appear as a consequence of accounting for the energy spectrum of turbulence in the low-frequency range (Meerschaert et al.\ \cite{meerschaert:sabzikar:phanikumar:zeleke:2014}), and is closely related to the property of \textit{semi-long range dependence} (semi-LRD). The increments of a stochastic process are said to display semi-LRD when their autocovariance function decays hyperbolically over small lags and exponentially fast over large lags. A canonical example is tempered FBM (TFBM; Meerschaert and Sabzikar \cite{meerschaert:sabzikar:2013,MeerschaertsabzikarSPA}), a transiently anomalous diffusion process whose increments are semi-long range dependent.

In this paper, we use recently developed operator-stable stochastic integration techniques (Kremer and Scheffler \cite{kremer:scheffler:2017,kremer:scheffler:2019}; see also Samorodnitsky and Taqqu \cite{samorodnitsky:taqqu:1994}) to put forward a broad mathematical framework that combines operator-stable laws, fractional anisotropy and transience. This is done by applying exponential-type dampening techniques to construct tempered fractional extensions of \eqref{moving_RF_aniso} and \eqref{harmo_RF_aniso}. These new families are called \textit{moving average-}, and \textit{harmonizable-tempered operator fractional operator-stable random fields} (MA-TRF and H-TRF, respectively; for precise expressions, see Definitions \ref{defn:MA-TRF} and \ref{defn:H-TRF}). Tempering produces more tractable mathematical objects, and can be made arbitrarily light, which is especially convenient in the development and applications of stochastic fractional calculus (Cartea and del-Castillo-Negrete \cite{cartea:del-castillo-negrete:2007}, Baeumer and Meerschaert \cite{baeumer:meerschaert:2010}, Meerschaert and Sabzikar \cite{MeerschaertsabzikarSPA}, Boniece et al.\ \cite{boniece:didier:sabzikar:2020}). Moreover, we use Bessel-type functions to further define \textit{moving average-Bessel-tempered operator fractional operator-stable random fields} (MA-B-TRF; see Definition \ref{defn:MA-B-TRF}). MA-B-TRF involves flexible Hurst eigenvalue-dependent tempering that turns out to be suitable for the calculation of Fourier transforms. We establish fundamental properties of the three subclasses of TRF (MA-TRF, MA-B-TRF and H-TRF), such as their existence, stochastic continuity and stationarity of increments (see Theorem \ref{thm:exponential_moving_TOSSF}, Corollary \ref{cor:stationary_increments_scaling_property_MA}, Theorem \ref{thm:Bessel_moving_TOSSF}, Corollary \ref{cor:stationary_increments_scaling_property_MA-B}, Theorem \ref{thm:General_harmo_TOSSRF}, Corollary \ref{c:stationary_increments_scaling_property_H}). All of them display operator scaling properties involving both domain ($E$) and range ($H$) matrix exponents. In other words, any TRF $X_{\lambda}$ satisfies
\begin{equation}\label{e:tempered_o.s.s.}
\{ X_{\lambda} (c^E {\bm x}) \}_{{\bm x} \in \mathbb{R}^d} \stackrel{f.d.}{=}
\{ c^H X_{c\lambda} ({\bm x}) \}_{{\bm x} \in \mathbb{R}^d}, \quad c > 0,
\end{equation}
where $\stackrel{f.d.}{=}$ denotes the equality of finite-dimensional distributions and $\lambda > 0$ is a tempering parameter (cf.\ Li and Xiao \cite{li:xiao:2011}, Didier et al.\ \cite{didier:meerschaert:pipiras:2017}). In particular, TRF encompasses tempered fractional stable motion (TFSM) in dimensions $d = n = 1$ (cf.\ Meerschaert and Sabzikar \cite{meerschaert:sabzikar:2016} and Sabzikar and Surgailis \cite{sabzikar:surgailis:2018}; see also Remark \ref{rem:TFSM_TFSMII}).

Establishing the equivalence of moving average and harmonizable representations of Gaussian random fields is a classical problem of great interest in both theory and applications (Rozanov \cite{rozanov:1967}, Brockwell and Davis \cite{brockwell:davis:1991}). In the tempered operator fractional framework, it can be shown that many Gaussian instances of TRF have covariance function
\begin{equation}\label{e:cov_function_under_isotropy}
\bbE X_{H}({\bm x})X_{H}({\bm y})^t = \frac{1}{2} \{\|{\bm x}\|^{H} C^2_{{\bm x},H} \|{\bm x}\|^{H^*} + \|{\bm y}\|^{H}C^2_{{\bm y},H} \|{\bm y}\|^{H^*} - \| {\bm x}-{\bm y} \|^{H}C^2_{{\bm x - y},H} \| {\bm x}-{\bm y} \|^{H^*}  \}, \quad {\bm x},{\bm y} \in \bbR^d
\end{equation}
for $C^2_{\bullet,H} \in {\mathcal S}_{\geq 0}(n,\bbR)$ (see Corollaries \ref{cor100} and \ref{cor100_Bes}). The fact that $C^2_{{\bm x},H}$ is a function of ${\bm x}$ even under isotropy (see \eqref{TFGFIdefn}) makes directly comparing covariance functions for different representations quite intricate. This obstacle is also encountered with \textit{an}isotropic Gaussian instances of \eqref{moving_RF_aniso} and \eqref{harmo_RF_aniso} (cf.\ Bierm\'{e} et al.\ \cite{bierme:meerschaert:scheffler:2007}, p.\ 325; Didier and Pipiras \cite{didier:pipiras:2011}, Baek et al.\ \cite{baek:didier:pipiras:2014}). In this paper, we construct harmonizable representations of isotropic and Gaussian MA- and MA-B-TRF by, instead, computing Fourier transforms of their kernels (Propositions \ref{p:TFGFI harmo} and \ref{p:I-B-TOFBF harmo}). Our results show that the harmonizable representation of isotropic Gaussian MA-B-TRF is mathematically simpler than that of MA-TRF, and that the former is the natural moving average-type counterpart to a subclass of isotropic Gaussian instances of H-TRF (see expression \eqref{I-B-TOFBF harmo_regular-param}). In addition, we provide some closed-form expressions for the covariances of both subclasses of random fields (Propositions \ref{prop2} and \ref{p:cov_low d}).

It is well known that sample path properties of random fields provide measures of fractality and regularity of global and local behavior (see Adler \cite{adler:1981}, Falconer \cite{falconer:1990}). In this paper, we further consider the sample path properties of scalar-valued, Gaussian instances of TRF. We call them (scalar-valued) \textit{tempered operator fractional Brownian fields} (TOFBF). We establish the H\"{o}lder regularity of sample paths, as well as the Hausdorff or box dimensions of the graphs of moving average (MA-TOFBF) and harmonizable (H-TOFBF) types. In particular, our results show that tempering does not affect the sample path properties of MA-TOFBF with respect to its non-tempered, operator fractional Brownian field counterpart, as studied in Bierm\'{e} et al.\ \cite{bierme:meerschaert:scheffler:2007} (on the sample path properties of the related Bessel type, see Remark \ref{r:Bessel_sample_path}).

This paper is organized as follows. In Section \ref{s:preliminaries}, we sum up fundamental concepts and lay out the notation used in the rest of the paper. In Section \ref{sec:time_domain}, we construct and study MA-TRF and MA-B-TRF, namely, in the multidimensional time domain. In Section \ref{sec:frequency_domain}, we construct and study H-TRF, namely, in the multidimensional Fourier domain. In Section \ref{s:Gaussian}, assuming Gaussianity and isotropy, we construct harmonizable representations of MA-TRF and MA-B-TRF, as well as some expressions for their covariance functions. In Section \ref{s:sample_path}, we establish the sample path properties of scalar-valued, Gaussian instances. Section \ref{s:appendix} contains all proofs and technical results. In the Conclusion, we sum up the results and discuss open problems.

\section{Preliminaries}\label{s:preliminaries}

Let ${\mathcal M}(d,\bbR)$, ${\mathcal M}(d,\bbC)$ be, respectively, the spaces of $d \times d$ matrices with $\bbR$-- and $\bbC$--valued entries, and let $O(d), SO(d) \subseteq {\mathcal M}(d,\bbR)$ be, respectively, the orthogonal and special orthogonal groups. Also let $GL(n,\bbC)$ be the general linear group of nonsingular, complex-valued matrices. Define ${\mathcal S}_{\geq 0}(d,\bbR)$, ${\mathcal S}_{\geq 0}(d,\bbC)$ as the cones of symmetric and Hermitian positive semidefinite matrices, respectively. Let $\textnormal{eig}(M)$ be the set of possibly repeated eigenvalues (characteristic roots) of $M \in {\mathcal M}(d,\bbR)$. For notational convenience, we also write
\begin{equation}\label{e:varpi_Upsilon}
\varpi_{M} = \inf \Re\textnormal{eig}(M) , \quad \Upsilon_{M} = \sup \Re\textnormal{eig}(M), \quad M \in {\mathcal M}(d,\bbC).
\end{equation}
Throughout the paper, $I \in {\mathcal M}(d,\bbR)$ is the identity matrix, and
\begin{equation}\label{e:E_in_M(d,R)}
E \in {\mathcal M}(d,\bbR)
\end{equation}
denotes a matrix exponent whose eigenvalues have real parts satisfying
\begin{equation}\label{e:0<a1<...<ap}
\varpi_{E} > 0.
\end{equation}
Throughout the paper, we use the notation
\begin{equation}\label{e:q=tr(E)}
q = \textnormal{tr}(E) > 0.
\end{equation}
Also let $\Gamma = \mathbb{R}^d \backslash \{ 0 \}$. For $c > 0$, we define matrix exponentiation in the usual way as
\begin{equation}\label{e:c^E}
c^{E} := \exp\{(\log c) \hspace{0.5mm}E \} = \sum^{\infty}_{k=0}\frac{((\log c) \hspace{0.5mm} E)^k}{k!}.
\end{equation}
As an exponent, the matrix $E$ induces a norm on $\mathbb{R}^d$. In other words, there exists a norm $\| \cdot \|_0$, associated with the unit sphere
\begin{equation}\label{e:S0}
S_0 = \{ {\bm x} \in \mathbb{R}^d \colon \| {\bm x} \|_0 = 1 \},
\end{equation}
in such a way that the mapping
$$
\Psi \colon (0, \infty) \times S_0 \rightarrow \Gamma, \quad \Psi(r, {\bm \theta}) = r^E {\bm \theta},
$$
is a homeomorphism (see Lemma 6.1.5 in Meerschaert and Scheffler \cite{meerschaert:scheffler:2001}). Therefore, any ${\bm x} \in \Gamma$ can be uniquely decomposed in anisotropic \textit{polar coordinates} as
\begin{equation}\label{e:x=polar_coord}
{\bm x} = \tau( {\bm x} )^E l( {\bm x} )
\end{equation}
for some radial (scalar) component $\tau( {\bm x} ) > 0$ and some directional (vector-valued) component $l( {\bm x} ) \in S_0$. One such norm can be explicitly calculated by means of the formula
\begin{equation}\label{e:|.|0_explicit}
\|{\bm x} \|_0 = \int^{1}_0 \|t^{E}{\bm x}\|_{\bullet}\frac{dt}{t},
\end{equation}
where $\|\cdot\|_{\bullet}$ is any norm in $\bbR^d$.

The matrix $E$ in \eqref{e:E_in_M(d,R)} and \eqref{e:c^E} determines matrix exponentiation in the domain $\bbR^d$; when constructing classes of vector-valued random fields, we also need to consider matrix exponentiation in the range $\bbR^n$. In this case, this is based on the so-called Hurst matrices $H \in {\mathcal M}(n,\bbR)$. Such matrices can be defined in Jordan form as
\begin{equation}\label{e:H=PJHP^(-1)}
H = P J_H P^{-1}, \quad P \in GL(n,\bbC),
\end{equation}
where we assume that
\begin{equation}\label{e:0<minReeig(H)=<maxReeig(H)<a1}
0 < \varpi_H \leq \Upsilon_H. 
\end{equation}

In this paper, tempered operator fractional random fields are developed based on stochastic integration frameworks in the multidimensional time and Fourier domains. This involves two main ingredients: appropriate \textit{integrands }and \textit{random measures}. For the reader's convenience, we provide a self-contained discussion; more details can be found in Jurek and Mason \cite{jurek:mason:1993}, Maejima and Mason \cite{maejima:mason:1994}, Samorodnitsky and Taqqu \cite{samorodnitsky:taqqu:1994}, Bierm\'{e} et al.\ \cite{bierme:meerschaert:scheffler:2007}, Kremer and Scheffler \cite{kremer:scheffler:2017,kremer:scheffler:2019}.

In regard to the first ingredient, the so-named $E$-homogeneous functions play the role of integrands or, equivalently, matrix-valued kernels. Let $E \in {\mathcal M}(d,\bbR)$ be a matrix whose eigenvalues satisfy \eqref{e:0<a1<...<ap}. We say a function $\varphi\colon \mathbb{R}^d \rightarrow \bbC$ is \textit{$E$-homogeneous} if, for ${\bm x} \neq {\mathbf 0}$ and $c > 0$,
\begin{equation}\label{e:E-homogeneous}
\varphi(c^E {\bm x}) = c \varphi({\bm x}).
\end{equation}
Throughout the paper, we only consider $E$-homogeneous functions $\varphi$ such that
\begin{equation}\label{e:E-homogeneous>0}
\varphi({\bm x}) > 0, \quad {\bm x} \neq {\mathbf 0},
\end{equation}
so condition \eqref{e:E-homogeneous>0} is omitted in statements. Given the sphere $S_0$ induced by $E$ (see \eqref{e:S0}), if in addition $\varphi$ is continuous, we write
\begin{equation}\label{e:M-phi>=m-phi}
M_{\varphi} \coloneqq \max_{{\bm \theta} \in S_0} \varphi({\bm \theta}) > 0 \quad \textnormal{and} \quad m_{\varphi} \coloneqq \min_{{\bm \theta} \in S_0} \varphi({\bm \theta}) > 0.
\end{equation}

\begin{rem} The exponent of an $E$-homogeneous function $\varphi$ is generally not unique. Consider the set of symmetries ${\mathcal S}(\varphi)$ of $\varphi$, i.e., those matrices $A \in {\mathcal M}(d,\bbR)$ such that $\varphi(A{\bm x}) = \varphi({\bm x})$, ${\bm x}\in \bbR^d$. Under mild technical assumptions, the set of possible exponents of $\varphi$ can be written as $E + T{\mathcal S}(\varphi)$, where $T {\mathcal S}(\varphi)$ is the tangent space of the Lie group ${\mathcal S}(\varphi)$ at the identity (for more details, see Meerschaert and Scheffler \cite{meerschaert:scheffler:2001}, Theorem 5.2.13).
\end{rem}

In regard to the second ingredient (random measures), we begin by describing operator-stable laws. Recall that a L\'{e}vy measure is defined as a Borel measure $\nu(d{\bm x})$  such that $\nu(\{0\})=0$ and $\int_{\bbR^n} \min\{1,\|{\bm x}\|^2\}\nu(d{\bm x}) < \infty$. It can be shown that, for a function $\widehat{\psi}: \bbR^n \rightarrow \bbC$,
\begin{equation}\label{e:mu-hat=exp(psi(u))}
\widehat{{\mu}} = \exp (\widehat{\psi})
\end{equation}
is the characteristic function of an infinitely-divisible distribution $\mu$ on $\bbR^n$ if and only if we can express
\begin{equation}\label{e:psi}
\widehat{\psi}({\mathbf u}) = \imag \langle \gamma,{\mathbf u}\rangle + \langle {\mathbf u}, Q {\mathbf u} \rangle + \int_{\bbR^n} \Big( e^{\imag \langle {\mathbf u},{\bm x} \rangle} - 1 - \frac{\imag \langle {\mathbf u},{\bm x}\rangle }{1 + \|{\bm x}\|^2} \Big) \nu(d{\bm x}), \quad {\mathbf u} \in \bbR^n,
\end{equation}
for some L\'{e}vy measure $\nu(d{\bm x})$, some shift component $\gamma$ and some $Q \in {\mathcal S}_{\geq 0}(n,\bbR)$. In this case, we write $\mu \sim [\gamma, Q, \nu]$, where the triplet $\gamma$, $Q$ and $\nu$ is uniquely determined by $\mu$. In particular, $\widehat{\psi}$ as in \eqref{e:psi} is the only continuous function satisfying $\widehat{\psi}(0) = 0$ and \eqref{e:mu-hat=exp(psi(u))}. The function $\widehat{\psi}$ is called the \textit{log-characteristic function} of $\mu$. Let $X, \{X_j\}_{j \in \bbN}$ be i.i.d.\ random vectors in $\bbR^n$. We say $X$ is \textit{operator-stable} if there are $A_j \in GL(n,\bbR)$ and shift vectors $a_j \in \bbR^n$ such that
$$
A_j \sum^{j}_{k=1}X_k + a_j \stackrel{d}= X, \quad n \in \bbN
$$
(Sharpe \cite{sharpe:1969}, Hudson and Mason \cite{hudson:mason:1981}). The matrix scaling factor often takes the form $A_j = j^{-B}$ for some $B \in {\mathcal M}(n,\bbR)$, in which case we refer to $B$ as the \textit{exponent} of the distribution ${\mu}(d{\bm x})$ of $X$. If $a_j = 0$, $j \in \bbN$, then $X$ is called \textit{strictly operator-stable}. A distribution on $\bbR^n$ is called \textit{full }when its support is not contained in any hyperplane. It can be shown that a full probability measure ${\mu}(d{\bm x})$ on $\bbR^n$ is operator-stable if and only if it is infinitely-divisible and, for some $B \in {\mathcal M}(n,\bbR)$ and some family $\{a_s\}_{s > 0} \subseteq \bbR^n$, its characteristic function $\widehat{{\mu}}$ satisfies
\begin{equation}\label{e:operator-stable_chf}
\widehat{{\mu}}({\mathbf u})^{s} = \widehat{{\mu}}(s^{B}{\mathbf u}) \exp\{ \imag \langle a_s,{\mathbf u}\rangle \}, \quad s > 0, \quad {\mathbf u} \in \bbR^n
\end{equation}
(Meerschaert and Scheffler \cite{meerschaert:scheffler:2001}, Theorem 7.2.1). Moreover, $\varpi_B \geq 1/2$, and $\varpi_B > 1/2$ if and only if the distribution ${\mu}(d{\bm x})$ of $X$ has no Gaussian component. In addition, ${\mu}(d{\bm x})$ is strictly operator-stable if and only if we can write $a_s = 0$, $s >0$. For an operator-stable distribution that is not necessarily full or strict, it can be shown that
\begin{equation}\label{e:s.psi-hat(u)=psi-hat(s^(B^*)u)}
s \cdot \widehat{\psi}({\mathbf u}) = \widehat{\psi}(s^{B^*}{\mathbf u}) + \imag \langle a_s,{\mathbf u}\rangle, \quad s > 0, \quad {\mathbf u} \in \bbR^n,
\end{equation}
where $^*$ denotes the Hermitian adjoint (see Kremer and Scheffler \cite{kremer:scheffler:2019}, Corollary 2.2).

\begin{example}
An elementary example of a strictly operator-stable distribution is given by a random vector $X = (X_1,\hdots,X_n)^t$ whose entries $X_{1},\hdots,X_{n}$ are independent symmetric $\alpha$-stable ($ S\alpha S$) random variables, each with parameter $\alpha_i \in (0,2]$, $i = 1,\hdots,n$, respectively. In this case, the characteristic function of $X$ has the form $\bbE \exp \imag \langle {\mathbf u}, X\rangle = \exp\{- \sum^{n}_{i=1}|u_i|^{\alpha_i}\}$, ${\mathbf u} = (u_1,\hdots,u_n)^t \in \bbR^n$. Moreover, the scaling relation \eqref{e:operator-stable_chf} is satisfied for $B = \textnormal{diag}(1/\alpha_1,\hdots,1/\alpha_n)$ and $a_{s} \equiv 0$.
\end{example}

We are now in a position to look into random measures. Consider the measure space
\begin{equation}\label{e:(Rd,BRd,Lebd)}
(\bbR^d, {\mathcal B}(\bbR^d),\textnormal{Leb}_d(d{\bm y}))
\end{equation}
and also the so-called $\delta$-ring
\begin{equation}\label{e:delta-ring}
{\mathcal S} := \{A \in {\mathcal B}(\bbR^d): \textnormal{Leb}_d(A) < \infty\}.
\end{equation}
An \textit{independently scattered random measure} (ISRM) is a mapping $M$ from sets in ${\mathcal S}$ to $\bbR^n$-valued random vectors satisfying two properties: $(i)$ it assumes independent values (i.e., random vectors) over disjoint sets in ${\mathcal S}$; $(ii)$ it is $\sigma$-additive a.s. Let $\mu$ be a full, operator-stable distribution in ${\mathcal B}(\bbR^n)$ with triplet $[\gamma,Q,\nu]$, exponent $B$ and log-characteristic function $\widehat{\psi}$. Consider an ISRM
\begin{equation}\label{e:M(dx)}
M(d{\bm x})
\end{equation}
generated by $\mu$ and $(\bbR^d, {\mathcal B}(\bbR^d),\textnormal{Leb}_d(d{\bm y}))$. In particular, this means that for $A \in {\mathcal S} \subseteq {\mathcal B}(\bbR^d)$, $M(A) \sim [\textnormal{Leb}_d(A)\hspace{0.5mm}\gamma,\textnormal{Leb}_d(A)\hspace{0.5mm}Q,\textnormal{Leb}_d(A)\hspace{0.5mm}\nu]$ (for details on the construction and existence of infinitely-divisible, $\bbR^n$-valued ISRMs, see Kremer and Scheffler~\cite{kremer:scheffler:2017}). Given a measurable mapping $f : \bbR^d \rightarrow {\mathcal M}(n,\bbR)$, a characterization of the integrability of $f$ with respect to $M(d{\bm x})$ is provided in Kremer and Scheffler \cite{kremer:scheffler:2019}, Theorem 2.3. However, assuming $\mu$ is full and strictly operator-stable with exponent $B$ (and $\varpi_B$, $\Upsilon_B$ as in \eqref{e:varpi_Upsilon}), a sufficient condition for the almost sure existence of the stochastic integral
\begin{equation}\label{e:int_f(x)M(dx)}
\int_{\bbR^d} f({\bm x}) M(d{\bm x})
\end{equation}
is that there are $0 < \delta_1 \leq \Upsilon^{-1}_B$, $\delta_2 > 0$ and $R > 0$ such that
\begin{equation}\label{e:suff_cond_existence_I(f)}
\int_{\{{\bm x}: \|f({\bm x})\|\leq R\}} \|f({\bm x})\|^{\frac{1}{\Upsilon_B}-\delta_1} d{\bm x}
+\int_{\{{\bm x}: \|f({\bm x})\|> R\}} \|f({\bm x})\|^{\frac{1}{\varpi_B}+\delta_2} d{\bm x} < \infty.
\end{equation}
In this case, the characteristic function of the stochastic integral \eqref{e:int_f(x)M(dx)} is given by
\begin{equation}\label{e:chf_stoch_integral}
\bbE \exp \Big\{ \imag \hspace{0.5mm}{\mathbf u}^t \int_{\bbR^d} f({\bm x}) M(d{\bm x}) \Big\} = \exp \Big\{ \int_{\bbR^d} \widehat{\psi}\Big(  f({\bm x})^* {\mathbf u} \Big) \hspace{1mm}d {\bm x}  \Big\}, \quad {\mathbf u} \in \bbR^n
\end{equation}
(Kremer and Scheffler \cite{kremer:scheffler:2017}, Theorem 5.4). Due to algebraic convenience, condition \eqref{e:suff_cond_existence_I(f)} is used in this paper in establishing the existence of tempered operator fractional random fields in the time domain.

\begin{example}\label{ex:SalphaS}
The simplest example of a full, strictly operator-stable random measure is given by a vector of i.i.d.\ $S\alpha S$ random measures, where $0 < \alpha \leq 2$. In fact, consider the measure space \eqref{e:(Rd,BRd,Lebd)} and $\delta$-ring \eqref{e:delta-ring}. A \textit{$\bbR$-valued $S \alpha S$ random measure on $(\bbR^d,{\mathcal B}(\bbR^d))$ with control measure $\textnormal{Leb}_d(d{\bm x})$} is an ISRM such that, for all $A \in {\mathcal S}$, $M(A)$ is a $S \alpha S$ random variable with parameter $\textnormal{Leb}^{1/\alpha}_d(A)$. The finite-dimensional distributions of a $S \alpha S$ random measure $M(d{\bm x})$ are uniquely determined by its control measure $\textnormal{Leb}_d(d{\bm x})$. Moreover, let $f$ be a $\bbR$-valued function such that $\int_{\bbR^d}|f({\bm x})|^{\alpha} d{\bm x} < \infty$. Then, the $\bbR$-valued integral $I(f) = \int_{\bbR^d} f({\bm x}) M(d{\bm x})$ is well defined and its characteristic function is given by
\begin{equation}\label{e:I(f)_chf}
\bbE \exp\{\imag \hspace{0.5mm}u  I(f) \} = \exp\Big\{ - |u|^{\alpha} \int_{\bbR^d}
|f({\bm x}) |^{\alpha} d{\bm x}\Big\}, \quad u \in \bbR.
\end{equation}
More details on $\bbR$- or $\bbC$-valued $S \alpha S$ random measures can be found in Samorodnitsky and Taqqu \cite{samorodnitsky:taqqu:1994}, Sections 3.1--3.3 and 6.1--6.3.
\end{example}

When considering integrands with $\bbC$-valued entries, it is natural to replace $\mu$ with a full, strictly operator-stable distribution $\widetilde{\mu} \sim [\widetilde{\gamma}, \widetilde{Q}, \widetilde{\nu}]$ on $\bbR^{2n}$ with log-characteristic function $\widetilde{\psi}$˜. In this case, we assume that an exponent of $\widetilde{\mu}$ is given by $\widetilde{B} = B \oplus B$ for some $B \in {\mathcal M}(n,\bbR)$. Hence, the ISRM $M(d{\bm x})$ generated by $\widetilde{\mu}$ and $\textnormal{Leb}_d(d{\bm x})$ can be identified with a $\bbC^n$-valued ISRM $\widetilde{M}(d{\bm x})$. Define the mapping $\bbC^n \ni {\bm z} \mapsto \Xi({\bm z}) = (\Re {\bm z}, \Im {\bm z})$, which breaks up a vector in $\bbC^n$ into its real components. Then, we can write $\Xi(\widetilde{M}) = M$, where $M$ is called the $\bbR^{2n}$-valued ISRM associated with $\widetilde{M}$. We say that a measurable function $f: \bbR^d \rightarrow {\mathcal M}(n,\bbC)$ is \textit{partially integrable} with respect to $\widetilde{M}$ if
\begin{equation}\label{e:Re_int_f(x)M-tilde(dx)}
\Re \int_{\bbR^d}f({\bm x})\widetilde{M}(d{\bm x})
\end{equation}
exists a.s. One such function $f$ may be partially integrable even if $\int_{\bbR^d}f({\bm x})\widetilde{M}(d{\bm x})$ does not exist. If $f$ is, indeed, partially integrable, then the characteristic function of the stochastic integral \eqref{e:Re_int_f(x)M-tilde(dx)}  is given by
\begin{equation}\label{e:chf_stoch_integral_partially_integ}
\bbE \exp \Big\{ \imag \hspace{0.5mm}{\mathbf u}^t \hspace{1mm}\Re \int_{\bbR^d} f({\bm x}) \widetilde{M}(d{\bm x}) \Big\} = \exp \Big\{ \int_{\bbR^d} \widehat{\psi}_{\Xi(M)} \Big( \Xi(f({\bm x})^* {\mathbf u}) \Big)\hspace{1mm} d {\bm x}   \Big\}, \quad {\mathbf u} \in \bbR^n,
\end{equation}
where $\widehat{\psi}_{\Xi(M)}$ is the log-characteristic function associated with $\Xi(M)$ (Kremer and Scheffler \cite{kremer:scheffler:2017}, Remark 5.12).

For $n = 1$, a distribution $\mu$ on ${\mathcal B}(\bbR^d)$ is called \textit{rotationally invariant} (isotropic) if $(O\mu) = \mu$ for every rotation matrix $O \in SO(n)$ (see Samorodnitsky and Taqqu \cite{samorodnitsky:taqqu:1994}, Definition 2.6.2). For $S \alpha S$-based harmonizable representations, one typically uses a $\bbC$-valued rotationally invariant $\alpha$-stable random measure (cf.\ Bierm\'{e} at al.\ \cite{bierme:meerschaert:scheffler:2007}). For appropriately defined integration kernels, it turns out that this conveniently yields stationary increments. However, in the operator-stable framework, a similar but generally weaker condition than rotational invariance is sufficient for stationary increments. Such condition is given by
\begin{equation}\label{e:Amu(dx)=mu(dx)}
(A\widetilde{\mu})(d{\bm x}) = \widetilde{\mu}(d{\bm x}), \quad  A \in {\mathcal T}(2n),
\end{equation}
where we define the (Abelian) group of rotation matrices
\begin{equation}\label{e:T(2n)}
{\mathcal T}(2n) = \Bigg\{
\begin{pmatrix}
(\cos \omega) I_n & (\sin \omega) I_n \\
-(\sin \omega) I_n & (\cos \omega) I_n
\end{pmatrix}, \omega \in [0,2\pi)
\Bigg\}
\end{equation}
(cf.\ Kremer and Scheffler \cite{kremer:scheffler:2019}, p.\ 20).

\section{Moving average TRF}\label{sec:time_domain}

As anticipated in the Introduction, in this section we introduce two different classes of moving average-type TRFs and establish their fundamental properties. The difference between these two classes lies in the tempering function. In one case, tempering is done by means of a traditional, scalar-valued exponential function that affects all entries of the fractional kernel; in the other case, by means of a matrix-valued function that slows down the eigenvalue-based scaling laws over large (time) scales. As it turns out, the essential distributional properties of the random fields by the two methods are qualitatively identical.

We begin with the more familiar (entry-wise) exponential tempering.
\begin{defn}\label{defn:MA-TRF}
Let $\mu$ be a full, strictly operator-stable measure on ${\mathcal B}(\bbR^n)$ with log-characteristic function $\widehat{\psi}$ and exponent $B$. Let $M (d{\bm y})$ be a $\bbR^n$-valued ISRM generated by $\mu$ and $(\bbR^d, {\mathcal B}(\bbR^d),\textnormal{Leb}_d(d{\bm y}))$. Let $H$ be a (Hurst) matrix as in \eqref{e:H=PJHP^(-1)} and \eqref{e:0<minReeig(H)=<maxReeig(H)<a1}. Fix $\lambda > 0$, and suppose
\begin{equation}\label{e:lambda(H-qB)+qlambda(B)>0}
\varpi_{H-qB} + q \varpi_B > 0.
\end{equation}
Let $\varphi \colon \mathbb{R}^d \rightarrow [0, \infty)$ be a continuous $E$-homogeneous function. A vector-valued \textit{moving average TRF} (MA-TRF) is the random field whose stochastic integral representation is given by
\begin{eqnarray}\label{e:MA-TOSSRF}
X_{\lambda} ({\bm x}) = \int_{\mathbb{R}^d}
\left[ e^{-\lambda \varphi({\bm x} - {\bm y})} \varphi({\bm x} - {\bm y})^{H - qB}
- e^{-\lambda \varphi(-{\bm y})} \varphi(-{\bm y})^{H - q B} \right] M (d{\bm y}), \quad {\bm x}  \in \bbR^d.
\end{eqnarray}
\end{defn}

Note that, by formally setting $\lambda=0$ in \eqref{e:MA-TOSSRF}, $X_{0} ({\bm x})$ corresponds to \eqref{moving_RF_aniso}. Also, for $n =1=d$ and $B = 1/2$ (Gaussian), \eqref{e:MA-TOSSRF} is a TFBM (Meerschaert and Sabzikar \cite{meerschaert:sabzikar:2013}).

In the following theorem and corollary, we establish the existence and fundamental properties of MA-TRF.

\begin{thm}\label{thm:exponential_moving_TOSSF}
The random field \eqref{e:MA-TOSSRF} exists for every ${\bm x} \in \bbR^d$.
\end{thm}

\begin{cor}\label{cor:stationary_increments_scaling_property_MA}
Let $X_{\lambda}:= \{ X_{\lambda} ({\bm x}) \}_{{\bm x} \in \mathbb{R}^d}$ be a MA-TRF. Then,
\begin{itemize}
\item [(a)] under the commutativity condition
\begin{equation}\label{e:HB=BH}
HB = BH,
\end{equation}
$X_{\lambda}$ is strictly operator-stable with exponent $B$ and satisfies the scaling property
\begin{eqnarray}
\label{e:X-lambda(c^Ex)=c^H_X-clambda(x)}
\{ X_{\lambda} (c^E {\bm x}) \}_{{\bm x} \in \mathbb{R}^d} \stackrel{f.d.}{=}
\{ c^H X_{c\lambda} ({\bm x}) \}_{{\bm x} \in \mathbb{R}^d}, \quad c > 0;
\end{eqnarray}
\item[(b)] $X_{\lambda}$ is stochastically continuous;
\item[(c)] $X_{\lambda}$ has stationary increments;
\item[(d)] $X_{\lambda}({\bm x})$ is full at any ${\bm x} \in \bbR^d \backslash\{0\}$.
\end{itemize}
\end{cor}
%
When constructing harmonizable representations as in Section \ref{s:Gaussian}, it can be very convenient to work with tempering tools that depend on Hurst eigenvalues. To construct matrix-based tempering, we start from the modified Bessel function of the second kind $K_{\nu}(x)$ and apply the so-named technique of primary matrix functions. This involves greater generality than what is strictly needed in Section \ref{s:Gaussian}, but is also of independent theoretical interest.

Recall that the modified Bessel function of the second kind can be represented as
\begin{equation}\label{e:mod_Bessel_second_kind}
\bbC \ni K_{\nu}(u) = \int^{\infty}_0 e^{-u \cosh(t)} \cosh(\nu t) dt, \quad u > 0, \quad \nu \in \bbC,
\end{equation}
where $\cosh(z) = (e^{-z}+e^{z})/2$, $z \in \bbC$ (see Temme \cite{temme:2011}, Section 9.6). The function $K_\nu$ is continuous and, for any $\nu \in \bbR$, it satisfies
\begin{equation}\label{e:mod_Bessel_second_kind_asymptotics_large_u}
K_{\nu}(u) \sim \sqrt{\frac{\pi}{2 u}} e^{-u}, \quad u \rightarrow \infty,
\end{equation}
\begin{equation}\label{e:mod_Bessel_second_kind_asymptotics}
K_{\nu}(u) \sim
\left\{\begin{array}{cc}
2^{|\nu|-1}\Gamma(|\nu|) u^{-|\nu|}, & \nu \neq 0 \\
- \log u , & \nu = 0 \\
\end{array}\right., \quad u \rightarrow 0^+
\end{equation}
(see Abramowitz and Stegun \cite{abramowitz:stegun:1970}, formulas 9.6.8, 9.6.9 and 9.7.2; cf.\ the proof of Theorem \ref{thm:Bessel_moving_TOSSF}). Hence, $K_\nu$ can be naturally seen as a tempering device assuming the singularity around the origin can be controlled.

Formally, a matrix-valued tempering function is obtained by replacing the parameter $\nu$ with the matrix $H - qB$ in expression \eqref{e:mod_Bessel_second_kind}. To make this procedure rigorous, for the reader's convenience we recall the definition of primary matrix functions (more details and properties can be found in Horn and Johnson \cite{horn:johnson:1991}, Sections 6.1 and 6.2). Consider
\begin{equation}\label{e:Lambda=PJP^(-1)}
\Lambda = PJP^{-1} \in M(n,\bbC),
\end{equation}
where $J$ is in Jordan form with Jordan blocks $J_{\vartheta_1},\ldots, J_{\vartheta_N}$ along the diagonal. Let
\begin{equation}\label{e:q(t)}
q_{\Lambda}(z) = (z - \vartheta_1)^{r_1}\ldots(z - \vartheta_N)^{r_N}
\end{equation}
be the minimal polynomial of $\Lambda$, where $\vartheta_1,\ldots,\vartheta_N$ are pairwise distinct, and $r_k \geq 1$ for $k=1,\ldots,N$, $N \leq n$.  Now, let $U  \subseteq \bbC$ be an open set. Given a function $h:U \rightarrow \bbC$ and some $\Lambda \in {\mathcal M}(n,\bbC)$ as in \eqref{e:Lambda=PJP^(-1)}, consider the conditions: (M1) $\vartheta_k \in U$, $k=1,\ldots,N$; (M2) if
$r_k > 1$, then $h(z)$ is analytic in a vicinity $U_k \ni \vartheta_k$, where $U_k \subseteq U$. Let
$$
{\mathcal M}_h =\{ \Lambda \in {\mathcal M}(n,\bbC); \textnormal{ conditions (M1) and (M2) hold at the characteristic roots $\vartheta_1,\ldots,\vartheta_N$ of $\Lambda$} \}.
$$
We can now define the primary matrix function $h(\Lambda)$ associated
with the scalar-valued stem function $h(z)$.
\begin{defn}\label{d:matrix_function}
The primary matrix function $h:{\mathcal M}_h \rightarrow
{\mathcal M}(n,\bbC)$ is defined as
$$
h(\Lambda) = Ph(J)P^{-1}=P\left(\begin{array}{ccc}
h(J_{\vartheta_1})
& \hdots & 0\\
\vdots & \ddots & \vdots \\
 0 & \hdots & h(J_{\vartheta_N})
\end{array}\right)P^{-1},
$$
where
$$
h(J_{\vartheta_k})= \left(\begin{array}{cccc} h(\vartheta_k) & 0 & \ldots & 0 \\
h'(\vartheta_k) & h(\vartheta_k) &  \ddots & 0\\
 \vdots & \ddots &  \ddots & \vdots\\
\frac{h^{(r_k-1)}(\vartheta_k)}{(r_k -1)!} & \ldots  & h'(\vartheta_k) & h(\vartheta_k) \\
\end{array}\right), \quad k = 1,\hdots,N.
$$
\end{defn}

Starting from the framework provided by Definition \ref{d:matrix_function}, in Proposition \ref{p:Bessel} we extend the univariate formula \eqref{e:mod_Bessel_second_kind} by establishing the integral representation
\begin{equation}\label{e:Bessel}
{\mathcal M}(n,\bbC) \ni K_{N}(u) = \int^{\infty}_0 e^{-u \cosh(t)} \cosh(N t) dt, \quad u > 0,
\end{equation}
for the primary matrix function $K_{N}(u)$, where $N \in {\mathcal M}(n,\bbC)$ and $\cosh(N t) = (e^{-Nt}+e^{Nt})/2$, $t > 0$, is also a primary matrix function.

We are now in a position to define the alternative moving average-type random field, based on matrix-based tempering.
\begin{defn}\label{defn:MA-B-TRF}
Let $\mu$ be a full, strictly operator-stable measure on ${\mathcal B}(\bbR^n)$ with log-characteristic function $\widehat{\psi}$ and exponent $B$. Let $M (d{\bm y})$ be a $\bbR^n$-valued ISRM generated by $\mu$ and $(\bbR^n, {\mathcal B}(\bbR^n),\textnormal{Leb}_d(d{\bm y}))$. Let $H$ be a (Hurst) matrix as in \eqref{e:H=PJHP^(-1)} and \eqref{e:0<minReeig(H)=<maxReeig(H)<a1}. Fix $\lambda > 0$, and define the matrix-valued function $\varrho_{H-qB,\lambda}({\bm y}) := K_{H-qB}(\lambda \varphi({\bm y}))$, where $K_{H-qB}(\cdot)$ is as in \eqref{e:Bessel}. Suppose
\begin{equation}\label{e:MA-B_condition_eig_exponent}
 2 \varpi_{H-qB} + q \varpi_B > 0.
\end{equation}
Let $\varphi \colon \mathbb{R}^d \rightarrow [0, \infty)$ be a continuous $E$-homogeneous function. A vector-valued \textit{moving average-Bessel-TRF} (MA-B-TRF) is the random field whose stochastic integral representation is given by
\begin{equation}\label{e:MA-B-TOSSRF}
X^{Bes}_{\lambda} ({\bm x}) = \int_{\mathbb{R}^d}
\Big[ \varrho_{H-qB,\lambda}({\bm x} - {\bm y}) \varphi({\bm x} - {\bm y})^{H - qB}
- \varrho_{H-qB,\lambda}(- {\bm y})  \varphi(-{\bm y})^{H - q B} \Big] M (d{\bm y}), \quad {\bm x}  \in \bbR^d.
\end{equation}
\end{defn}

In the following theorem and corollary, we establish the existence and fundamental properties of MA-B-TRF. Note that the latter qualitatively match those of MA-TRF.
\begin{thm}\label{thm:Bessel_moving_TOSSF}
The random field \eqref{e:MA-B-TOSSRF} exists for every ${\bm x} \in \bbR^d$.
\end{thm}

\begin{cor}
\label{cor:stationary_increments_scaling_property_MA-B}
Let $X^{Bes}_{\lambda}:= \{ X^{Bes}_{\lambda} ({\bm x}) \}_{{\bm x} \in \mathbb{R}^d}$ be a MA-B-TRF. Then,
\begin{itemize}
\item [(a)] under the commutativity condition \eqref{e:HB=BH}, $X^{Bes}_{\lambda}$ is strictly operator-stable with exponent $B$ and satisfies the scaling property
\begin{eqnarray}
\label{e:Bes_scaling}
\{ X^{Bes}_{\lambda} (c^E {\bm x}) \}_{{\bm x} \in \mathbb{R}^d} \stackrel{f.d.}{=}
\{ c^H X^{Bes}_{c\lambda} ({\bm x}) \}_{{\bm x} \in \mathbb{R}^d}, \quad c > 0;
\end{eqnarray}
\item[(b)] $X^{Bes}_{\lambda}$ is stochastically continuous;
\item[(c)] $X^{Bes}_{\lambda}$ has stationary increments;
\item[(d)] for $n =1$, $X^{Bes}_{\lambda}({\bm x})$ is full at any ${\bm x} \in \bbR^d \backslash\{0\}$;
\item[(e)] for $E = I$ and assuming $\textnormal{eig}(H-qB) \subseteq \bbR^n$, $X^{Bes}_{\lambda}({\bm x})$ is full at any ${\bm x} \in \bbR^d \backslash\{0\}$.
\end{itemize}
\end{cor}

\begin{rem}
In the general anisotropic case of $E$ satisfying relation \eqref{e:0<a1<...<ap}, it remains an open problem whether or not $X^{Bes}_{\lambda}({\bm x})$ is full at any ${\bm x} \in \bbR^d \backslash\{0\}$.
\end{rem}


\section{Harmonizable TRF}\label{sec:frequency_domain}

As discussed in the Introduction, in this section we turn to the Fourier domain. We construct harmonizable-type TRF and investigate its essential properties. Since the integration kernels involved are naturally entry-wise $\bbC$-valued, we resort to $\bbC^n$-valued random measures.
\begin{defn}\label{defn:H-TRF}
Let $\widetilde{\mu}$ be a full, strictly operator-stable measure on ${\mathcal B}(\bbR^{2n})$ with log-characteristic function $\widetilde{\psi}$ and exponent $\widetilde{B}:=B \oplus B$. Let $\widetilde{M}(d{\bm y})$ be a $\bbC^n$-valued ISRM identified with a $\bbR^{2n}$-valued ISRM $M$ generated by $\widetilde{\mu}$ and $(\bbR^d, {\mathcal B}(\bbR^d), \textnormal{Leb}_d(\bbR^d))$. Fix $\lambda > 0$, and let $H$ be a (Hurst) matrix as in \eqref{e:H=PJHP^(-1)} and \eqref{e:0<minReeig(H)=<maxReeig(H)<a1}. Let $\varphi: \bbR^d \rightarrow [0,\infty)$ be a continuous $E^*$-homogeneous function. We define a vector-valued \textit{harmonizable-TRF} (H-TRF) as the random field whose stochastic integral representation is given by
\begin{eqnarray}\label{e:H-TRF}
\widetilde{X}_{\lambda} ({\bm x}) = \Re
\int_{\mathbb{R}^d} (e^{- \imag \langle \bm{x},\bm{\xi}\rangle} - 1) ( \lambda + \varphi({\bm \xi}))^{-H} ( \lambda + \varphi({\bm \xi}))^{-qB} \widetilde{M}(d{\bm \xi}), \quad {\bm x} \in \bbR^d.
\end{eqnarray}
 \end{defn}

Note that, by formally setting $\lambda=0$ in \eqref{e:H-TRF}, $\widetilde{X}_{0} ({\bm x})$ corresponds to \eqref{harmo_RF_aniso}. Also, for $n =1=d$ and $B = 1/2$ (Gaussian), \eqref{e:H-TRF} is a TFBM (Meerschaert and Sabzikar \cite{meerschaert:sabzikar:2013}).

In the following theorem and corollary, we establish the existence and fundamental properties of H-TRF. As discussed in Section \ref{s:preliminaries}, by comparison to moving average-type TRFs, showing the stationarity of increments requires an additional assumption.
\begin{thm}\label{thm:General_harmo_TOSSRF}
The random field \eqref{e:H-TRF} exists for every ${\bm x} \in \bbR^d$.
\end{thm}

\begin{cor}\label{c:stationary_increments_scaling_property_H}
Let $\{ \widetilde{X}_{\lambda} ({\bm x}) \}_{{\bm x} \in \mathbb{R}^d}$ be a H-TRF. Then,
\begin{itemize}
\item [(a)] under the commutativity condition \eqref{e:HB=BH},
$\widetilde{X}_{\lambda}$ is strictly operator-stable with exponent $B$ and satisfies the scaling property
\begin{eqnarray}
\label{e:scaling_H-TRF}
\{ \widetilde{X}_{\lambda} (c^E {\bm x}) \}_{{\bm x} \in \mathbb{R}^d} \stackrel{f.d.}{=}
\{ c^H \widetilde{X}_{c\lambda}({\bm x}) \}_{{\bm x} \in \mathbb{R}^d}, \quad c > 0;
\end{eqnarray}
\item[(b)] $\widetilde{X}_{\lambda}$ is stochastically continuous;
\item[(c)] if $\widetilde{\mu}(d{\bm x})$ as in Definition \ref{defn:H-TRF} satisfies the symmetry condition \eqref{e:Amu(dx)=mu(dx)}, then $\widetilde{X}_{\lambda}$ has stationary increments;
\item[(d)] $\widetilde{X}_{\lambda}({\bm x})$ is full at any ${\bm x} \in \bbR^d \backslash\{0\}$.
\end{itemize}
\end{cor}

\begin{rem}\label{rem:TFSM_TFSMII}
It is illustrative to revisit, in the framework of TRF, the univariate stochastic processes tempered fractional stable motion of the first and second kinds (TFSM\hspace{0.5mm} and TFSM\hspace{0.5mm}II, respectively; Meerschaert and Sabzikar \cite{meerschaert:sabzikar:2016}, Sabzikar and Surgailis \cite{sabzikar:surgailis:2018}). While dropping condition \eqref{e:E-homogeneous>0}, moving average- and harmonizable-type TFSMs correspond to the instances generated by taking $\varphi(x) = (-x)_{+} = \max\{x,0\}$ and $\varphi(x)= (\lambda + \imag \xi)^{-1}$ in \eqref{e:MA-TOSSRF} and \eqref{e:MA-B-TOSSRF}, respectively. On the other hand, TFSM\hspace{0.5mm}II finds no corresponding random field among moving average-type TRF defined in this paper. In fact, the moving average subclass of TFSM\hspace{0.5mm}II is given by $X^{I\!I}_{H,\alpha,\lambda}(t):=\int_\rr h_{H,\alpha,\lambda}(t;y) Z_{\alpha}(dy)$, $t \in \bbR$, where
\begin{equation*}\label{TFSM2}
h_{H,\alpha,\lambda}(t;y) := (t-{\bm y})_+^{H - \frac{1}{\alpha}} e^{-\lambda (t-{\bm y})_+} - (-{\bm y})_+^{H - \frac{1}{\alpha}} e^{-\lambda (-{\bm y})_+} \\
+ \lambda \int_{0}^{t} (s-{\bm y})_{+}^{H-\frac{1}{\alpha}}e^{-\lambda(s-{\bm y})_{+}}\ ds.
\end{equation*}
The same holds for its harmonizable subclass, defined by $\widetilde{X}^{I\!I}_{H,\alpha,\lambda}(t)=
\frac{1}{\sqrt{2\pi}}\int_{\rr}\frac{ e^{ i\xi t}-1}{i\xi}(\lambda+ i\xi)^{\frac{1}{\alpha}-H} {W}_{\alpha}(d\xi)$, $t \in \bbR$.

\end{rem}

\section{Representations of isotropic tempered operator fractional Brownian fields}\label{s:Gaussian}

%

It is well known that moving average and harmonizable non-Gaussian fractional random fields are generally non-equivalent. However, assuming Gaussianity, equivalence can be established for certain subclasses of random fields. As discussed in the Introduction, it is intricate to relate moving average and harmonizable representations of Gaussian TRF directly by means of its covariance function, even under isotropy. So, in this section, we establish harmonizable representations of Gaussian and isotropic instances of MA-TRF and MA-B-TRF by computing the Fourier transforms of their kernels.

In the next definition, we recap the notion of isotropy for random fields (cf.\ Didier et al.\ \cite{didier:meerschaert:pipiras:2018:symmetries}, Section 3.4.1).
\begin{defn}
We say a $\bbR^n$-valued random field $X = \{X({\bm x})\}_{{\bm x} \in \bbR^d}$ is \textit{isotropic} when
$$
\{X(O{\bm x})\}_{{\bm x} \in \bbR^d} \stackrel{f.d.}= \{X({\bm x})\}_{{\bm x} \in \bbR^d}, \quad O \in O(d).
$$
\end{defn}
The Gaussian, isotropic instance of MA-TRF is described in the next definition (cf.\ expression \eqref{e:MA-TOSSRF}).
\begin{defn}\label{TFGFsecond_diff}
Fix $\lambda>0$. Let $H$ be a (Hurst) matrix as in \eqref{e:H=PJHP^(-1)} and \eqref{e:0<minReeig(H)=<maxReeig(H)<a1}, and let $\bm{Z}(d\bm{y})$ be a $\bbR^n$-valued Gaussian ISRM with Lebesgue control measure on $\bbR^d$. The random field
\begin{equation} \label{TFGFIdefn}
B_{H, \lambda}(\bm{x})  := \int_{\mathbb{R}^{d}} \Big[
e^{- \lambda \|\bm{x}-\bm{y}\|} \|  \bm{x}-\bm{y} \| ^ {H-\frac{d}{2}I} -
e^{- \lambda \|-\bm{y}\|}\  \|-\bm{y}\| ^ {H-\frac{d}{2}I}
\Big] \bm{Z}(d\bm{y}), \quad {\bm x} \in \bbR^d,
\end{equation}
is called an \textit{isotropic tempered operator fractional Brownian field} (ITOFBF).
\end{defn}
Because $B_{H, \lambda}(\bm{x})$ is an instance of the random field $X_{\lambda}$ with $\varphi(\bm x)= \|\bm x\|$ and $M(d\bm x) = Z(d\bm x)$, all the properties stated in Theorem \ref{thm:exponential_moving_TOSSF} and Corollary \ref{cor:stationary_increments_scaling_property_MA} hold with $B = \frac{1}{2}I$. In addition, ITOFBF has finite second moments, whose structure we provide in the following corollary. In particular, note that the covariance structure of ITOFBF does have the form \eqref{e:cov_function_under_isotropy}.
%
%
\begin{cor}
\label{cor100}
Let $B_{H, \lambda} = \{B_{H, \lambda}({\bm x})\}_{{\bm x} \in \bbR^d}$ be an ITOFBF. Then,
\begin{itemize}
\item [$(i)$] $B_{H, \lambda}$ is, indeed, isotropic;
\item [$(ii)$] $B_{H, \lambda}$ has covariance function
$$
\textnormal{Cov}[ B_{H, \lambda}(\bm{x}), B_{H, \lambda}(\bm{x^{\prime}}) ]
$$
\begin{equation} \label{covariance TFGFI}
= \frac{1}{2} [  \| \bm{x} \|^{H} C^2_{\bm{x}} \| \bm{x} \|^{H^t} + \| \bm{x^{\prime}} \|^{H}C^2_{\bm{x^{\prime}}} \|\bm{x^{\prime}} \|^{H^t}
- \| \bm{x-x^{\prime}} \|^{H}C^2_{\bm{x-x^{\prime}}} \|\bm{x-x^{\prime}} \|^{H^t} ]
\end{equation}
for all $\bm{x}, \bm{x^{\prime}} \in \mathbb{R}^{d}$, where
\begin{equation}\label{e:Cx}
\displaystyle{ {\mathcal S}_{\geq 0}(n,\bbR) \ni C^2_{\bm{x}} = \mathbb{E} \left[ B_{H, \| \bm{x} \| \lambda} (\bm{e_{1}})B_{H, \| \bm{x} \| \lambda}(\bm{e_{1}})^{t} \right] } \ , \quad \bm{e^t_{1}} = (1, 0, \hdots, 0).
\end{equation}
\end{itemize}
\end{cor}

In the next proposition, we obtain the harmonizable representation of ITOFBF. We focus on the case of greatest interest in practice, namely, when the Hurst matrix $H = PJ_HP^{-1}$ has simple and real eigenvalues, with arbitrary coordinates $P \in GL(n,\bbC)$. Throughout the rest of the paper, $_2F_1(a;b;c;z)$ represents a Gaussian hypergeometric function, which is defined as
\begin{equation}\label{Fhyeorgeometric}
\begin{split}
_2F_1(a;b;c;z)&=\sum_{j=0}^{\infty}\frac{\Gamma(a+j)\Gamma(b+j)\Gamma(c)}{\Gamma(a)\Gamma(b)\Gamma(c+j)\Gamma(j+1)}z^j\\
&=1+\frac{a\cdot b}{c\cdot 1}z+\frac{a(a+1)b(b+1)}{c(c+1)\cdot 1\cdot 2}z^{2}+\ldots
\end{split}
\end{equation}
for all $a,b \in \bbC$, all complex $|z|<1$ and real $c$ not a negative integer. In the statement of the proposition,
\begin{equation}\label{e:h-ell,ell=1,..,n}
h_\ell, \quad \ell = 1,\hdots, n,
\end{equation}
denote the possibly repeated, $\bbR$--valued and ordered eigenvalues (characteristic roots) of $H$.
\begin{prop}\label{p:TFGFI harmo}
Let $B_{H, \lambda}= \{B_{H, \lambda} ( \bm{x} ) \}_{\bm{x} \in \bbR^d}$ be an ITOFBF as in \eqref{TFGFIdefn} with parameters $\lambda > 0$ and $H = PJ_HP^{-1}$, $P \in GL(n,\bbC)$, where the eigenvalues of the Hurst matrix are real and simple. Then, $B_{H, \lambda}$ admits the harmonizable representation
\begin{equation}\label{TFGFI harmo}
\{B_{H, \lambda} ( \bm{x} )\}_{{\bm x}\in \bbR^d} \stackrel{f.d.}= \Big\{\frac{1}{\sqrt{2\pi}} \hspace{1mm}C_{H,\lambda}
 \int_{\mathbb{R}^{d}} (e^{- \imag \langle\bm{\xi},\bm{x}\rangle} - 1)
\ {}_2 F_1 \Big( \frac{I \frac{d}{2} + H} {2}, \frac{ I\frac{d}{2} + H + I} {2};
\frac{d}{2}; - \frac{\| \bm{\xi} \|^2} {\lambda^2} \Big)\
W(d\bm{\xi})\Big\}_{{\bm x}\in \bbR^d},
\end{equation}
where $W(d\bm{\xi})$ is a $\bbC^n$-valued Gaussian random measure with Lebesgue control measure such that $W(-d\bm{\xi}) = \overline{W(d\bm{\xi})}$ a.s. In \eqref{TFGFI harmo}, we use \eqref{Fhyeorgeometric} to define the matrix-valued ${}_2 F_1$ function as
\begin{equation}\label{e:2F1_matrix-valued}
{\mathcal M}(n,\bbR) \ni {}_2 F_1 \Big( \frac{ I\frac{d}{2} + H} {2}, \frac{ I\frac{d}{2} + H + I} {2};\frac{d}{2}; - \frac{\| \bm{\xi} \|^2} {\lambda^2} \Big)
\end{equation}
$$
:= P \textnormal{diag}\Big( {}_2 F_1 \Big( \frac{ \frac{d}{2} + h_1} {2}, \frac{ \frac{d}{2} + h_1 + 1} {2}; \frac{d}{2}; - \frac{\| \bm{\xi} \|^2} {\lambda^2} \Big),\hdots,{}_2 F_1 \Big( \frac{ \frac{d}{2} + h_n} {2}, \frac{ \frac{d}{2} + h_n + 1} {2}; \frac{d}{2}; - \frac{\| \bm{\xi} \|^2} {\lambda^2} \Big)  \Big) P^{-1},
$$
and $C_{H,\lambda}$ is a matrix constant given by
$$
C_{H,\lambda}= \frac{(2 \pi)^{ \frac{d} {2} }}{\lambda^{\frac{d}{2}}2^{ \frac{d-2} {2} } \Gamma(\frac{d}{2})}  P \textnormal{diag}\Big( \Gamma\Big(\frac{d}{2} + h_1\Big)
 \lambda^{ - h_1 }, \hdots,
\Gamma\Big(\frac{d}{2} + h_n\Big)
 \lambda^{ -h_n }   \Big)
P^{-1}.
$$
\end{prop}



The following proposition builds upon \eqref{TFGFI harmo} to provide a Fourier-domain alternative covariance formula to \eqref{covariance TFGFI}. In the statement and proof of the proposition, we use the matrix-valued function ${}_2 F_1$ in the sense of \eqref{e:2F1_matrix-valued}, and $J_\cdot(\cdot)$ denotes the modified Bessel function of the first kind (cf.\ \eqref{e:Bessel_shows_up}).
\begin{prop}
\label{prop2}
Let $B_{H, \lambda}= \{B_{H, \lambda} ( \bm{x} ) \}_{\bm{x} \in \bbR^d}$ be an ITOFBF as in \eqref{TFGFIdefn} with parameters $\lambda > 0$ and $H = PJ_HP^{-1}$, $P \in GL(n,\bbC)$, where the eigenvalues of the Hurst matrix are real and simple. Define the matrix-valued function
$$
{\mathcal S}_{\geq 0}(n,\bbC) \ni {}_2 {\mathbf F}_1 \left( \frac{ I\frac{d}{2} + H} {2}, \frac{ I\frac{d}{2} + H + I} {2};
\frac{d}{2}; - \frac{r^2} {\lambda^2} \right)
$$
$$
= {}_2 F_1 \left( \frac{ I\frac{d}{2} + H} {2}, \frac{ I\frac{d}{2} + H + I} {2};
\frac{d}{2}; - \frac{r^2} {\lambda^2} \right) \hspace{1mm}{}_2 F_1 \left( \frac{I \frac{d}{2} + H} {2}, \frac{ I\frac{d}{2} + H + 1} {2};
\frac{d}{2}; - \frac{r^2} {\lambda^2} \right)^*.
$$
Then, we can express the covariance function of $B_{H, \lambda}$ as
\begin{equation}\label{tmp}
\begin{split}
&{\rm Cov}[ B_{H, \lambda} (\bm{x}), B_{H, \lambda} (\bm{x^{\prime}}) ]
= \| \bm{x} - \bm{x^{\prime}} \|^{ 1 - \frac{d} {2} }\left( 2\pi \right)^{ \frac{d} {2} }
\int_{0}^{\infty} r^{ \frac{d}{2} } J_{ \frac{d-2} {2} } (\| \bm{x} - \bm{x^{\prime}} \|r)
\ {}_2 {\mathbf F}_1 \left( \frac{I \frac{d}{2} + H} {2}, \frac{I \frac{d}{2} + H + I} {2};
\frac{d}{2}; - \frac{r^2} {\lambda^2} \right) dr \\
&
- \| \bm{x} \|^{ 1 - \frac{d} {2} }\left( 2\pi \right)^{ \frac{d} {2} }
\int_{0}^{\infty} r^{ \frac{d}{2} } J_{ \frac{d-2} {2} } (\| \bm{x} \|r)
\ {}_2 {\mathbf F}_1 \left( \frac{ I \frac{d}{2} + H} {2}, \frac{ \frac{d}{2} + H + I} {2};
\frac{d}{2}; - \frac{r^2} {\lambda^2} \right) dr \\
&
-  \| \bm{x^{\prime}} \|^{ 1 - \frac{d} {2} }\left( 2\pi \right)^{ \frac{d} {2} }
\int_{0}^{\infty} r^{ \frac{d}{2} } J_{ \frac{d-2} {2} } (\| \bm{x^{\prime}} \|r)
\ {}_2 {\mathbf F}_1 \left( \frac{I \frac{d}{2} + H} {2}, \frac{I \frac{d}{2} + H + I} {2};
\frac{d}{2}; - \frac{r^2} {\lambda^2} \right) dr \\
& +
\begin{cases}
\displaystyle{ \frac{2  (2\pi)^{m+1} } {(2m+1)!!}
\int_{0}^{\infty} r^{d-1}\ {}_2 {\mathbf F}_1 \left( \frac{ I \frac{d}{2} + H} {2}, \frac{ I \frac{d}{2} + H + I} {2};
\frac{d}{2}; - \frac{r^2} {\lambda^2} \right) dr }, & \textnormal{if} \ \ d = 2m+1; \\
\displaystyle{ \frac{(2\pi)^{m+1} } {(2m)!!}
\int_{0}^{\infty} r^{d-1}\ {}_2 {\mathbf F}_1 \left( \frac{ I \frac{d}{2} + H} {2}, \frac{ I \frac{d}{2} + H + I} {2};
\frac{d}{2}; - \frac{r^2} {\lambda^2} \right) dr }, & \textnormal{if} \ \ d = 2m,
\end{cases} \\
&=: \|\bm{x} - \bm{x^{\prime}}\|^{ 1 - \frac{d} {2} } {\mathbf C}^2_{H,\lambda} (\|{\bm{x- x'}}\|) - \|\bm{x}\|^{ 1 - \frac{d} {2} } {\mathbf C}^2_{H,\lambda} (\|{\bm{x}}\|) - \|\bm{x'}\|^{ 1 - \frac{d} {2} } {\mathbf C}^2_{H,\lambda} (\|{\bm{x'}}\|) + {\mathbf D}^2_{H,\lambda}.
\end{split}
\end{equation}
\end{prop}


Expressions \eqref{TFGFI harmo} and \eqref{tmp} show that the harmonizable representation and covariance function of ITOFBF have complicated forms, involving matrix-valued hypergeometric-based functions. In the next definition, we consider the Gaussian, isotropic version of MA-B-TRF (cf.\ expression \eqref{e:MA-B-TOSSRF}). As it turns out, expressions for the latter are more mathematically manageable.
\begin{defn}
Fix $\lambda>0$. Let $H$ be a Hurst matrix as in \eqref{e:H=PJHP^(-1)} and \eqref{e:0<minReeig(H)=<maxReeig(H)<a1}, and let $\bm{Z}(d\bm{y})$ be a $\bbR^n$-valued Gaussian ISRM with Lebesgue control measure on $\bbR^d$. The random field
\begin{equation} \label{I-B-TOFBFdefn}
 B^{Bes}_{H, \lambda} (\bm{x})
= C^{Bes}_{H,\lambda} \int_{\mathbb{R}^{d}} \left[ \| \bm{x}-\bm{y} \| ^ { H - \frac{d}{2} I } K_{ H - \frac{d}{2} I}
\left( \lambda \| \bm{x}-\bm{y} \| \right) -
\| -\bm{y} \| ^ { H - \frac{d}{2} I } K_{ H - \frac{d}{2} I } \left( \lambda \| -\bm{y} \| \right)
\right] {\bm Z}(d\bm{y})
\end{equation}
is called an \textit{isotropic-Bessel-tempered operator fractional Gaussian field} (I-B-TOFBF). In \eqref{I-B-TOFBFdefn}, $C^{Bes}_{H,\lambda} := \lambda^{(d/2) I-H} \Gamma(H)^{-1}2^{I-H}$, where the terms $\Gamma(H), K_{ H - \frac{d}{2} I }(x) \in {\mathcal M}(n,\bbR)$ are defined by
$$
\Gamma(H) := P \textnormal{diag}(\Gamma(h_1), \hdots,\Gamma(h_n))P^{-1}, \quad K_{ H - \frac{d}{2} I }(x) :=
P \textnormal{diag}\Big( K_{ h_1 - \frac{d}{2}, \hdots, h_n - \frac{d}{2}}(x) \Big)P^{-1}.
$$
\end{defn}


Because $B^{Bes}_{H, \lambda}(\bm{x})$ is an instance of the random field $X^{Bes}_{\lambda}$ with $\varphi(\bm x)= \|\bm x\|$ and $M(d\bm x) = Z(d\bm x)$, all the properties stated in Theorem \ref{thm:Bessel_moving_TOSSF} and Corollary \ref{cor:stationary_increments_scaling_property_MA-B} hold with $B = \frac{1}{2}I$. In addition, \textit{mutatis mutandis}, the second order properties provided in Corollary \ref{cor100} for ITOFBF also hold for I-B-TOFBF and can be shown by similar arguments. For this reason, we state the following result without proof. As with ITOFBF, note that the covariance structure of I-B-TOFBF does have the form \eqref{e:cov_function_under_isotropy}.
\begin{cor}
\label{cor100_Bes}
Let $B^{Bes}_{H, \lambda} = \{B^{Bes}_{H, \lambda}({\bm x})\}_{{\bm x} \in \bbR^d}$ be a I-B-TOFBF. Then,
\begin{itemize}
\item [$(i)$] $B^{Bes}_{H, \lambda}$ is, indeed, isotropic;
\item [$(ii)$] $B^{Bes}_{H, \lambda}$ has covariance function
$$
\textnormal{Cov} [ B^{Bes}_{H, \lambda}(\bm{x}), B^{Bes}_{H, \lambda}(\bm{x^{\prime}}) ]
$$
\begin{equation*} \label{covariance I-B-TOFBF}
= \frac{1}{2} [  \| \bm{x} \|^{H} C^2_{\bm{x},Bes} | \bm{x} \|^{H^t} + \| \bm{x^{\prime}} \|^{H}C^2_{\bm{x^{\prime}},Bes} \|\bm{x^{\prime}} \|^{H^t}
- \| \bm{x-x^{\prime}} \|^{H}C^2_{\bm{x-x^{\prime}},Bes} \|\bm{x-x^{\prime}} \|^{H^t} ]
\end{equation*}
for all $\bm{x}, \bm{x^{\prime}} \in \mathbb{R}^{d}$, where
\begin{equation*}\label{e:Cx_I-B-TOFBF}
\displaystyle{ {\mathcal S}_{\geq 0}(n,\bbR) \ni C^2_{\bm{x},Bes} = \mathbb{E} \left[ B^{Bes}_{H, \| \bm{x} \| \lambda} (\bm{e_{1}})B^{Bes}_{H, \| \bm{x} \| \lambda}(\bm{e_{1}})^{t} \right] } \ , \quad \bm{e^t_{1}} = (1, 0, \hdots, 0).
\end{equation*}
\end{itemize}
\end{cor}
%
In the next proposition, we obtain a harmonizable representation of I-B-TOFBF.
\begin{prop}\label{p:I-B-TOFBF harmo}
Let $B^{Bes}_{H, \lambda} = \{B^{Bes}_{H, \lambda} (\bm{x})\}_{{\bm x} \in \bbR^d}$ be a I-B-TOFBF as in \eqref{I-B-TOFBFdefn} with parameters $\lambda > 0$ and $H = PJ_HP^{-1}$, $P \in GL(n,\bbC)$, where the eigenvalues of the Hurst matrix are real and simple. Then, $B^{Bes}_{H, \lambda}$ admits the harmonizable representation
\begin{equation}\label{I-B-TOFBF harmo}
\{ B^{Bes}_{H, \lambda} (\bm{x})\}_{{\bm x }\in \bbR^d}  \stackrel{f.d.}= \Big\{C^*_{H,\lambda}  \int_{\rr^d} (e^{- i \langle\bm{\xi},\bm{x}\rangle} - 1)  ( \lambda^2 + \|{\bm{\xi}}\|^2 )^{-H} \ W(d\bm{\xi}) \Big\}_{{\bm x}\in \bbR^d}
\end{equation}
for some matrix constant $C^*_{H,\lambda}$, where $W(d\bm{\xi})$ is a $\bbC^n$-valued Gaussian random measure with Lebesgue control measure on $\bbR^d$ such that $W(-d\bm{\xi}) = \overline{W(d\bm{\xi})}$ a.s.
\end{prop}
Let $Y = \{Y({\bm x})\}_{{\bm x} \in \bbR^d}$ be the random field appearing on the right-hand side of \eqref{I-B-TOFBF harmo}. Then, $Y$ is, indeed, a Gaussian H-TRF. In fact, in the parametrization \eqref{e:H-TRF},
\begin{equation}\label{I-B-TOFBF harmo_regular-param}
\{Y({\bm x})\}_{{\bm x }\in \bbR^d}  = \Big\{C^*_{\widetilde{H},\eta}  \int_{\rr^d} (e^{- i \langle\bm{\xi},\bm{x}\rangle} - 1)  ( \widetilde{\lambda} + \widetilde{\varphi}({\bm \xi}) )^{-(\widetilde{H}+\frac{\widetilde{q}}{2}I)} \ W(d\bm{\xi}) \Big\}_{{\bm x}\in \bbR^d}
\end{equation}
for some matrix $C^*_{\widetilde{H},\widetilde{\lambda}}$, where $\widetilde{\lambda} = \lambda^2$, $\widetilde{\varphi}({\bm \xi}) = \|{\bm \xi}\|^2$, $\widetilde{H}+\frac{\widetilde{q}}{2}I = H$. In particular, $Y$ satisfies relation \eqref{e:scaling_H-TRF} with domain and range scaling matrices $\widetilde{E} = \frac{1}{2}I$ and $\widetilde{H}$, respectively, where $\tr(\widetilde{E}) = \widetilde{q}$. Consequently, in this specific sense, I-B-TOFBF appears as a natural moving average-type counterpart to a subclass of the Gaussian, isotropic instances of \eqref{e:H-TRF}.

In the following proposition, we provide an explicit expression for the covariance function of I-B-TOFBF in low dimension $d$, based on the modified Bessel function of the second kind.
\begin{prop}\label{p:cov_low d}
Let $B^{Bes}_{H, \lambda} = \{B^{Bes}_{H, \lambda} (\bm{x})\}_{{\bm x} \in \bbR^d}$ be a I-B-TOFBF as in \eqref{I-B-TOFBFdefn} with parameters $\lambda > 0$ and $H = PJ_HP^{-1}$, $P \in GL(n,\bbC)$, where the eigenvalues of the Hurst matrix are real and simple. Let
\begin{equation}\label{e:(P*P)^(-1)=q_ell,ell'}
(q_{\ell \ell'})_{\ell.\ell'=1,\hdots,n} := (P^*P)^{-1} \in {\mathcal S}_{> 0}(n,\bbC)
\end{equation}
and suppose
\begin{equation}\label{e:h1>d/4}
\varpi_H = h_1 > \frac{d}{4}.
\end{equation}
Then, for $\bm{x},\bm{x^{\prime}} \in \bbR^d$, the covariance function of $B^{Bes}_{H, \lambda}$ can be expressed as
\begin{eqnarray}
\label{e:cov_low d}
& & \displaystyle{ \textnormal{Cov}[ B^{Bes}_{H, \lambda} (\bm{x}), B^{Bes}_{H, \lambda} (\bm{x^{\prime}}) ] } \\
&=&
 \left( 2\pi \right)^{ \frac{d} {2}}  \lambda^{-\frac{d}{2}}
P  \Big( q_{\ell \ell'}\hspace{1mm}\frac { {\|{\bm x} - {\bm x'}\|}^{h_\ell + h_{\ell'} - \frac{d} {2}} } { 2^{h_\ell + h_{\ell'}-1} \Gamma (h_\ell + h_{\ell'}) } K_{\frac{d}{2} - (h_\ell + h_{\ell'})} (\lambda \|{\bm x - x'}\|) \Big)_{\ell,\ell'=1,\hdots,n} P^* \nonumber \\
&& \qquad
- \left( 2\pi \right)^{ \frac{d} {2}}  \lambda^{-\frac{d}{2}}
P  \Big( q_{\ell \ell'}\hspace{1mm}\frac { {\|{\bm x}\|}^{h_\ell + h_{\ell'} - \frac{d} {2}} } { 2^{h_\ell + h_{\ell'}-1} \Gamma (h_\ell + h_{\ell'}) } K_{\frac{d}{2} - (h_\ell + h_{\ell'})} (\lambda \|{\bm x}\|) \Big)_{\ell,\ell'=1,\hdots,n} P^* \nonumber \\
&& \qquad
- \left( 2\pi \right)^{ \frac{d} {2}}  \lambda^{-\frac{d}{2}}
P  \Big( q_{\ell \ell'}\hspace{1mm}\frac { {\|-{\bm x'}\|}^{h_\ell + h_{\ell'} - \frac{d} {2}} } { 2^{h_\ell + h_{\ell'}-1} \Gamma (h_\ell + h_{\ell'}) } K_{\frac{d}{2} - (h_\ell + h_{\ell'})} (\lambda \|-{\bm x'}\|) \Big)_{\ell,\ell'=1,\hdots,n} P^* \nonumber \\
&& \qquad +
(2\pi)^{m}\lambda^{-2H} P \left( q_{\ell \ell'}\hspace{1mm}B \Big( \frac{d}{2}, h_\ell + h_{\ell'} - \frac{d}{2} \Big) \right)_{\ell,\ell'=1,\hdots,n}P^*\begin{cases}
\displaystyle{ \frac{1} {(2m+1)!! } }, & \textnormal{ if }\ \ d = 2m+1; \\
\displaystyle{ \frac{1} {2 (2m)!!}}, & \textnormal{ if }\ \ d = 2m.
\end{cases}\\
&=:&  {\mathcal C}^2_{H,\lambda}(\|{\bm x} - \bm{x^{\prime}}\|) - {\mathcal C}^2_{H,\lambda}(\|{\bm{x}}\|) - {\mathcal C}^2_{H,\lambda}(\|-\bm{x^{\prime}\|} ) + {\mathcal D}^2_{H,\lambda}, \nonumber
\end{eqnarray}
where $B(a,b):= \frac{\Gamma(a)\Gamma(b)}{\Gamma(a+b)}$, $a,b > 0$, is the Beta function.

\end{prop}

\section{Sample path properties of (anisotropic) scalar-valued tempered operator fractional Brownian fields}\label{s:sample_path}

In this section, we investigate sample paths properties of scalar-valued ($n=1$), Gaussian instances of TRF -- hereinafter called (scalar-valued) \textit{moving average- and harmonizable-tempered operator fractional Brownian field} (MA-TOFBF and H-TOFBF, respectively). Note that, unlike in Section \ref{s:Gaussian}, no assumption of isotropy is made. Specifically, we establish the H\"older continuity of the sample paths of $X_{\lambda}$ and $\widetilde{X}_{\lambda}$, and also compute the box-counting and the Hausdorff-dimensions of their graphs. In particular, it is shown that tempering does not affect the sample path properties of MA-TOFBF and H-TOFBF with respect to their non-tempered, OFBF counterparts (cf.\ Bierm\'{e} et al.\ \cite{bierme:meerschaert:scheffler:2007}, Theorem 5.4 and Theorem 5.6).


We start by recapping the definitions of H\"older critical exponent and directional regularity.
\begin{defn}\label{CriticalHolder}
We say a $\bbR$-valued random field $\{ X(\bm{x}) \}_{ \bm{x} \in \mathbb{R}^d }$ has \textit{H\"older critical exponent $\eta \in (0,1)$} if it satisfies the following two properties.
\begin{itemize}
\item[(a)] For any $\beta \in (0, \eta)$, the sample paths of $\{ X(\bm{x}) \}_{ \bm{x} \in \mathbb{R}^d }$ satisfy almost surely a uniform H\"older condition of order $\beta$ on any compact set. That is, for any compact set $K \subset \mathbb{R}^d$, there exists a positive random variable $A$ such that
\begin{eqnarray*}
| X(\bm{x}) - X(\bm{y}) | \leq A \| \bm{x} - \bm{y} \|^{\beta} \quad  \bm{x}, \bm{y} \in K;
\end{eqnarray*}
\item[(b)] for any $\beta \in (\eta, 1)$, the sample paths of $\{ X(\bm{x}) \}_{ \bm{x} \in \mathbb{R}^d }$ do not almost surely satisfy any uniform H\"older condition of order $\beta$.
\end{itemize}
\end{defn}

\begin{defn}
Let $\{ X(\bm{x}) \}_{ \bm{x} \in \mathbb{R}^d }$ be a $\bbR$-valued random field with stationary increments. Let $\bm{r}$ be any direction on the Euclidean unit sphere. We say that $X$ has \textit{regularity $\alpha(\bm{r})$ in the direction $\bm{r}$} if $\{ X(t\bm{r}) \}_{ t \in \mathbb{R} }$ has the H\"older critical exponent $\eta(\bm{r})$.
\end{defn}

Turning to measures of fractal dimension, let $K$ be a compact set on $\mathbb{R}^d$ and let $\mathcal{G}(X)(\omega) = \{ (\bm{x}, X(\bm{x})(\omega)); \bm{x} \in K \}$ be the graph of the path of $\{ X(\bm{x}) \}_{ \bm{x} \in \mathbb{R}^d }$ on the set $K$. Let $dim_{\textnormal{Haus}} \mathcal{G}(X)$ and $dim_{\textnormal{box}} \mathcal{G}(X)$ be the Hausdorff and the box-counting dimensions of $\mathcal{G}(X)$, respectively. Recall that, in most cases, the dimension measures coincide (Falconer \cite{falconer:1990}).

Moreover, let $V_1, \hdots, V_p$ be the spectral decomposition of $\mathbb{R}^d$ with respect to a matrix $E$ as in \eqref{e:E_in_M(d,R)}. Define
\begin{eqnarray}\label{e:Wi}
W_i = V_1 \oplus \hdots \oplus V_i, \quad i = 1, \hdots, p,
\end{eqnarray}
and $W_0 = \{ \bm{0} \}$. Also for $i = 1, \hdots, p$, let
\begin{equation}\label{e:a1<...<ai}
a_1 < \hdots < a_i
\end{equation}
be the real parts of the eigenvalues of $E|_{W_i}$.

We are now in a position to state our main results on the sample path properties of MA-TOFBF and H-OTFBF. Theorem \ref{thm6.1} and Corollary \ref{cor6.3} establish that the H\"{o}lder critical exponent, the directional Hausdorff dimension, and the Hausdorff and box-counting dimensions of the graph of TOFBF.

\begin{thm}\label{thm6.1}
Fix $n = 1$.
\begin{itemize}
 \item [$(i)$] Let $X_{\lambda} = \{X_{\lambda}({\bm x})\}_{{\bm x} \in \bbR^d}$ be a MA-TOFBF for a $\bbR$-valued Gaussian ISRM ${\bm Z}(d{\bm y})$ with Lebesgue control measure on $\bbR^d$ and whose Hurst parameter satisfies $H-\frac{q}{2} \neq 0$. Then, any continuous version of $X_{\lambda}$ has H\"older critical exponent $H/a_p$. Furthermore, for any $i = 1, \hdots, p$ and for any $\bm{r} \in W_i \backslash W_{i-1}$, $X_{\lambda}$ has regularity $H/a_i$ in the direction $\bm{r}$.
 \item [$(ii)$] Let $\widetilde{X}_{\lambda} = \{\widetilde{X}_{\lambda} ({\bm x})\}_{{\bm x} \in \bbR^d}$ be a H-TOFBF for a $\bbC$-valued Gaussian ISRM ${\bm Z}(d{\bm \xi})$ with Lebesgue control measure on $\bbR^d$ such that ${\bm Z}(-d{\bm \xi}) = \overline{{\bm Z}(d{\bm \xi})}$ a.s. Suppose, in addition, that
     \begin{equation}\label{e:UpsilonH<varphiE}
     H < \varpi_E.
     \end{equation}
     Then, the statements in $(i)$ hold for any continuous version of $\widetilde{X}_{\lambda}$.
 \end{itemize}
\end{thm}

\begin{cor}\label{cor6.3}
Suppose the conditions of Theorem \ref{thm6.1} hold. Then, for any continuous version of the MA-TOFBF $X_{\lambda}= \{X_{\lambda}({\bm x})\}_{{\bm x}\in \bbR^d}$,
\begin{eqnarray}\label{e:Haus_box_dim_MA-OFBF}
dim_{\textnormal{Haus}} \mathcal{G} (X_{\lambda}) = dim_{\textnormal{box}} \mathcal{G} (X_{\lambda}) = d + 1 - H/a_p \quad a.s.
\end{eqnarray}
\end{cor}

\begin{rem}
Though we do not provide a proof, the same techniques can be applied to construct, for H-TOFBF, the analogous claim to that in Corollary \ref{cor6.3}.
\end{rem}

\begin{rem}\label{r:Bessel_sample_path}
Preliminary results indicate that the sample path properties of MA-Bessel-TOFBF differ from those of MA-TOFBF and require \textit{ad hoc} efforts. This is the topic of future work.
\end{rem}

\section{Conclusion}

In this paper, we define new and broad classes of vector-valued random fields called \textit{moving average--}, \textit{moving average--Bessel--} and \textit{harmonizable--tempered operator fractional operator-stable random fields} (MA-TRF, MA-B-TRF and H-TRF, respectively). These classes of random fields bring together the research literatures on stable models, anisotropic operator scaling, as well as semi-long range dependence and transient anomalous diffusion in physics. MA-TRF and H-TRF are constructed by tempering (matrix-) homogeneous, matrix-valued kernels in time- and Fourier-domain stochastic integrals with respect to vector-valued, strictly operator-stable random measures. In particular, they generalize tempered fractional stable stochastic processes. We establish the existence and fundamental properties of MA-TRF, MA-B-TRF  and H-TRF. The random fields are generally non-equivalent and non-Gaussian; however, assuming both Gaussianity and isotropy, we show the equivalence between two subclasses of random fields. In addition, we establish sample path properties (H\"{o}lder-exponents and the Hausdorff dimension) in the scalar-valued case for certain Gaussian instances. The results in this paper lead to a number of interesting open problems. First, establishing the relationship between moving average and harmonizable representations of anisotropic TOFBF, whose kernels involve general $E$-homogeneous functions. Second, characterizing symmetries and exponents of TRF, which to the best of our knowledge has not been done for any class of tempered operator fractional random field. Third, studying the sample path properties of TRF in the full multivariate and operator-stable case, including those of the Bessel type. Fourth, starting from the rich framework of TRF, developing applications in physics and statistical modeling.

\section*{Acknowledgments}

The authors are grateful to Mark M.\ Meerschaert for suggesting the original concept of tempered fractional random fields. The authors would like to thank Victor H.\ Moll and Atul Dixit for their suggestions and comments on this work. G.D.\ is grateful to the Physics Lab at ENS de Lyon for the rich research environment and useful discussions. G.D.\ would also like to thank ENS de Lyon, CNRS and the Carol Lavin Bernick faculty grant for supporting his research visits.

\appendix \label{s:appendix}

\section{Section \ref{sec:time_domain}: proofs}

In proofs, $C$ denotes a generic constant that may change from one line to the next.\\

\noindent {\sc Proof of Theorem \ref{thm:exponential_moving_TOSSF}}: Fix ${\bm x} \in \bbR^d$. In view of \eqref{e:suff_cond_existence_I(f)}, it suffices to show that, for some $0 < \delta < \Upsilon^{-1}_B \leq \varpi^{-1}_B$,
$$
\int_{\bbR^d} \Big[ \|e^{-\lambda \varphi({\bm x}-{\bm y})}\varphi({\bm x}-{\bm y})^{H-qB}-e^{-\lambda \varphi(-{\bm y})}\varphi(-{\bm y})^{H-qB}\|^{\frac{1}{\Upsilon_B}-\delta}\Big] d{\bm y}
$$
\begin{equation}\label{e:existence_basic_expression}
+ \int_{\bbR^d} \Big[ \|e^{-\lambda \varphi({\bm x}-{\bm y})}\varphi({\bm x}-{\bm y})^{H-qB}-e^{-\lambda \varphi(-{\bm y})}\varphi(-{\bm y})^{H-qB}\|^{\frac{1}{\varpi_B}+\delta} \Big] d{\bm y} < \infty.
\end{equation}
We first consider the second integral in \eqref{e:existence_basic_expression}. So, fix $\delta > 0$. Recall the elementary bound
\begin{equation}\label{e:(a+b)^p=<(Ca^p+Cb^p}
(a+b)^{p} \leq C_{p}(a^{p}+b^{p}), \quad a,b \geq 0, \quad p > 0,
\end{equation}
where $C_{p} > 0$ does not depend on $a,b$. Then, up to a constant, the second integral in \eqref{e:existence_basic_expression} is bounded by
\begin{equation}\label{e:existence_basic_expression_lambdaB}
\int_{\bbR^d} \Big[ \|e^{-\lambda \varphi({\bm x}-{\bm y})}\varphi({\bm x}-{\bm y})^{H-qB}\|^{\frac{1}{\varpi_B}+\delta} + \|e^{-\lambda \varphi(-{\bm y})}\varphi(-{\bm y})^{H-qB}\|^{\frac{1}{\varpi_B}+\delta} \Big] d{\bm y}.
\end{equation}
Consider the constants $m_{\varphi}, M_{\varphi}> 0$ as in \eqref{e:M-phi>=m-phi}. By a change of variable into polar coordinates induced by $E$ (Lemma \ref{l:polar_integration}), and by the $E$-homogeneity of $\varphi$,
$$
\int_{\bbR^d} \Big[ \|e^{-\lambda \varphi(-{\bm y})}\varphi(-{\bm y})^{H-qB}\|^{\frac{1}{\varpi_B}+\delta} \Big] d{\bm y}
= \int^{\infty}_0 \int_{S_0}  \|e^{-\lambda \varphi(r^{E}\theta)}\varphi(r^{E}\theta)^{H-qB}\|^{\frac{1}{\varpi_B}+\delta} r^{q-1}\sigma(d \theta) dr
$$
$$
= \int^{\infty}_0 \int_{S_0}  \|e^{-\lambda r \varphi(\theta)}r^{H-qB} \varphi(\theta)^{H-qB}\|^{\frac{1}{\varpi_B}+\delta} r^{q-1}\sigma(d \theta) dr
$$
$$
\leq \int^{\infty}_0 e^{-\lambda r m_{\varphi} (\frac{1}{\varpi_B}+\delta )}   \|r^{H-qB}\|^{\frac{1}{\varpi_B}+\delta} r^{q-1} \int_{S_0}  \|\varphi(\theta)^{H-qB}\|^{\frac{1}{\varpi_B}+\delta} \sigma(d \theta) dr
$$
\begin{equation}\label{e:int_|exptemp*phihomog|^(power)dy}
\leq C \int^{\infty}_0  e^{-\lambda r m_{\varphi} (\frac{1}{\varpi_B}+\delta )}  \|r^{H-qB}\|^{\frac{1}{\varpi_B}+\delta} r^{q-1}  dr.
\end{equation}
The integrand in \eqref{e:int_|exptemp*phihomog|^(power)dy} is clearly integrable at $r \rightarrow \infty$. So, we need only consider its behavior around $r = 0$. In fact, for some $n_* \in \bbN$,
\begin{equation}\label{e:int01_exp*power-law}
\int^{1}_0  e^{-\lambda r m_{\varphi} (\frac{1}{\varpi_B}+\delta )}  \|r^{H-qB}\|^{\frac{1}{\varpi_B}+\delta} r^{q-1}  dr \leq \int^{1}_0  |\log r|^{n_*} r^{\varpi_{H-qB} \hspace{0.5mm}(\frac{1}{\varpi_B}+\delta)+q-1} dr.
\end{equation}
So, for integrability it suffices that $\varpi_{H-qB} \hspace{0.5mm}(\frac{1}{\varpi_B}+\delta)+q-1 > -1$. Under condition \eqref{e:lambda(H-qB)+qlambda(B)>0}, this holds for small enough $\delta > 0$. Hence, the right-hand side of \eqref{e:int_|exptemp*phihomog|^(power)dy} is finite. In turn, after a change of variable ${\bm y}' = {\bm x} - {\bm y}$, the first integral in \eqref{e:existence_basic_expression_lambdaB} can be shown to be finite by the same argument. This establishes that the sum in \eqref{e:existence_basic_expression_lambdaB} is finite.

In regard to the first integral in \eqref{e:existence_basic_expression}, by replacing $\frac{1}{\varpi_B}+\delta$ with $\frac{1}{\Upsilon_B}-\delta$ in the argument for \eqref{e:existence_basic_expression_lambdaB}, we conclude that finiteness holds if $\varpi_{H-qB} + q \Upsilon_B > 0$. This condition, in turn, is implied by \eqref{e:lambda(H-qB)+qlambda(B)>0}. Therefore, \eqref{e:existence_basic_expression} holds and $X_{\lambda}({\bm x})$ exists, as claimed. $\Box$\\

\noindent {\sc Proof of Corollary \ref{cor:stationary_increments_scaling_property_MA}}: We first prove part (a). Fix $m \in \bbN$, $u_1,\hdots,u_m \in \bbR^n$ and $r > 0$. By Kremer and Scheffler \cite{kremer:scheffler:2017}, Theorem 5.4, (b), and Example 3.7, (a) (see also expression \eqref{e:chf_stoch_integral} in this paper), the characteristic function of the vector $(X_{\lambda} ({\bm x}_1) , \hdots,X_{\lambda} ({\bm x}_m))$ at $r^{B^*}{\mathbf u}_1,\hdots,r^{B^*}{\mathbf u}_m$ is given by
$$
\bbE \exp \Big\{ \imag \sum^{m}_{j=1}(r^{B^*}{\mathbf u}_j)^t X_{\lambda} ({\bm x}_j) \Big\}
$$
$$
= \exp \Big\{ \int_{\bbR^d} \widehat{\psi}\Big( \sum^{m}_{j=1} \hspace{0.5mm}  \Big[ e^{-\lambda \varphi({\bm x}_j- {\bm y})} \varphi({\bm x}_j-{\bm y})^{H-qB}
- e^{-\lambda \varphi(-{\bm y})} \varphi(-{\bm y})^{H-qB} \Big]^* r^{B^*} {\mathbf u}_j \hspace{1mm} \Big)d {\bm y}  \Big\}.
$$
By condition \eqref{e:HB=BH} and property \eqref{e:s.psi-hat(u)=psi-hat(s^(B^*)u)}, we obtain
$$
\exp \Big\{ \int_{\bbR^d} r \cdot \widehat{\psi}\Big( \sum^{m}_{j=1} \hspace{0.5mm}  \Big[ e^{-\lambda \varphi({\bm x}_j- {\bm y})} \varphi({\bm x}_j-{\bm y})^{H-qB}
- e^{-\lambda \varphi(-{\bm y})} \varphi(-{\bm y})^{H-qB} \Big]^* {\mathbf u}_j \hspace{1mm} \Big)d {\bm y}   \Big\}
$$
$$
= \Big(\bbE \exp \Big\{ \imag \sum^{m}_{j=1}{\mathbf u}^t X_{\lambda} ({\bm x}_j) \Big\} \Big)^r.
$$
Therefore, $X_{\lambda}$ is strictly operator-stable with exponent $B$, as claimed. Moreover, fix $c > 0$ and recall that $q$ is given by expression \eqref{e:q=tr(E)}. Again by using characteristic functions, we can see that the random measure in expression \eqref{e:MA-TOSSRF} satisfies the scaling relation
\begin{equation}\label{e:M(c^Edz)=c^(qB)M(dz)}
M (c^E d{\bm z}) \stackrel{d}{=} c^{q B} M (d{\bm z}), \quad c > 0.
\end{equation}
By condition \eqref{e:HB=BH}, $\varphi({\bm z})^{H - qB} = \varphi({\bm z})^{H}\varphi({\bm z})^{- qB}$. Thus, in view of the $E$-homogeneity of $\varphi$, by a change of variable ${\bm y} = c^E {\bm z}$ we obtain
\begin{eqnarray*}
 \{ X_{\lambda} (c^E {\bm x})\}_{{\bm x} \in \mathbb{R}^d}
&=& \Big\{ \int_{\mathbb{R}^d}
\left( e^{-\lambda \varphi(c^E {\bm x_j} - {\bm y})} \varphi(c^E {\bm x} - {\bm y})^{H - qB}
- e^{-\lambda \varphi(-{\bm y})} \varphi(-{\bm y})^{H - qB} \right) M (d{\bm y}) \Big\}_{{\bm x} \in \mathbb{R}^d}\\
&\stackrel{f.d.}{=}& \Big\{ \int_{\mathbb{R}^d}
\left( e^{-c \lambda \varphi({\bm x_j} - {\bm z})} c^{ H - qB } \varphi({\bm x} - {\bm z})^{H - qB}
- e^{-c \lambda \varphi(-{\bm z})} c^{H - qB} \varphi(-{\bm z})^{H - qB} \right) c^{qB} M (d{\bm z})\Big\}_{{\bm x} \in \mathbb{R}^d} \\
&=& \{ \hspace{1mm}c^H X_{c\lambda} ({\bm x}) \}_{{\bm x} \in \mathbb{R}^d}.
\end{eqnarray*}
This establishes relation (\ref{e:X-lambda(c^Ex)=c^H_X-clambda(x)}). Thus, (a) holds.

We now show (b). By Theorem 5.4, (c), in Kremer and Scheffler \cite{kremer:scheffler:2017}, after a change of variables it suffices to show that
\begin{equation}\label{e:int_psi-tilde->0}
\int_{\bbR^d} \widehat{\psi}\Big([e^{-\lambda \varphi({\bm x} - {\bm y})}\varphi({\bm x} - {\bm y})^{H-qB}- e^{-\lambda \varphi(- y)}\varphi(- y)^{H-qB}]^*{\mathbf u} \Big) d{\bm y} \rightarrow 0, \quad {\bm x} \rightarrow 0,
\end{equation}
for any ${\mathbf u} \in \bbR^n \backslash \{0\}$. The proof of \eqref{e:int_psi-tilde->0} consists in, first, constructing an integrable function that bounds the integrand of \eqref{e:int_psi-tilde->0} for every small $\|{\bm x}\|$. Then, by the dominated convergence theorem and the continuity of the function $\widehat{\psi}$, we can conclude that $X_{\lambda}$ is stochastically continuous at ${\bm x} = 0$, as claimed.

So, without loss of generality, suppose
\begin{equation}\label{e:|x|=<1}
\|{\bm x}\| \leq 1 - \eta_0
\end{equation}
for some fixed $\eta_0 \in (0,1)$. For notational simplicity, let
\begin{equation}\label{e:fx(y)}
f_{{\bm x}}({\bm y}) = e^{-\lambda \varphi({\bm x} - {\bm y})}\varphi({\bm x} - {\bm y})^{H-qB}- e^{-\lambda \varphi(- {\bm y})}\varphi(- {\bm y})^{H-qB}, \quad {\bm x} \in \bbR^d.
\end{equation}
Bearing in mind that $\widehat{\psi}(0) = 0$, as in Kremer and Scheffler \cite{kremer:scheffler:2019}, Theorem 2.5, we define the sets
$$
A_0 = \Big\{{\bm y}: \|f_{{\bm x}}({\bm y})\| \leq 1, f_{{\bm x}}({\bm y})^*{\mathbf u} \neq 0\Big\}, \quad
A_1 = \Big\{{\bm y}: \|f_{{\bm x}}({\bm y})\| > 1, 0 < \|f_{{\bm x}}({\bm y})^*{\mathbf u}\| \leq  1\Big\},
$$
$$
A_2 = \Big\{{\bm y}: \|f_{{\bm x}}({\bm y})\| > 1, \|f_{{\bm x}}({\bm y})^*{\mathbf u}\| > 1\Big\}.
$$
For ${\bm z} \in \bbR^n \backslash \{0\}$, let ${\bm z} = \tau_{B^*}({\bm z})^{B^*} l_{B^*}({\bm z})$ be the decomposition of ${\bm z}$ into polar coordinates induced by the matrix $B^*$, and note that $\varpi_{B^*} = \varpi_{B}$, $\Upsilon_{B^*} = \Upsilon_{B}$. So, for any small $\delta_1, \delta_2 >0$, by property \eqref{e:s.psi-hat(u)=psi-hat(s^(B^*)u)} of log-characteristic functions and by Lemma \ref{lem1.1} with $B^*$ in place of $E$,
$$
\Big| \widehat{\psi}(f_{{\bm x}}({\bm y})^* {\mathbf u} )\Big|
= \Big| \widehat{\psi}\Big(\tau_{B^*}(f_{{\bm x}}({\bm y})^* {\mathbf u} )^{B^*} l_{B^*}(f_{{\bm x}}({\bm y})^* {\mathbf u}) \Big)\Big| \Big( 1_{A_0}({\bm y})+1_{A_1}({\bm y})+1_{A_2}({\bm y})\Big)
$$
$$
\leq C 1_{A_0}({\bm y})\hspace{0.5mm}\|f_{{\bm x}}({\bm y})^*{\mathbf u}\|^{\frac{1}{\Upsilon_B}- \delta_1}
+ C' 1_{A_1}({\bm y}) \hspace{0.5mm} \|f_{{\bm x}}({\bm y})^*{\mathbf u}\|^{\frac{1}{\Upsilon_B}-\delta_1}
+ C'' 1_{A_2}({\bm y}) \hspace{0.5mm}\|f_{{\bm x}}({\bm y})^*{\mathbf u}\|^{\frac{1}{\varpi_B}+ \delta_2}
$$
$$
\leq C 1_{A_0}({\bm y})\hspace{0.5mm}\|f_{{\bm x}}({\bm y})^*{\mathbf u}\|^{\frac{1}{\Upsilon_B}- \delta_1}
+ C' 1_{A_1}({\bm y}) \hspace{0.5mm}\|{\mathbf u}\|^{\frac{1}{\Upsilon_B}-\delta_1} \|f_{{\bm x}}({\bm y})\|^{\frac{1}{\varpi_B}+\delta_2}
+ C'' 1_{A_2}({\bm y}) \hspace{0.5mm}\Big( \|{\mathbf u}\| \hspace{0.5mm}\|f_{{\bm x}}({\bm y})\|\Big)^{\frac{1}{\varpi_B}+ \delta_2}
$$
\begin{equation}\label{e:psi-tilde(f_x(y))=<three_terms}
\leq C 1_{\{{\bm y}:\|f_{{\bm x}}({\bm y})\| \leq 1 \}} \hspace{0.5mm} \Big(\|f_{{\bm x}}({\bm y})\| \hspace{0.5mm}\|{\mathbf u}\|\Big)^{\frac{1}{\Upsilon_B}- \delta_1}
+ \Big[C' \|{\mathbf u}\|^{\frac{1}{\Upsilon_B}-\delta_1}  + C'' \|{\mathbf u}\|^{\frac{1}{\varpi_B}+ \delta_2}\Big]
\hspace{0.5mm}\|f_{{\bm x}}({\bm y})\|^{\frac{1}{\varpi_B}+ \delta_2}1_{\{{\bm y}:\|f_{{\bm x}}({\bm y})\| > 1\}}.
\end{equation}
We begin by considering the first sum term on the right-hand side of \eqref{e:psi-tilde(f_x(y))=<three_terms}, i.e.,
\begin{equation}\label{e:1st_term_RHS_bound_psi-hat(f*u)}
C \hspace{0.5mm}  1_{\{{\bm y}:\|f_{{\bm x}}({\bm y})\| \leq 1 \}} \hspace{0.5mm} \Big(\|f_{{\bm x}}({\bm y})\| \hspace{0.5mm}\|{\mathbf u}\|\Big)^{\frac{1}{\Upsilon_B}- \delta_1}.
\end{equation}
Restricted to the range $\|{\bm y}\|\leq 1$, since $\delta_1 > 0$ is assumed small enough, it is clear that
\begin{equation}\label{e:|f|^power*indicator_small_f_and_y}
\|f_{{\bm x}}({\bm y})\|^{\frac{1}{\Upsilon_B}- \delta_1} \hspace{0.5mm}1_{\{{\bm y}:\|f_{{\bm x}}({\bm y})\| \leq 1 \hspace{0.5mm} \cap \hspace{0.5mm}\|{\bm y}\| \leq 1\}} \leq 1_{\{{\bm y}:\hspace{0.5mm}\|{\bm y}\| \leq 1\}} \in L^1(\bbR^d).
\end{equation}
Now consider \eqref{e:1st_term_RHS_bound_psi-hat(f*u)} over the range $\|{\bm y}\| > 1$. By \eqref{e:(a+b)^p=<(Ca^p+Cb^p},
\begin{equation}\label{e:psi-tilde(f_x(y))=<three_terms_1st_term}
\|f_{{\bm x}}({\bm y})\|^{\frac{1}{\Upsilon_B}- \delta_1}
\leq C \hspace{0.5mm} \Big( \|e^{-\lambda \varphi({\bm x} - {\bm y})}\varphi({\bm x} - {\bm y})^{H-qB}\|^{\frac{1}{\Upsilon_B}- \delta_1} + \| e^{-\lambda \varphi(- {\bm y})}\varphi(- {\bm y})^{H-qB}\|^{\frac{1}{\Upsilon_B}- \delta_1} \Big)
\end{equation}
(for any ${\bm y}$). Let $\tau_E$ be the radial component induced by the polar decomposition of a vector induced by the matrix $E$. By condition \eqref{e:|x|=<1} and Lemma \ref{lem1.1}, for small $\delta > 0$ and $\|{\bm y}\| > 1$, we can bound
\begin{equation}\label{e:tau(x-y)>=|y-1|^(1/ap-delta-1)}
\tau_E({\bm x}- {\bm y}) \geq C\hspace{0.5mm}  \|{\bm x} - {\bm y}\|^{\frac{1}{\Upsilon_E}-\delta}_0 \geq C \hspace{0.5mm}  |\hspace{0.5mm} \|{\bm y}\|_0 - 1 |^{\frac{1}{\Upsilon_E}-\delta} .
\end{equation}
Therefore,
$$
\|e^{-\lambda \varphi({\bm x}- {\bm y})}\varphi({\bm x}- {\bm y})^{H-qB}\| \leq C e^{-\frac{\lambda}{2}\varphi({\bm x}- {\bm y})} \leq  C e^{-C'| \hspace{0.5mm}\|{\bm y}\|_0 - 1 |^{\frac{1}{\Upsilon_E}-\delta}}.
$$
A similar bound can be constructed for $\|e^{-\lambda \varphi(- {\bm y})}\varphi(- {\bm y})^{H-qB}\|$. In view of \eqref{e:psi-tilde(f_x(y))=<three_terms_1st_term} restricted to $\|{\bm y}\| > 1$,
\begin{equation}\label{e:|fx(y)|^power*indic=<bound}
\|f_{{\bm x}}({\bm y})\|^{\frac{1}{\Upsilon_B}- \delta_1} 1_{\{{\bm y}: \|{\bm y}\| > 1\}}
\leq C \Big( e^{-C'({\frac{1}{\Upsilon_B}- \delta_1}) \hspace{0.5mm}  | \hspace{0.5mm}\|{\bm y}\|_0 - 1 |^{\frac{1}{\Upsilon_E}-\delta}} + \| e^{-\lambda \varphi(- {\bm y})}\varphi(- {\bm y})^{H-qB}\|^{\frac{1}{\Upsilon_B}- \delta_1} \Big) 1_{\{{\bm y}: \|{\bm y}\| > 1\}}.
\end{equation}
We conclude that the first sum term on the right-hand side of \eqref{e:psi-tilde(f_x(y))=<three_terms} is bounded uniformly in ${\bm x}$ by a $d{\bm y}$-integrable function. Therefore,
$$
\int_{\bbR^d}\|f_{{\bm x}}({\bm y})\|^{\frac{1}{\Upsilon_B}- \delta_1}1_{\{{\bm y}:\|f_{{\bm x}}({\bm y})\| \leq 1\}} d{\bm y} \rightarrow 0, \quad {\bm x}\rightarrow 0.
$$

We now turn to the second sum term on the right-hand side of \eqref{e:psi-tilde(f_x(y))=<three_terms}, i.e.,
\begin{equation}\label{e:2nd_term_RHS_bound_psi-hat(f*u)}
\Big[C' \|{\mathbf u}\|^{\frac{1}{\Upsilon_B}-\delta_1}  + C'' \|{\mathbf u}\|^{\frac{1}{\varpi_B}+ \delta_2}\Big]
\hspace{0.5mm}\|f_{{\bm x}}({\bm y})\|^{\frac{1}{\varpi_B}+ \delta_2}1_{\{{\bm y}:\|f_{{\bm x}}({\bm y})\| > 1\}}.
\end{equation}
For the range $\|{\bm y}\| > 1$, one can adapt the argument leading up to \eqref{e:|fx(y)|^power*indic=<bound}. So, we can assume $\|{\bm y}\| \leq 1$. In view of \eqref{e:psi-tilde(f_x(y))=<three_terms_1st_term} with $\frac{1}{\varpi_B}+ \delta_2$ in place of $\frac{1}{\Upsilon_B}- \delta_1$, it suffices to consider the function
$$
\|e^{-\lambda \varphi({\bm x} - {\bm y})}\varphi({\bm x} - {\bm y})^{H-qB}\|^{\frac{1}{\varpi_B}+ \delta_2}1_{\{{\bm y}:\|f_{{\bm x}}({\bm y})\| > 1 \cap \|{\bm y}\| \leq 1\}},
$$
which in turn is bounded by
$$
\|e^{-\lambda \varphi({\bm x} - {\bm y})}\varphi({\bm x} - {\bm y})^{H-qB}\|^{\frac{1}{\varpi_B}+ \delta_2}1_{\{{\bm y}: \|{\bm y}\| \leq 1\}} \rightarrow
\|e^{-\lambda \varphi(- {\bm y})}\varphi(- {\bm y})^{H-qB}\|^{\frac{1}{\varpi_B}+ \delta_2}1_{\{{\bm y}: \|{\bm y}\| \leq 1\}}, \quad {\bm x}\rightarrow 0.
$$
Note that, after a change of variable ${\bm z}={\bm x} - {\bm y}$,
$$
\int_{\bbR^d}\|e^{-\lambda \varphi({\bm x} - {\bm y})}\varphi({\bm x} - {\bm y})^{H-qB}\|^{\frac{1}{\varpi_B}+ \delta_2}1_{\{{\bm y}: \|{\bm y}\| \leq 1\}}
d{\bm y} =
\int_{\bbR^d}\|e^{-\lambda \varphi({\bm z})}\varphi({\bm z})^{H-qB}\|^{\frac{1}{\varpi_B}+ \delta_2}1_{\{{\bm z}: \|{\bm x}-{\bm z}\| \leq 1\}}
d{\bm z},
$$
where
$$
\|e^{-\lambda \varphi({\bm z})}\varphi({\bm z})^{H-qB}\|^{\frac{1}{\varpi_B}+ \delta_2}1_{\{{\bm z}: \|{\bm x}-{\bm z}\| \leq 1\}} \leq
\|e^{-\lambda \varphi({\bm z})}\varphi({\bm z})^{H-qB}\|^{\frac{1}{\varpi_B}+ \delta_2}1_{\{{\bm z}: \|{\bm z}\| \leq 2\}}, \quad \|{\bm x}\|\leq 1 - \eta_0.
$$
Therefore, by the dominated convergence theorem, as ${\bm x}\rightarrow 0$,
$$
\int_{\bbR^d}\|e^{-\lambda \varphi({\bm x} - {\bm y})}\varphi({\bm x} - {\bm y})^{H-qB}\|^{\frac{1}{\varpi_B}+ \delta_2}1_{\{{\bm y}: \|{\bm y}\| \leq 1\}}
d{\bm y} \rightarrow
\int_{\bbR^d}\|e^{-\lambda \varphi(- {\bm y})}\varphi(- {\bm y})^{H-qB}\|^{\frac{1}{\varpi_B}+ \delta_2}1_{\{{\bm y}: \|{\bm y}\| \leq 1\}}
d{\bm y}.
$$
Thus, again by the dominated convergence theorem,
$$
\int_{\bbR^d}\|f_{{\bm x}}({\bm y})\|^{\frac{1}{\varpi_B}+ \delta_2}1_{\{{\bm y}:\|f_{{\bm x}}({\bm y})\| > 1\}} d{\bm y} \rightarrow 0, \quad {\bm x}\rightarrow 0.
$$
This establishes (b).

In regard to (c), we need to show that
\begin{eqnarray*}
\{ X_{\lambda} ({\bm x} + {\bm h}) - X_{\lambda} ({\bm h}) \}_{{\bm x} \in \mathbb{R}^d}
\stackrel{f.d.}{=} \{ X_{\lambda} ({\bm x}) \}_{{\bm x} \in \mathbb{R}^d}, \quad {\bm h} \in \bbR^d.
\end{eqnarray*}
For this purpose, we use characteristic functions. In fact, for any ${\bm h} \in \bbR^d$ and for any $m \in \bbN$, fix ${\bm x}_1,\hdots,{\bm x}_m$. Then, by a change of variable ${\bm h} - {\bm y} = {\bm z}$,
$$
\bbE \exp \Big\{ \imag \sum^{m}_{j=1} {\mathbf u}^t_j \hspace{0.5mm}( X_{\lambda} ({\bm x}_j + {\bm h}) - X_{\lambda} ({\bm h}) ) \Big\}
$$
$$
= \bbE \exp \Big\{ \imag \sum^{m}_{j=1} {\mathbf u}^t_j \hspace{0.5mm}\Big( \int_{\bbR^d} \Big[ e^{-\lambda \varphi({\bm x}_j+{\bm h}-{\bm y})} \varphi({\bm x}_j+{\bm h}-{\bm y})^{H-qB}
- e^{-\lambda \varphi({\bm h}-{\bm y})} \varphi({\bm h}-{\bm y})^{H-qB} \Big] M(d {\bm y}) \Big)  \Big\}
$$
$$
= \exp \Big\{ \int_{\bbR^d} \widehat{\psi}\Big( \sum^{m}_{j=1} \Big[ e^{-\lambda \varphi({\bm x}_j+{\bm h}-{\bm y})} \varphi({\bm x}_j+{\bm h}-{\bm y})^{H-qB}
- e^{-\lambda \varphi({\bm h}-{\bm y})} \varphi({\bm h}-{\bm y})^{H-qB} \Big]^* {\mathbf u}_j \Big)\hspace{1mm}d {\bm y}\Big\}
$$
$$
= \exp \Big\{ \int_{\bbR^d} \widehat{\psi}\Big( \sum^{m}_{j=1}  \Big[ e^{-\lambda \varphi({\bm x}_j-{\bm z})} \varphi({\bm x}_j-{\bm z})^{H-qB}
- e^{-\lambda \varphi(-{\bm z})} \varphi(-{\bm z})^{H-qB} \Big]^* {\mathbf u}_j \hspace{1mm}\Big)  \hspace{1mm}d {\bm z} \Big\}
$$
$$
= \bbE \exp \Big\{\imag \sum^{m}_{j=1} {\mathbf u}^t_j \hspace{0.5mm}\Big( \int_{\bbR^d} \Big[ e^{-\lambda \varphi({\bm x}-{\bm z})} \varphi({\bm x}_j-{\bm z})^{H-qB}
- e^{-\lambda \varphi(-{\bm z})} \varphi(-{\bm z})^{H-qB} \Big]  M(d{\bm z})  \Big) \Big\}.
$$
This establishes (c).

To show (d), note that, by Proposition 2.6, (a), in Kremer and Scheffler \cite{kremer:scheffler:2019}, for fixed ${\bm x} \in \bbR^d \backslash\{0\}$, it suffices to show that there is a positive Lebesgue measure set on which the kernel \eqref{e:fx(y)} has full rank. In fact, by way of contradiction, suppose
\begin{equation}\label{e:Leb-d(kernel(x-y)_neq_kernel(-y))=0}
\textnormal{Leb}_d \Big\{{\bm y}: \det\Big[e^{-\lambda \varphi({\bm x}-{\bm y})}\varphi({\bm x}-{\bm y})^{H-qB} - e^{-\lambda \varphi(-{\bm y})}\varphi(-{\bm y})^{H-qB}\Big] \neq 0 \Big\} = 0.
\end{equation}
Without loss of generality, consider the most encompassing case where
\begin{equation}\label{e:H-qB=Pdiag(J-,J.,J+)P^(-1)}
H-qB = P\textnormal{diag}(J_{\ominus}, J_{\odot}, J_{\oplus})P^{-1}, \quad P \in GL(n,\bbC),
\end{equation}
is the Jordan decomposition of $H-qB$. In \eqref{e:H-qB=Pdiag(J-,J.,J+)P^(-1)}, the matrices $J_{\ominus}$, $J_{\odot}$ and $J_{\oplus}$ contain Jordan blocks with negative, zero and positive real parts, respectively.  Note that
$$
\det\Big[e^{-\lambda \varphi({\bm x}-{\bm y})}\varphi({\bm x}-{\bm y})^{H-qB} - e^{-\lambda \varphi(-{\bm y})}\varphi(-{\bm y})^{H-qB}\Big]
$$
$$
= \prod_{J=J_{\ominus},J_{\odot},J_{\oplus}} \det\Big( e^{-\lambda \varphi({\bm x}-{\bm y})}\varphi({\bm x}-{\bm y})^{J} - e^{-\lambda \varphi(-{\bm y})}\varphi(-{\bm y})^{J}\Big).
$$
Consider, first, $J = J_{\oplus}$. Then, by taking ${\bm y}\rightarrow 0$ and by continuity,
\begin{equation}\label{e:exp_temp_J+}
\det\Big[ e^{-\lambda \varphi({\bm x}-{\bm y})}\varphi({\bm x}-{\bm y})^{J_\oplus} - e^{-\lambda \varphi(-{\bm y})}\varphi(-{\bm y})^{J_\oplus}\Big] \rightarrow \det[e^{-\lambda \varphi({\bm x})}\varphi({\bm x})^{J_{\oplus}}] > 0.
\end{equation}
For $J = J_{\ominus}$, $\|e^{-\lambda \varphi(-{\bm y})}\varphi(-{\bm y})^{J_{\ominus}}\|\rightarrow \infty$ as ${\bm y}\rightarrow 0$. Hence,
\begin{equation}\label{e:exp_temp_J-}
\Big| \det\Big[ e^{-\lambda \varphi({\bm x}-{\bm y})}\varphi({\bm x}-{\bm y})^{J_\ominus} - e^{-\lambda \varphi(-{\bm y})}\varphi(-{\bm y})^{J_\ominus}\Big]\hspace{1mm} \Big|\rightarrow \infty.
\end{equation}
For $J = J_{\odot}$, each eigenvalue of $J_{\odot}$ can be expressed as $\imag \eta$, where $\eta \in \bbR$. In view of \eqref{e:z^Jlambda}, for each diagonal entry of the matrix-valued expression $e^{-\lambda \varphi({\bm x}-{\bm y})}\varphi({\bm x}-{\bm y})^{J_{\odot}} - e^{-\lambda \varphi(-{\bm y})}\varphi(-{\bm y})^{J_{\odot}}$,
$$
|e^{-\lambda \varphi({\bm x}-{\bm y})}\varphi({\bm x}-{\bm y})^{\imag \eta} - e^{-\lambda \varphi(-{\bm y})}\varphi(-{\bm y})^{\imag \eta}|\sim |e^{-\lambda \varphi({\bm x})}\varphi({\bm x})^{\imag \eta }  - e^{-\lambda \varphi(-{\bm y})}\varphi(-{\bm y})^{\imag \eta}|, \quad {\bm y}\rightarrow 0.
$$
If $\eta = 0$, since $\varphi({\bm x}) > 0$ and $\varphi$ is $E$-homogeneous and continuous, we obtain, as ${\bm y}\rightarrow 0$,
$$
|e^{-\lambda \varphi({\bm x})} - 1 | \neq 0.
$$
Alternatively, if $\eta \neq 0$, since $\varphi$ is $E$-homogeneous and continuous, then $\varphi(-{\bm y})^{\imag \eta} = \exp \{{\imag \eta}\hspace{0.5mm} \log \varphi(-{\bm y})\}$ does not converge. Therefore, Lemma \ref{lem1.1} implies that
\begin{equation}\label{e:exp_temp_J.}
\Big|\det\Big( e^{-\lambda \varphi({\bm x}-{\bm y})}\varphi({\bm x}-{\bm y})^{J_\odot} - e^{-\lambda \varphi(-{\bm y})}\varphi(-{\bm y})^{J_\odot}\Big)\Big|
\end{equation}
is bounded away from zero over a positive $d{\bm y}$ measure set around the origin ${\bm y} = 0$. Thus, in view of \eqref{e:exp_temp_J+}, \eqref{e:exp_temp_J-} and \eqref{e:exp_temp_J.}, relation \eqref{e:Leb-d(kernel(x-y)_neq_kernel(-y))=0} cannot hold. This establishes (d). $\Box$\\

\noindent {\sc Proof of Theorem \ref{thm:Bessel_moving_TOSSF}}: Fix ${\bm x} \in \bbR^d$. In view of \eqref{e:suff_cond_existence_I(f)}, it suffices to show that, for some $0 < \delta < \Upsilon^{-1}_B \leq \varpi^{-1}_B$,
$$
\int_{\bbR^d} \Big[ \|\varrho_{H-qB,\lambda}({\bm x} - {\bm y})\varphi({\bm x}-{\bm y})^{H-qB}-\varrho_{H-qB,\lambda}(- {\bm y})\varphi(-{\bm y})^{H-qB}\|^{\frac{1}{\Upsilon_B}-\delta}\Big] d{\bm y}
$$
\begin{equation}\label{e:existence_MA-B_basic_expression}
+ \int_{\bbR^d} \Big[\varrho_{H-qB,\lambda}({\bm x}- {\bm y})\varphi({\bm x}-{\bm y})^{H-qB}-\varrho_{H-qB,\lambda}(- {\bm y})\varphi(-{\bm y})^{H-qB}\|^{\frac{1}{\varpi_B}+\delta} \Big] d{\bm y} < \infty.
\end{equation}
So, fix $\delta > 0$. By the elementary bound \eqref{e:(a+b)^p=<(Ca^p+Cb^p}, up to a constant the second sum term on the right-hand side of \eqref{e:existence_MA-B_basic_expression} is bounded by
\begin{equation}\label{e:int_Bess-filter^1/lambdaB+delta}
\int_{\bbR^d} \Big[ \|\varrho_{H-qB,\lambda}({\bm x} - {\bm y})\varphi({\bm x}-{\bm y})^{H-qB}\|^{\frac{1}{\varpi_B}+\delta} + \|\varrho_{H-qB,\lambda}(-{\bm y})\varphi(-{\bm y})^{H-qB}\|^{\frac{1}{\varpi_B}+\delta} \Big] d{\bm y}.
\end{equation}
Because we can make a change of variable ${\bm y}' = {\bm x} - {\bm y}$ in the first integral in \eqref{e:int_Bess-filter^1/lambdaB+delta}, it suffices to show that the second integral in \eqref{e:int_Bess-filter^1/lambdaB+delta} is finite, i.e.,
\begin{equation}\label{e:int_Bess-filter_2nd_term^1/lambdaB+delta_finite}
\int_{\bbR^d} \|\varrho_{H-qB,\lambda}(- {\bm y})\varphi(-{\bm y})^{H-qB}\|^{\frac{1}{\varpi_B}+\delta} d{\bm y} < \infty.
\end{equation}
So, fix a matrix $N \in {\mathcal M}(n,\bbC)$. Without loss of generality, suppose we can decompose
\begin{equation}\label{e:N=Pdiag(J-,J.,J+)P^(-1)}
N = P\textnormal{diag}(J_{\ominus}, J_{\odot}, J_{\oplus})P^{-1},
\end{equation}
where the matrices $J_{\ominus}$, $J_{\odot}$ and $J_{\oplus}$ contain Jordan blocks with negative, zero and positive real parts, respectively. We study the behavior of $K_N(u)$, defined as the primary matrix function \eqref{e:Bessel}, as $u \rightarrow 0^+$ and $\infty$. First, consider the limit $u \rightarrow 0^+$. By a change of variable $w = \frac{u \hspace{0.5mm}e^t}{2}$, for some $n_*$ we can rewrite
\begin{equation}\label{e:KN(x)_as_x->0}
P^{-1}K_N(u) P
\end{equation}
$$
= \begin{pmatrix}
2^{-J_{\ominus}-I} u^{J_{\ominus}} \int^{\infty}_{u/2} e^{-\frac{u^2}{4w}} e^{-w} w^{-J_{\ominus}-I} dw + O_{u \rightarrow 0}(1) & 0 & 0 \\
0 &  |\log u|^{n_*} O_{u \rightarrow 0}(1)  & 0\\
0 & 0 & 2^{J_{\oplus}-I} u^{-J_{\oplus}} \int^{\infty}_{u/2} e^{- \frac{u^2}{4w}} e^{-w} w^{J_{\oplus}-I} dw + O_{u \rightarrow 0}(1)
\end{pmatrix},
$$
where $O_{u \rightarrow 0}(1)$ is a matrix whose norm is bounded in $u$. Turning to the limit $u \rightarrow \infty$, fix any $\eta \in (0,1)$. By the same change of variable $w = \frac{u \hspace{0.5mm}e^t}{2}$,
$$
\|K_{J_{\oplus}}(u)\| = \Big\|2^{J_{\oplus}-I} u^{-J_{\oplus}}\int^{\infty}_{u/2} e^{- \frac{u^2}{4w}} e^{-wI} w^{J_{\oplus}-I} dw
+ 2^{-J_{\oplus}-I} u^{J_{\oplus}}\int^{\infty}_{u/2} e^{- \frac{u^2}{4w}} e^{-wI} w^{-J_{\oplus}-I} dw \Big\|
$$
$$
\leq \max\{\|u^{-J_{\oplus}}\|,\|u^{J_{\oplus}}\|  \} \hspace{0.5mm} e^{-\frac{u}{2}\eta} \hspace{0.5mm}\Bigg\{\|2^{J_{\oplus}-I}\| \hspace{1mm} \Big\|\int^{\infty}_{u/2} e^{- \frac{u^2}{4w}} e^{-w(1-\eta)I} w^{J_{\oplus}-I} dw \Big\|
$$
$$
+ \|2^{-J_{\oplus}-I}\| \hspace{1mm} \Big\|\int^{\infty}_{u/2} e^{- \frac{u^2}{4w}} e^{-w(1-\eta)I} w^{-J_{\oplus}-I} dw \Big\|\Bigg\}
\leq C e^{-\frac{u}{2}\eta_*} o_{u \rightarrow \infty}(1)
$$
for some $\eta_* \in(0, \eta)$, where $o_{u \rightarrow \infty}(1)$ is a matrix whose norm goes to zero as $u \rightarrow \infty$. The same bound holds for $\|K_{J_{\ominus}}(u)\|$ and $\|K_{J_{\odot}}(u)\|$. Therefore, from \eqref{e:N=Pdiag(J-,J.,J+)P^(-1)},
\begin{equation}\label{e:|KJ+(x)|)}
\|K_{N}(u)\| \leq  e^{-\frac{u}{2}\eta_*} o_{u \rightarrow \infty}(1), \quad u \rightarrow \infty,
\end{equation}
%
for some $\eta_* \in (0,1)$. So, turning back to the integral in \eqref{e:int_Bess-filter_2nd_term^1/lambdaB+delta_finite}, by a change of variable into polar coordinates induced by $E$ (Lemma \ref{l:polar_integration}), we obtain
$$
\int_{\bbR^d} \|\varrho_{H-qB,\lambda}(- {\bm y})\varphi(-{\bm y})^{H-qB}\|^{\frac{1}{\varpi_B}+\delta} d{\bm y}
$$
$$
= \Bigg\{\int^{1}_{0} + \int^{\infty}_{1} \Bigg\}\int_{S_0} \|K_{H-qB}(\lambda r \varphi(\theta))\hspace{1mm}(r\varphi(\theta))^{H-qB}\|^{\frac{1}{\varpi_B}+\delta} r^{q-1}\sigma(\theta) dr.
$$
To show that this integral is finite, we first consider the integration range $[0,1] \times S_0$. Set
$$
N = H-qB, \quad u = \lambda r \varphi(\theta),
$$
in expression \eqref{e:KN(x)_as_x->0}, and note that, in block-diagonal form, we can write $r^{H-qB} = P \textnormal{diag}(r^{J_{\ominus}},r^{J_{\odot}},r^{J_{\oplus}})P^{-1}$. Then, for some $n_{*}$ we obtain
$$
K_{H-qB}(\lambda r \varphi(\theta))\hspace{0.5mm}r^{H-qB}
= P \textnormal{diag}\Big(r^{2J_{\ominus}} \hspace{0.5mm} O_{r,\theta}(1) ,  |\log r|^{n_*} O_{r,\theta}(1)  , O_{r,\theta}(1)\Big)P^{-1},
$$
where each $O_{r,\theta}(1)$ denotes a matrix with bounded norm in both $r$ and $\theta$. Therefore,
$$
\int^{1}_{0}  \int_{S_0} \|K_{H-qB}(\lambda r \varphi(\theta)) \hspace{1mm}(r\varphi(\theta))^{H-qB}\|^{\frac{1}{\varpi_B}+\delta} r^{q-1}\sigma(\theta) dr
$$
$$
\leq C' \int^{1}_{0} \int_{S_0} \Big\| \textnormal{diag}\Big(r^{2J_{\ominus}} \hspace{0.5mm} O_\theta(1) , C |\log r|^{n_*} I , O_\theta(1)\Big) \Big\|^{\frac{1}{\varpi_B}+\delta} r^{q-1}\sigma(d \theta) dr
$$
\begin{equation}\label{e:int01_MA-B_kernel}
\leq C'' \int^{1}_{0}  \Big\| \textnormal{diag}\Big(r^{2J_{\ominus}}, |\log r|^{n_*} I, I\Big) \Big\|^{\frac{1}{\varpi_B}+\delta} r^{q-1} dr.
\end{equation}
For the integral on the right-hand side of \eqref{e:int01_MA-B_kernel} to be finite, it suffices that $2 \varpi_{H-qB}(\frac{1}{\varpi_B}+\delta) + (q-1) > -1$. This holds under condition \eqref{e:MA-B_condition_eig_exponent}. Turning to the integration range $(1,\infty) \times S_0$, by setting $x = \lambda r \varphi(\theta)$ in expression \eqref{e:|KJ+(x)|)}, we arrive at
$$
\int^{\infty}_{1} \int_{S_0} \|K_{H-qB}(\lambda r \varphi(\theta))(r\varphi(\theta))^{H-qB}\|^{\frac{1}{\varpi_B}+\delta} r^{q-1}\sigma(\theta) dr
$$
$$
\leq \int^{\infty}_{1} \int_{S_0} \|e^{-\frac{x}{2}\eta_*} o_{r}(1) r^{H-qB}\|^{\frac{1}{\varpi_B}+\delta} r^{q-1}\sigma(\theta) dr < \infty.
$$
Therefore, \eqref{e:int_Bess-filter_2nd_term^1/lambdaB+delta_finite} holds.

By a similar procedure with $\frac{1}{\Upsilon_B} - \delta$ in place of $\frac{1}{\varpi_B} + \delta$, condition \eqref{e:MA-B_condition_eig_exponent} ensures that the first integral in \eqref{e:existence_MA-B_basic_expression} is also finite. Therefore, \eqref{e:existence_MA-B_basic_expression} holds. This establishes the claim. $\Box$\\
%
%

\noindent {\sc Proof of Corollary \ref{cor:stationary_increments_scaling_property_MA-B}}: In regard to (a), both strict operator-stability and the scaling property \eqref{e:Bes_scaling} can be established by a simple adaptation of the proof of Corollary \ref{cor:stationary_increments_scaling_property_MA}, (a).

To prove (b), the argument is similar to that for showing Corollary \ref{cor:stationary_increments_scaling_property_MA}, (b). For the reader's convenience, we highlight the main steps. Again by Theorem 5.4, (c), in Kremer and Scheffler \cite{kremer:scheffler:2019}, after a change of variables it suffices to show that
\begin{equation}\label{e:int_psi-tilde->0_Bes}
\int_{\bbR^d} \widehat{\psi}\Big([\varrho_{H-qB,\lambda}({\bm x} - {\bm y}) \varphi({\bm x} - {\bm y})^{H-qB}- \varrho_{H-qB,\lambda}(- {\bm y}) \varphi(- {\bm y})^{H-qB}]^*u \Big) d{\bm y} \rightarrow 0, \quad {\bm x} \rightarrow 0,
\end{equation}
for any $u \in \bbR^n \backslash\{0\}$. So, for notational simplicity, let
\begin{equation}\label{e:fx(y)_Bes}
f^{Bes}_{{\bm x}}({\bm y}) = \varrho_{H-qB,\lambda}({\bm x} - {\bm y})  \varphi({\bm x} - {\bm y})^{H-qB}- \varrho_{H-qB,\lambda}(-{\bm y})  \varphi(- {\bm y})^{H-qB}, \quad {\bm x} \in \bbR^d.
\end{equation}
By the same argument as in the proof of Corollary \ref{cor:stationary_increments_scaling_property_MA}, (b), bound \eqref{e:psi-tilde(f_x(y))=<three_terms} also holds with the function $f^{Bes}_{{\bm x}}({\bm y})$ in place of $f_{{\bm x}}({\bm y})$. In other words, without loss of generality suppose condition \eqref{e:|x|=<1} is in place. For any small $\delta_1, \delta_2 >0$,
\begin{equation}\label{e:psi-tilde(f_x(y))=<three_terms_Bes}
\Big| \widehat{\psi}(f^{Bes}_{{\bm x}}({\bm y}))\Big|
\end{equation}
$$
\leq C 1_{\{{\bm y}:\|f^{Bes}_{{\bm x}}({\bm y})\| \leq 1 \}} \hspace{0.5mm} \Big(\|f^{Bes}_{{\bm x}}({\bm y})\| \hspace{0.5mm}\|{\mathbf u}\|\Big)^{\frac{1}{\Upsilon_B}- \delta_1}
+ \Big[C' \|{\mathbf u}\|^{\frac{1}{\Upsilon_B}-\delta_1}  + C'' \|{\mathbf u}\|^{\frac{1}{\varpi_B}+ \delta_2}\Big]
\hspace{0.5mm}\|f^{Bes}_{{\bm x}}({\bm y})\|^{\frac{1}{\varpi_B}+ \delta_2}1_{\{{\bm y}:\|f^{Bes}_{{\bm x}}({\bm y})\| > 1\}}.
$$
Analogously to \eqref{e:|f|^power*indicator_small_f_and_y}, consider the first sum term on the right-hand side of \eqref{e:psi-tilde(f_x(y))=<three_terms_Bes}, namely,
$$
C 1_{\{{\bm y}:\|f^{Bes}_{{\bm x}}({\bm y})\| \leq 1 \}} \hspace{0.5mm} \Big(\|f^{Bes}_{{\bm x}}({\bm y})\| \hspace{0.5mm}\|{\mathbf u}\|\Big)^{\frac{1}{\Upsilon_B}- \delta_1}
$$
Restricted to the range $\|{\bm y}\|\leq 1$,
$$
1_{\{{\bm y}:\|f^{Bes}_{{\bm x}}({\bm y})\| \leq 1 \cap \|{\bm y}\| \leq 1\}} \hspace{0.5mm} \Big(\|f^{Bes}_{{\bm x}}({\bm y})\| \hspace{0.5mm}\|{\mathbf u}\|\Big)^{\frac{1}{\Upsilon_B}- \delta_1}
$$
is bounded by the integrable function $C' \hspace{0.5mm}1_{\{{\bm y}: \hspace{0.5mm}\|{\bm y}\| \leq 1\}}$ for any small $\|{\bm x}\|$. Now consider the range
\begin{equation}\label{e:|y|>1}
\|{\bm y}\| > 1.
\end{equation}
By \eqref{e:(a+b)^p=<(Ca^p+Cb^p},
\begin{equation}\label{e:psi-tilde(f_x(y))=<three_terms_1st_term_Bes}
\|f^{Bes}_{{\bm x}}({\bm y})\|^{\frac{1}{\Upsilon_B}- \delta_1}
\leq C  \Big( \| \varrho_{H-qB,\lambda}({\bm x} - {\bm y})\varphi({\bm x} - {\bm y})^{H-qB}\|^{\frac{1}{\Upsilon_B}- \delta_1} + \| \varrho_{H-qB,\lambda}(- {\bm y})\varphi(- {\bm y})^{H-qB}\|^{\frac{1}{\Upsilon_B}- \delta_1} \Big)
\end{equation}
(for any ${\bm y}$). By condition \eqref{e:|x|=<1} and Lemma \ref{lem1.1}, for small $\delta > 0$ the bound \eqref{e:tau(x-y)>=|y-1|^(1/ap-delta-1)} holds. Therefore, under \eqref{e:|y|>1}, expression \eqref{e:|KJ+(x)|)} implies that, for some $\eta_* \in (0,1)$,
$$
\|K_{H-qB}(\lambda \varphi({\bm x}- {\bm y}))\varphi({\bm x}- {\bm y})^{H-qB}\| \leq C e^{- \varphi({\bm x}- {\bm y})\eta_*} \leq  C e^{-C'|\hspace{0.5mm} \|{\bm y}\|_0 - 1 |^{\frac{1}{\Upsilon_E}-\delta}}.
$$
A similar bound can be constructed for $\|K_{H-qB}(\lambda \varphi(- {\bm y}))\varphi(- {\bm y})^{H-qB}\|$. By \eqref{e:psi-tilde(f_x(y))=<three_terms_1st_term_Bes} restricted to $\|{\bm y}\| > 1$,
$$
\|f^{Bes}_{{\bm x}}({\bm y})\|^{\frac{1}{\Upsilon_B}-\delta_1}1_{\{{\bm y}: \|{\bm y}\|> 1\}}
$$
\begin{equation}\label{e:|fx(y)|^power*indic=<bound_Bes}
\leq C\Big( e^{-C'(\frac{1}{\Upsilon_B}-\delta_1)\hspace{0.5mm}|\hspace{0.5mm}\|{\bm y}\|_0 - 1|^{\frac{1}{\Upsilon_E}- \delta}}
+ e^{-C'' (\frac{1}{\Upsilon_B}-\delta_1)\hspace{0.5mm}\|{\bm y}\|^{\frac{1}{\Upsilon_B}-\delta_1}_0} \Big) 1_{\{{\bm y}:\|{\bm y}\|> 1\}}.
\end{equation}
We conclude that the first sum term on the right-hand side of \eqref{e:psi-tilde(f_x(y))=<three_terms_Bes} is bounded uniformly in ${\bm x}$ by a $d{\bm y}$--integrable function. Therefore,
$$
\int_{\bbR^d}\|f^{Bes}_{{\bm x}}({\bm y})\|^{\frac{1}{\Upsilon_B}- \delta_1}1_{\{{\bm y}:\|f_{{\bm x}}({\bm y})\| \leq 1\}} d{\bm y} \rightarrow 0, \quad {\bm x}\rightarrow 0.
$$

We now turn to the second sum term on the right-hand side of \eqref{e:psi-tilde(f_x(y))=<three_terms_Bes}, namely,
$$
\Big[C' \|{\mathbf u}\|^{\frac{1}{\Upsilon_B}-\delta_1}  + C'' \|{\mathbf u}\|^{\frac{1}{\varpi_B}+ \delta_2}\Big]
\hspace{0.5mm}\|f^{Bes}_{{\bm x}}({\bm y})\|^{\frac{1}{\varpi_B}+ \delta_2}1_{\{{\bm y}:\|f^{Bes}_{{\bm x}}({\bm y})\| > 1\}}.
$$
 For the range $\|{\bm y}\| > 1$, one can adapt the argument leading up to \eqref{e:|fx(y)|^power*indic=<bound_Bes}. So, we can assume $\|{\bm y}\| \leq 1$. In view of \eqref{e:psi-tilde(f_x(y))=<three_terms_1st_term_Bes} with $\frac{1}{\varpi_B}+ \delta_2$ in place of $\frac{1}{\Upsilon_B}- \delta_1$, it suffices to consider the function
$$
\|K_{H-qB}(\lambda \varphi({\bm x} - {\bm y}))\varphi({\bm x} - {\bm y})^{H-qB}\|^{\frac{1}{\varpi_B}+ \delta_2}1_{\{{\bm y}:\|f^{Bes}_{{\bm x}}({\bm y})\| > 1 \cap \|{\bm y}\| \leq 1\}}.
$$
Such function, in turn, is bounded by
$$
\|K_{H-qB}(\lambda \varphi({\bm x} - {\bm y}))\varphi({\bm x} - {\bm y})^{H-qB}\|^{\frac{1}{\varpi_B}+ \delta_2}1_{\{{\bm y}: \|{\bm y}\| \leq 1\}}
$$
\begin{equation}\label{e:K(lambda*varphi(x-y))varphi(x-y)^power_lim_x->0}
\rightarrow \|K_{H-qB}(\lambda \varphi(- {\bm y}))\varphi(- {\bm y})^{H-qB}\|^{\frac{1}{\varpi_B}+ \delta_2}1_{\{{\bm y}: \|{\bm y}\| \leq 1\}}, \quad {\bm x}\rightarrow 0,
\end{equation}
where all functions involved in expression \eqref{e:K(lambda*varphi(x-y))varphi(x-y)^power_lim_x->0} are $d{\bm y}$-integrable by \eqref{e:int_Bess-filter_2nd_term^1/lambdaB+delta_finite}. Note that, after a change of variable ${\bm z}={\bm x} - {\bm y}$,
$$
\int_{\bbR^d}\|K_{H-qB}(\lambda \varphi({\bm x} - {\bm y})) \varphi({\bm x} - {\bm y})^{H-qB}\|^{\frac{1}{\varpi_B}+ \delta_2}1_{\{{\bm y}: \|{\bm y}\| \leq 1\}}
d{\bm y}
$$
$$
=\int_{\bbR^d}\|K_{H-qB}(\lambda \varphi({\bm z})) \varphi({\bm z})^{H-qB}\|^{\frac{1}{\varpi_B}+ \delta_2}1_{\{{\bm z}: \|{\bm x}-{\bm z}\| \leq 1\}}
d{\bm z},
$$
where
$$
\|K_{H-qB}(\lambda \varphi({\bm z})) \varphi({\bm z})^{H-qB}\|^{\frac{1}{\varpi_B}+ \delta_2}1_{\{{\bm z}: \|{\bm x}-{\bm z}\| \leq 1\}} \leq
\|K_{H-qB}(\lambda \varphi({\bm z})) \varphi({\bm z})^{H-qB}\|^{\frac{1}{\varpi_B}+ \delta_2}1_{\{{\bm z}: \|{\bm z}\| \leq 2\}}, \quad \|{\bm x}\|\leq 1-\eta_0.
$$
Moreover,
$$
\int_{\bbR^d} \|K_{H-qB}(\lambda \varphi({\bm z})) \varphi({\bm z})^{H-qB}\|^{\frac{1}{\varpi_B}+ \delta_2}1_{\{{\bm z}: \|{\bm z}\| \leq 2\}} d{\bm z} < \infty
$$
by a simple adaptation of the bound \eqref{e:int01_MA-B_kernel} and by condition \eqref{e:MA-B_condition_eig_exponent}. Therefore, by the dominated convergence theorem, as ${\bm x}\rightarrow 0$,
$$
\int_{\bbR^d}\|K_{H-qB}(\lambda \varphi({\bm x} - {\bm y})) \varphi({\bm x} - {\bm y})^{H-qB}\|^{\frac{1}{\varpi_B}+ \delta_2}1_{\{{\bm y}: \|{\bm y}\| \leq 1\}}
d{\bm y}
$$
$$
\rightarrow
\int_{\bbR^d}\|K_{H-qB}(\lambda \varphi(- {\bm y})) \varphi(- {\bm y})^{H-qB}\|^{\frac{1}{\varpi_B}+ \delta_2}1_{\{{\bm y}: \|{\bm y}\| \leq 1\}}
d{\bm y}.
$$
Thus, again by the dominated convergence theorem,
$$
\int_{\bbR^d}\|f^{Bes}_{{\bm x}}({\bm y})\|^{\frac{1}{\varpi_B}+ \delta_2}1_{\{{\bm y}:\|f_{{\bm x}}({\bm y})\| > 1\}} d{\bm y} \rightarrow 0, \quad {\bm x}\rightarrow 0.
$$
This establishes (b).

By a change of variables in the characteristic function, the same argument in the proof of Corollary \ref{cor:stationary_increments_scaling_property_MA}, (c), shows that the increments of $X^{Bes}_\lambda$ are stationary. In other words, (c) holds.

To show (d), recall that, by Proposition 2.6, (a), in Kremer and Scheffler \cite{kremer:scheffler:2019}, for fixed ${\bm x} \in \bbR^d \backslash\{0\}$, it suffices to show that there is a positive Lebesgue measure set on which the kernel \eqref{e:fx(y)_Bes} has full rank. So, by way of contradiction, fix ${\bm x} \neq 0$ and suppose
\begin{equation}\label{e:Leb-d(Bess-kernel(x-y)_neq_kernel(-y))=0}
\textnormal{Leb}_d\{{\bm y}: K_{H-qB}(\lambda \varphi({\bm x}-{\bm y}))\hspace{1mm}\varphi({\bm x}-{\bm y})^{H-qB} \neq K_{H-qB}(\lambda \varphi(-{\bm y}))\hspace{1mm}\varphi(-{\bm y})^{H-qB}\} = 0.
\end{equation}
First consider the case where $H-qB \leq 0$. By continuity and relation \eqref{e:mod_Bessel_second_kind_asymptotics_large_u}, as ${\bm y}\rightarrow 0$,
$$
K_{H-qB}(\lambda \varphi({\bm x}-{\bm y}))\hspace{1mm}\varphi({\bm x}-{\bm y})^{H-qB} \rightarrow
K_{H-qB}(\lambda \varphi({\bm x}))\hspace{1mm}\varphi({\bm x})^{H-qB}, \quad \|K_{H-qB}(\lambda \varphi(-{\bm y}))\hspace{1mm}\varphi(-{\bm y})^{H-qB}\| \rightarrow \infty.
$$
Again by continuity, this contradicts \eqref{e:Leb-d(Bess-kernel(x-y)_neq_kernel(-y))=0}. Now consider the alternative case where $H-qB > 0$. By \eqref{e:Leb-d(Bess-kernel(x-y)_neq_kernel(-y))=0},
\begin{equation}\label{e:K*varphi=K*varphi}
K_{H-qB}(\lambda \varphi({\bm x}-{\bm y}))\varphi({\bm x}-{\bm y})^{H-qB} = K_{H-qB}(\lambda \varphi(-{\bm y}))\varphi(-{\bm y})^{H-qB} \quad \textnormal{d${\bm y}$--a.e.}
\end{equation}
By continuity, equality \eqref{e:K*varphi=K*varphi} holds for ${\bm y} \in \bbR^d \backslash\{0,{\bm x}\}$. So, pick ${\bm y}$ so that $\{k {\bm x} - {\bm y}\}_{k \in \bbN} \subseteq \bbR^d \backslash\{0\}$. By induction, relation \eqref{e:K*varphi=K*varphi} implies that
$$
K_{H-qB}(\lambda \varphi(k{\bm x}-{\bm y}))\hspace{1mm}\varphi(k{\bm x}-{\bm y})^{H-qB} = K_{H-qB}(\lambda \varphi(-{\bm y}))\hspace{1mm}\varphi(-{\bm y})^{H-qB}, \quad k \in \bbN.
$$
However, $K_{H-qB}(\lambda \varphi(k{\bm x}-{\bm y}))\hspace{1mm}\varphi(k{\bm x}-{\bm y})^{H-qB} \rightarrow 0$, as $k \rightarrow \infty$, which is a contradiction. Therefore, (d) holds.

We now turn to (e). Since $E = I$, then by \eqref{e:|.|0_explicit} without loss of generality we can assume $\|\cdot\|_0 = \|\cdot\|$, i.e., the Euclidean norm. So, fix ${\bm x} \neq 0$. We want to show that
\begin{equation}\label{e:Leb-d(Bess-kernel(x-y)_neq_kernel(-y))=0_E=I}
\textnormal{Leb}_d \Big\{{\bm y}: \det\Big[K_{H-qB}(\lambda \varphi({\bm x}-{\bm y}))\hspace{1mm}\varphi({\bm x}-{\bm y})^{H-qB} - K_{H-qB}(\lambda \varphi(-{\bm y}))\hspace{1mm}\varphi(-{\bm y})^{H-qB} \Big] \neq 0\Big\} > 0.
\end{equation}
Again without loss of generality, we only consider the case where the Jordan decomposition of $H$ satisfies \eqref{e:H-qB=Pdiag(J-,J.,J+)P^(-1)}. Then, we can write
$$
\det\Big[K_{H-qB}(\lambda \varphi({\bm x}-{\bm y}))\hspace{1mm}\varphi({\bm x}-{\bm y})^{H-qB} - K_{H-qB}(\lambda \varphi(-{\bm y}))\hspace{1mm}\varphi(-{\bm y})^{H-qB}\Big]
$$
$$
= \prod_{J=J_\ominus,J_\odot,J_\oplus}\det\Big(K_{J}(\lambda \varphi({\bm x}-{\bm y}))\hspace{1mm}\varphi({\bm x}-{\bm y})^{J} - K_{J}(\lambda \varphi(-{\bm y}))\hspace{1mm}\varphi(-{\bm y})^{J}\Big).
$$
Fix $\nu \in \textnormal{eig}(H-qB) \subseteq \bbR^n$. Let ${\bm y}(r) := r \frac{{\bm x}}{\|{\bm x}\|}$. As $r \rightarrow \infty$, by \eqref{e:mod_Bessel_second_kind_asymptotics_large_u},
$$
\frac{K_\nu\Big(\lambda \varphi({\bm x}- {\bm y}(r))\Big) \varphi({\bm x}- {\bm y}(r))^{\nu}}{K_\nu\Big(\lambda \varphi(- {\bm y}(r))\Big)\varphi(- {\bm y}(r))^{\nu}} =
\frac{K_\nu\Big(\lambda \varphi({\bm x}-r \frac{{\bm x}}{\|{\bm x}\|})\Big) \varphi({\bm x}-r \frac{{\bm x}}{\|{\bm x}\|})^{\nu}}{K_\nu\Big(\lambda \varphi(- r \frac{{\bm x}}{\|{\bm x}\|})\Big)\varphi(- r \frac{{\bm x}}{\|{\bm x}\|})^{\nu}}
$$
$$
=
\frac{K_\nu\Big(\lambda (r- \|{\bm x}\|)\varphi(-\frac{{\bm x}}{\|{\bm x}\|})\Big) (r- \|{\bm x}\|)^{\nu} }{K_\nu\Big(\lambda r \varphi(-\frac{{\bm x}}{\|{\bm x}\|})\Big) r^{\nu} }
\sim
\frac{K_\nu\Big(\lambda (r- \|{\bm x}\|)\varphi(-\frac{{\bm x}}{\|{\bm x}\|})\Big) }{K_\nu\Big(\lambda r\varphi(-\frac{{\bm x}}{\|{\bm x}\|})\Big) }
$$
$$
\sim \frac{\frac{1}{\sqrt{\|(r- \|{\bm x}\|)\varphi(-\frac{{\bm x}}{\|{\bm x}\|})\|}}e^{- (r- \|{\bm x}\|)\varphi(-\frac{{\bm x}}{\|{\bm x}\|})}}{\frac{1}{\sqrt{\|r\varphi(-\frac{{\bm x}}{\|{\bm x}\|})\|}}e^{- r \varphi(-\frac{{\bm x}}{\|{\bm x}\|})}}
\sim e^{ \|{\bm x}\|\varphi(-\frac{{\bm x}}{\|{\bm x}\|})}.
$$
Since $\|{\bm x}\|\varphi(-\frac{{\bm x}}{\|{\bm x}\|}) > 0$, then, for large enough $r$,
\begin{equation}\label{e:K*varphi-K*varphi_neq_0}
K_\nu(\lambda \varphi({\bm x}-{\bm y}(r)))\varphi({\bm x}-{\bm y}(r))^{\nu} - K_\nu(\lambda \varphi(-{\bm y}(r)))\varphi(-{\bm y}(r))^{\nu}\neq 0.
\end{equation}
Let $\sigma(d\theta)$ be the uniform surface measure on the (Euclidean) sphere. By the continuity of $K_\nu$, there exists a region of the form
\begin{equation}\label{e:A(r1,r2,Theta)}
A((r_1,r_2),\Theta) = \Big\{{\bm y}: r_1 < \|{\bm y}\| < r_2, \frac{{\bm y}}{\|{\bm y}\|} \in \Theta \Big\}, \quad \sigma(\Theta) > 0,
\end{equation}
for some Borel set $\Theta$, over which
\begin{equation}\label{e:K*varphi-K*varphi_neq_0_any_y}
K_\nu(\lambda \varphi({\bm x}-{\bm y}))\varphi({\bm x}-{\bm y})^{\nu} - K_\nu(\lambda \varphi(-{\bm y}))\varphi(-{\bm y})^{\nu}\neq 0.
\end{equation}
Since $\textnormal{card}(\textnormal{eig}(H-qB))< \infty$, then we can assume  \eqref{e:K*varphi-K*varphi_neq_0_any_y} holds over ${\bm y} \in A((r_1,r_2),\Theta)$ for all $\nu \in \textnormal{eig}(H-qB)$. This shows \eqref{e:Leb-d(Bess-kernel(x-y)_neq_kernel(-y))=0_E=I} and, thus, (e). $\Box$\\

\section{Section \ref{sec:frequency_domain}: proofs}

\noindent {\sc Proof of Theorem \ref{thm:General_harmo_TOSSRF}}:
For ${\bm x},{\bm \xi} \in \bbR^d$, let
$$
\widetilde{f}_{{\bm x}}(\xi) =
\begin{pmatrix}
(\cos \langle {\bm x},{\bm \xi}\rangle-1)( \lambda + \varphi({\bm \xi}))^{-H} ( \lambda + \varphi({\bm \xi}))^{-qB} &
-\sin \langle {\bm x},{\bm \xi}\rangle ( \lambda + \varphi({\bm \xi}))^{-H} ( \lambda + \varphi({\bm \xi}))^{-qB}\\
0 & 0
\end{pmatrix}.
$$
In particular, $\widetilde{f}_{{\bm x}}(\xi) \in {\mathcal M}(2n,\bbR)$. Note that, for fixed $\mathbf{u}^t = (\mathbf{u}^t_1, \mathbf{u}^t_2) \in \bbR^{2n}$,
\begin{equation}\label{e:Xi(f*u1)=f-tilde*u}
\Xi(f_{{\bm x}}({\bm \xi})^* \mathbf{u}_1) = \widetilde{f}_{{\bm x}}({\bm \xi})^* \mathbf{u}.
\end{equation}
So, bearing in mind relation \eqref{e:chf_stoch_integral_partially_integ}, by Kremer and Scheffler \cite{kremer:scheffler:2017}, Proposition 5.10, it suffices to show that $\widetilde{f}_{{\bm x}}({\bm \xi})$ is integrable with respect to $M = \Xi(\widetilde{M})$. In other words, we want to show that
\begin{equation}\label{e:harm_exist_basic_integ_relation}
\int_{\bbR^d} |\widehat{\psi}_{\Xi(M)}(\widetilde{f}_{{\bm x}}({\bm \xi})^* \mathbf{u} )|d{\bm \xi} < \infty, \quad \mathbf{u}^t = (\mathbf{u}^t_1,\mathbf{u}^t_2)\in \bbR^{2n}.
\end{equation}
For
\begin{equation}\label{e:0<delta<Lambda^(-1)=<lambda^(-1)}
0 < \delta < \Upsilon^{-1}_B \leq \varpi^{-1}_B,
\end{equation}
set
\begin{equation}\label{e:rho1,rho2}
\rho_1 = \Upsilon^{-1}_B - \delta, \quad \rho_2 = \varpi^{-1}_B + \delta.
\end{equation}
The following two bounds (i.e., \eqref{e:psi-hat(x)=<C(|x|^(rho1)+|x|^(rho2))} and \eqref{e:harm_exist_Euclidean_bound}) can be established by arguing as in Kremer and Scheffler \cite{kremer:scheffler:2019}, p.\ 18--19; see that reference for more details. The first bound pertains to the log-characteristic function. It states that, for some $C > 0$,
\begin{equation}\label{e:psi-hat(x)=<C(|x|^(rho1)+|x|^(rho2))}
|\widehat{\psi}_{\Xi(M)}({\bm x})| \leq C (\|{\bm x}\|^{\rho_1} + \|{\bm x}\|^{\rho_2}), \quad {\bm x} \in \bbR^d.
\end{equation}
Starting from \eqref{e:psi-hat(x)=<C(|x|^(rho1)+|x|^(rho2))}, one can develop the second bound, which involves the integral in \eqref{e:harm_exist_basic_integ_relation}. In fact, for ${\mathbf u}^t = ({\mathbf u}^t_1,{\mathbf u}^t_2)$,
$$
\int_{\bbR^d} |\widehat{\psi}_{\Xi(M)}(\widetilde{f}_{{\bm x}}({\bm \xi})^* {\mathbf u} )|d{\bm \xi}
\leq C \sum^{2}_{j=1} \int_{\bbR^d} ( \lambda + \varphi({\bm \xi}))^{-q}
\left\|\begin{pmatrix}
(\cos \langle {\bm x},{\bm \xi}\rangle-1)( \lambda + \varphi({\bm \xi}))^{-H^*}{\mathbf u}_1 \\
-\sin \langle {\bm x},{\bm \xi}\rangle ( \lambda + \varphi({\bm \xi}))^{-H^*}{\mathbf u}_1
\end{pmatrix}\right\|^{\rho_j}d{\bm \xi}
$$
\begin{equation}\label{e:harm_exist_Euclidean_bound}
\leq C' \sum^{2}_{j=1} \|{\mathbf u}_1\|^{\rho_j} \int_{\bbR^d} ( \lambda + \varphi({\bm \xi}))^{-q}\hspace{0.5mm}
\Big(|\cos \langle {\bm x},{\bm \xi}\rangle-1|^{\rho_j} + |\sin \langle {\bm x},{\bm \xi}\rangle|^{\rho_j} \Big)
\hspace{1mm}\|( \lambda + \varphi({\bm \xi}))^{-H}\|^{\rho_j} \hspace{0.5mm}d{\bm \xi}.
\end{equation}
Now, by a change of variable into polar coordinates induced by $E$ (Lemma \ref{l:polar_integration}), for $j = 1,2$ we can bound the integral in the sum on the right-hand side of \eqref{e:harm_exist_Euclidean_bound} by
$$
\int^{\infty}_0 \int_{S_0} (\lambda+ r\varphi(\theta))^{-q}\|(\lambda + r \varphi(\theta))^{-H}\|^{\rho_j} r^{q-1}\sigma(d \theta) d r
$$
$$
\leq C \sigma(S_0) \int^{\infty}_0  (\lambda+ r C')^{-q}\hspace{0.5mm}\|(\lambda + r C')^{-H}\|^{\rho_j} \hspace{0.5mm}r^{q-1} d r
$$
\begin{equation}\label{e:harm_exist_polar_bound}
= C'' \int^{\infty}_0 (1 + s)^{-q} \|(1 + s)^{-H}\|^{\rho_j} s^{q-1} ds.
\end{equation}
In \eqref{e:harm_exist_polar_bound}, we make the change of variable $s = rC'$. The integral on the right-hand side of \eqref{e:harm_exist_polar_bound} is finite if $q + \varpi_{H} \rho_j - q + 1 > 1$. Equivalently, in view of \eqref{e:rho1,rho2}, it is finite when $\varpi_{H} (\Upsilon^{-1}_B - \delta) > 0$ ($j=1$) or $\varpi_{H} (\varpi^{-1}_B + \delta) > 0$ ($j=2$). In view of \eqref{e:0<delta<Lambda^(-1)=<lambda^(-1)}, both conditions are met. Therefore, \eqref{e:harm_exist_basic_integ_relation} holds, whence $\widetilde{X}_{\lambda} ({\bm x})$ exists, as claimed. $\Box$\\

\noindent {\sc Proof of Corollary \ref{c:stationary_increments_scaling_property_H}}: We begin by showing (a). For any $k \in \bbN$, fix ${\bm x}_1,\hdots,{\bm x}_k \in \bbR^d$. Also, fix $c > 0$. Then, in view of relations \eqref{e:chf_stoch_integral_partially_integ} and \eqref{e:Xi(f*u1)=f-tilde*u}, the characteristic function of $\widetilde{X}_{\lambda}(c^{E}{\bm x}_1),\hdots,\widetilde{X}_{\lambda}(c^{E}{\bm x}_k)$ is given by
$$
{\mathbf u}_1,\hdots,{\mathbf u}_k \in \bbR^n \mapsto \exp \int_{\bbR^d} \widehat{\psi}_{\Xi(M)}
\begin{pmatrix}
\sum^{k}_{j=1}(\cos \langle c^{E}{\bm x}_j, {\bm \xi}\rangle-1) (\lambda + \varphi({\bm \xi}))^{-H^* - qB^*}{\mathbf u}_j \\
- \sum^{k}_{j=1}\sin \langle c^{E}{\bm x}_j, {\bm \xi}\rangle (\lambda + \varphi({\bm \xi}))^{-H^* - qB^*}{\mathbf u}_j
\end{pmatrix} d {\bm \xi}.
$$
Since $\widetilde{B}^* = B^* \oplus B^*$ and $\varphi$ is $E^*$-homogeneous, by making a change of variable ${\bm \zeta} = c^{E^*}{\bm \xi}$, and by applying the commutativity condition \eqref{e:HB=BH} as well as property \eqref{e:s.psi-hat(u)=psi-hat(s^(B^*)u)}, we obtain
$$
\exp \int_{\bbR^d} c^{-q}\widehat{\psi}_{\Xi(M)}
\begin{pmatrix}
\sum^{k}_{j=1}(\cos \langle {\bm x}_j, {\bm \zeta}\rangle-1) c^{H^*+qB^*}((c\lambda) + \varphi({\bm \zeta}))^{-H^* - qB^*}{\mathbf u}_j \\
- \sum^{k}_{j=1}\sin \langle {\bm x}_j, {\bm \zeta}\rangle c^{H^*+qB^*} ((c\lambda) + \varphi({\bm \zeta}))^{-H^* - qB^*}{\mathbf u}_j
\end{pmatrix} d {\bm \zeta}
$$
$$
= \exp \int_{\bbR^d} \widehat{\psi}_{\Xi(M)}
\begin{pmatrix}
\sum^{k}_{j=1}(\cos \langle {\bm x}_j, {\bm \zeta}\rangle-1) c^{H^*}((c\lambda) + \varphi({\bm \zeta}))^{-H^* - qB^*}{\mathbf u}_j \\
- \sum^{k}_{j=1}\sin \langle {\bm x}_j, {\bm \zeta}\rangle c^{H^*} ((c\lambda) + \varphi({\bm \zeta}))^{-H^* - qB^*}{\mathbf u}_j
\end{pmatrix} d {\bm \zeta}.
$$
A similar argument with characteristic functions further shows that $\widetilde{X}_{\lambda}$ is strictly operator-stable with exponent $B$. This establishes (a).

Before showing (b), we show (c). For $z \in \bbC \backslash\{0\}$, consider the rotation matrices
\begin{equation}\label{e:R(z)}
R(z)= \frac{1}{|z|}
\begin{pmatrix}
(\Re z)I_n & (\Im z)I_n \\
-(\Im z)I_n & (\Re z)I_n
\end{pmatrix} \in {\mathcal T}(2n).
\end{equation}
For ${\bm x}, {\bm h} \in \bbR^d$, the integral representation of the increment of $\widetilde{X}_{\lambda}$ is given by
$$
\widetilde{X}_{\lambda}({\bm x}+{\bm h}) - \widetilde{X}_{\lambda}({\bm h})
= \Re \int_{\bbR^d} e^{\imag \langle {\bm x},{\bm h}\rangle} (e^{\imag \langle {\bm x},{\bm \xi} \rangle}-1) (\lambda + \varphi({\bm \xi}))^{-H}(\lambda + \varphi({\bm \xi}))^{-qB} \widetilde{M}(d {\bm {\bm \xi}}) \quad \textnormal{a.s.}
$$
Hence, for ${\mathbf u} \in \bbR^n$,
$$
\bbE \exp(\imag \hspace{0.5mm}{\mathbf u}^t(\widetilde{X}_{\lambda}({\bm x}+{\bm h}) -\widetilde{X}_{\lambda}({\bm h}))) =  \exp \int_{\bbR^d} \widehat{\psi}_{\Xi(M)}(\zeta(h,{\bm x},{\bm \xi},{\mathbf u})) d{\bm \xi},
$$
where, for $R(\cdot)$ as in \eqref{e:R(z)},
\begin{equation}\label{e:zeta(h,x,xi,u)}
\zeta({\bm h},{\bm x},{\bm \xi},{\mathbf u})
:= R(e^{\imag \langle -{\bm h},{\bm x}\rangle})^* \begin{pmatrix}
(\cos \langle {\bm x},{\bm \xi}\rangle-1) (\lambda + \psi({\bm \xi}))^{-qB^*}(\lambda + \psi({\bm \xi}))^{-H^*}{\mathbf u}\\
-\sin \langle {\bm x},{\bm \xi}\rangle  (\lambda + \psi({\bm \xi}))^{-qB^*}(\lambda + \psi({\bm \xi}))^{-H^*}{\mathbf u}
\end{pmatrix}.
\end{equation}
By condition \eqref{e:Amu(dx)=mu(dx)}, $\widehat{\psi}_{\Xi(M)}(A^* {\bm x}) = \widehat{\psi}_{\Xi(M)}({\bm x})$ for $A \in {\mathcal T}(2n)$, ${\bm x} \in \bbR^{2n}$.  Therefore,
$$
\widehat{\psi}_{\Xi(M)}(\zeta({\bm h},{\bm x},{\bm \xi},{\mathbf u})) = \widehat{\psi}_{\Xi(M)}(\zeta(0,{\bm x},{\bm \xi},{\mathbf u})).
$$
The claim now follows by adapting the proof of Corollary \ref{cor:stationary_increments_scaling_property_MA}, (c).

We now turn to (b). Since $\widehat{\psi}_{\Xi(M)}(0) = 0$ and $\widehat{\psi}_{\Xi(M)}$ is continuous, by both Theorem 5.4 and Remark 5.12 in Kremer and Scheffler \cite{kremer:scheffler:2017}, it suffices to show that, for $\mathbf{u} \in \bbR^n$,
$$
\int_{\bbR^d} \widehat{\psi}_{\Xi(M)}(\zeta({\bm x}_0,{\bm x},{\bm \xi},{\mathbf u})) d {\bm \xi} \rightarrow 0, \quad {\bm x} \rightarrow {\bm x}_0,
$$
where $\zeta({\bm x}_0,{\bm x},{\bm \xi},{\mathbf u})$ is as in \eqref{e:zeta(h,x,xi,u)}. Note that
$$
\widehat{\psi}_{\Xi(M)}(\zeta({\bm x}_0,{\bm x},{\bm \xi},{\mathbf u})) =
(\lambda + \varphi({\bm \xi}))^{-q}
\widehat{\psi}_{\Xi(M)} \Bigg(R(e^{\imag \langle {\bm x}_0,{\bm \xi} \rangle})
\begin{pmatrix}
(\cos \langle {\bm x},{\bm \xi}\rangle-1)(\lambda + \varphi({\bm \xi}))^{-H^*}{\mathbf u} \\
-\sin \langle {\bm x},{\bm \xi}\rangle (\lambda + \varphi({\bm \xi}))^{-H^*}{\mathbf u}
\end{pmatrix}\Bigg).
$$
Thus, by relation \eqref{e:harm_exist_Euclidean_bound},
$$
\int_{\bbR^d}|\widehat{\psi}_{\Xi(M)}(\zeta({\bm x}_0,{\bm x},{\bm \xi},{\mathbf u}))|d {\bm \xi}
$$
$$
\leq C \sum^{2}_{j=1} \int_{\bbR^d} (\lambda + \varphi({\bm \xi}))^{-q}
\Bigg\| R(e^{\imag \langle {\bm x}_0,{\bm \xi} \rangle})
\begin{pmatrix}
(\cos \langle {\bm x},{\bm \xi}\rangle-1)(\lambda + \varphi({\bm \xi}))^{-H^*}{\mathbf u} \\
-\sin \langle {\bm x},{\bm \xi}\rangle (\lambda + \varphi({\bm \xi}))^{-H^*}{\mathbf u}
\end{pmatrix}\Bigg\|^{\rho_j} d {\bm \xi}
$$
\begin{equation}\label{e:harm_stoch_cont_chf_bound}
\leq C' \sum^{2}_{j=1} \|{\mathbf u}\|^{\rho_j}\int_{\bbR^d} (\lambda + \varphi({\bm \xi}))^{-q} \Big(|\cos \langle {\bm x},{\bm \xi}\rangle-1)\|^{\rho_j}
+ |\sin \langle {\bm x},{\bm \xi}\rangle|^{\rho_j} \Big) \hspace{1mm}\|(\lambda + \varphi({\bm \xi}))^{-H^*}\|^{\rho_j} d {\bm \xi} .
\end{equation}
However, as in the proof of Theorem \ref{thm:General_harmo_TOSSRF} (see \eqref{e:harm_exist_polar_bound}), the right-hand side of \eqref{e:harm_stoch_cont_chf_bound} is bounded uniformly in ${\bm x}$ by an integrable function. Thus, by the dominated convergence theorem, (b) holds.

We now turn to (d). Let $f: \bbR^d \rightarrow {\mathcal M}(n,\bbC)$ be a measurable function. By Kremer and Scheffler \cite{kremer:scheffler:2019}, Corollary 2.7, if
$$
\det | \Re f({\bm \xi})| + \det | \Im f({\bm \xi})| > 0
$$
over a set with positive Lebesgue measure, then $\Re \int_{\bbR^d} f({\bm \xi})\widetilde{M}({\bm \xi})$ is full. So, recast the kernel of $\widetilde{X}_{\lambda}$ as
$$
\Re \Big\{ (e^{\imag \langle {\bm x},{\bm \xi} \rangle}-1) (\lambda + \varphi({\bm \xi}))^{-H}(\lambda + \varphi({\bm \xi}))^{-qB} \Big\}
$$
\begin{equation}\label{e:Re(ker(harm-RF))}
=  (\cos \langle {\bm x},{\bm \xi} \rangle -1 ) (\lambda + \varphi({\bm \xi}))^{-H}(\lambda + \varphi({\bm \xi}))^{-qB}.
\end{equation}
The determinant of the expression on the right-hand side of \eqref{e:Re(ker(harm-RF))} is positive outside the set $\{{\bm \xi}: \cos \langle {\bm x}, {\bm \xi}\rangle = 1\} $ for ${\bm x} \neq 0$. This shows (d). $\Box$\\

\section{Section \ref{s:Gaussian}: proofs}

\noindent {\sc Proof of Corollary \ref{cor100}}: Claim $(i)$ is a consequence of the fact that $\|O\bm{x}\| = \|\bm{x}\|$ for any $O \in O(d)$. So, we now turn to claim $(ii)$. As a consequence of the scaling property \eqref{e:scaling_H-TRF},
\begin{equation}\label{secondmomentI}
\displaystyle{ \mathbb{E} \left[ B_{H, \lambda} ( \bm{x} )B_{H, \lambda} ( \bm{x} )^t \right] }
= \| \bm{x} \| ^{H }\mathbb{E} \left[ B_{H, \| \bm{x} \| \lambda} \left( \frac{ \bm{x} } {\| \bm{x} \|} \right)B_{H, \| \bm{x} \| \lambda} \left( \frac{ \bm{x} } {\| \bm{x} \|} \right)^{t} \right]\| \bm{x} \| ^{H^t }.
\end{equation}
Moreover, by claim $(i)$, $\{B_{H, \lambda} ( -I \bm{x} )\}_{{{\bm x} \in \bbR^d}} \stackrel{f.d.}= \{B_{H, \lambda} ( \bm{x} )\}_{{{\bm x} \in \bbR^d}}$. Therefore, by an adaptation of the argument in Didier and Pipiras \cite{didier:pipiras:2011}, Proposition 5.1, $\bbE B_{H, \lambda} (\bm{x} )B_{H, \lambda} (\bm{y} )^t = \bbE B_{H, \lambda} (\bm{y} )B_{H, \lambda} (\bm{x} )^t$. Thus, by the stationary increments of $B_{H, \lambda}$, for all $\bm{x}, \bm{x^{\prime}} \in \mathbb{R}^{d}$,
$$
\textnormal{Cov} [ B_{H, \lambda} (\bm{x}), B_{H, \lambda} (\bm{x^{\prime}}) ]
$$
$$
= \frac{1}{2} \Big\{  \| \bm{x} \|^{H} \mathbb{E} \left[ B_{H, \| \bm{x} \| \lambda} \Big( \frac{\bm{x}}{\|\bm{x}\|} \Big)B_{H, \| \bm{x} \| \lambda} \Big( \frac{\bm{x}}{\|\bm{x}\|} \Big)^{t} \right] | \bm{x} \|^{H^t} + \| \bm{x^{\prime}} \|^{H}\mathbb{E} \left[ B_{H, \| \bm{x'} \| \lambda} \Big(\frac{\bm{x^{\prime}}}{\|\bm{x^{\prime}}\|} \Big)B_{H, \| \bm{x^{\prime}} \| \lambda} \Big(\frac{\bm{x^{\prime}}}{\|\bm{x^{\prime}}\|} \Big)^{t} \right]\|\bm{x^{\prime}} \|^{H^t}
$$
$$
- \| \bm{x-x^{\prime}} \|^{H} \mathbb{E} \left[ B_{H, \| \bm{x-x^{\prime}} \| \lambda} \Big(\frac{\bm{x-x^{\prime}} }{\|\bm{x-x^{\prime}}\| }\Big)B_{H, \| \bm{x-x'} \| \lambda} \Big(\frac{\bm{x-x^{\prime}} }{\|\bm{x-x^{\prime}}\| }\Big)^{t} \right] \|\bm{x-x^{\prime}} \|^{H^t} \Big\}.
$$
Moreover, by claim $(i)$, relation \eqref{e:Cx} holds. This shows \eqref{covariance TFGFI}. $\Box$\\

In the proof of Proposition \ref{p:TFGFI harmo}, which we provide next,
\begin{equation}\label{e:F(f)(xi)}
{\mathcal F}(f)({\bm \xi})
\end{equation}
denotes the Fourier transform of a function $f$ at the point ${\bm \xi}$. \\

\noindent {\sc Proof of Proposition \ref{p:TFGFI harmo}}:  We establish \eqref{TFGFI harmo} by computing the Fourier transform of the generic kernel function component
\begin{equation}\label{kernelTFGFI}
\bbR \ni g_{h, \lambda, \bm{t}} (\bm{x})  \coloneqq
e^{- \lambda \| \bm{t}-\bm{x} \|} \| \bm{t}-\bm{x} \| ^ {h-\frac{d}{2}} -
e^{- \lambda \| -\bm{x} \|} \| -\bm{x} \| ^ {h-\frac{d}{2}}, \quad h \in (0,1).
\end{equation}
In fact,
\begin{eqnarray*}
\mathcal{F}\Big[\hat{g}_{h, \lambda, \bm{x}} \Big](\bm{\xi})
&=& \frac {1} {\sqrt{2\pi}} \int_{\mathbb{R}^{d}} e^{- \imag \langle \bm{\xi},\bm{y}\rangle}
\left[ e^{- \lambda \| \bm{x}-\bm{y} \|} \| \bm{x}-\bm{y} \| ^ {h-\frac{d}{2}} -
e^{- \lambda \| -\bm{y} \|} \| -\bm{y} \| ^ {h-\frac{d}{2}}  \right] d\bm{y} \\
&=& \frac {1} {\sqrt{2\pi}} \left[ \int_{\mathbb{R}^{d}} e^{- \imag \langle\bm{\xi},\bm{y}\rangle
- \lambda \| \bm{x}-\bm{y} \|} \| \bm{x}-\bm{y} \| ^ {h-\frac{d}{2}} d\bm{y}
- \int_{\mathbb{R}^{d}} e^{- \imag \langle\bm{\xi},\bm{y}\rangle
- \lambda \| -\bm{y} \|} \| -\bm{y} \| ^ {h-\frac{d}{2}} d\bm{y} \right].
\end{eqnarray*}
After a change of variable $-\bm{y^{\prime}} = \bm{x}-\bm{y}$, we can rewrite the first term on the right-hand side as
\begin{eqnarray*}
\int_{\mathbb{R}^{d}} e^{- \imag \langle\bm{\xi},\bm{y}\rangle
- \lambda \| \bm{x}-\bm{y} \|} \| \bm{x}-\bm{y} \| ^ {h-\frac{d}{2}} d\bm{y}
&=& \int_{\mathbb{R}^{d}} e^{- \imag \langle\bm{\xi},\bm{x} + \bm{y^{\prime}}\rangle
- \lambda \| -\bm{y^{\prime}} \|} \| -\bm{y^{\prime}} \| ^ {h-\frac{d}{2}}\ d\bm{y^{\prime}} \\
&=& e^{- \imag \langle\bm{\xi},\bm{x}\rangle} \int_{\mathbb{R}^{d}} e^{- \imag \langle\bm{\xi},\bm{y^{\prime}}\rangle
- \lambda \| \bm{y^{\prime}} \|}\ \| \bm{y^{\prime}} \| ^ {h-\frac{d}{2}}\ d\bm{y^{\prime}}.
\end{eqnarray*}
Hence,
\begin{eqnarray*}
\mathcal{F}\Big[\hat{g}_{h, \lambda, \bm{x}} \Big](\bm{\xi})
&=& \frac {e^{- \imag \langle\bm{\xi},\bm{x}\rangle} - 1} {\sqrt{2\pi}} \int_{\mathbb{R}^{d}} e^{- \imag \langle\bm{\xi},\bm{y}\rangle
- \lambda \| \bm{y} \|}\ \| \bm{y} \| ^ {h-\frac{d}{2}} d\bm{y}.
\end{eqnarray*}
Now, by changing from Euclidean to polar coordinates in dimension $d$, we obtain
\begin{eqnarray}\label{Fg-hat_polar}
\mathcal{F}\Big[\hat{g}_{h, \lambda, \bm{x}} \Big](\bm{\xi})
= \frac {e^{- \imag \langle\bm{\xi},\bm{x}\rangle} - 1} {\sqrt{2\pi}}
\int_{0}^{\infty} e^{-\lambda r} r^{h-\frac{d}{2}}
\left[ \int_{S_{r}} e^{- \imag \langle\bm{\xi},\bm{y}\rangle} dS \right] dr.
\end{eqnarray}
In \eqref{Fg-hat_polar}, $S_{r}$ is the surface of an $d$-dimensional unit ball of radius $r$, $dS$ is a differential surface element of the $d$-dimensional unit ball, and the inner integral is evaluated over the surface $S_{r}$.
By using the standard formula (see Reed et al.\ \cite{reed:lee:truong:1995}, p.1445), the latter is given by
\begin{eqnarray}\label{e:Bessel_shows_up}
\displaystyle{ \int_{S_{r}} e^{- \imag \langle\bm{\xi},\bm{y}\rangle} dS
= \left( \frac{ 2 \pi r } { \| \bm{\xi} \| } \right)^{ \frac{d} {2} } \| \bm{\xi} \| J_{ \frac{d-2} {2} } (\| \bm{\xi} \|r) },
\end{eqnarray}
where $J_{\nu}(\cdot)$ is Bessel or a modified Bessel function of the first kind. Thus,
\begin{eqnarray*}
\mathcal{F}\Big[\hat{g}_{h, \lambda, \bm{x}} \Big](\bm{\xi})
&=& \frac {e^{- \imag \langle\bm{\xi},\bm{x}\rangle} - 1} {\sqrt{2\pi}}
\int_{0}^{\infty} e^{-\lambda r} r^{h-\frac{d}{2}}
\left( \frac{ 2 \pi r } { \| \bm{\xi} \| } \right)^{ \frac{d} {2} } \| \bm{\xi} \| J_{ \frac{d-2} {2} } (\| \bm{\xi} \|r) dr \\
&=& \frac {e^{- \imag \langle\bm{\xi},\bm{x}\rangle} - 1} {\sqrt{2\pi}}
\frac{ (2 \pi)^{ \frac{d} {2} } } { \| \bm{\xi} \|^{ \frac{d} {2} - 1 } }
\int_{0}^{\infty} e^{-\lambda r} r^{h} J_{ \frac{d-2} {2} } (\| \bm{\xi} \|r) dr.
\end{eqnarray*}
Note that
\begin{eqnarray*}
\displaystyle
{ \int_{0}^{\infty} e^{ - \alpha x} J_{\nu} \left(\beta x \right) x ^ {\mu-1} dr }
= \frac{ \left( \frac{\beta} {2 \alpha} \right)^{\nu} \Gamma(\nu + \mu) }
{\alpha^{\mu} \Gamma(\nu+1) }
{}_2 F_1 \left( \frac{\nu+\mu} {2}, \frac{\nu + \mu + 1} {2}; \nu + 1; - \frac{\beta^2} {\alpha^2}
\right),
\end{eqnarray*}
where $ {}_2 F_1 $ is a Gaussian hypergeometric function given by \eqref{Fhyeorgeometric} (see Gradshteyn and Ryzhik \cite{gradshteyn:ryzhik:1994}, p.1088). Then,
\begin{equation*}
\begin{split}
\mathcal{F}\Big[\hat{g}_{h, \lambda, \bm{x}} \Big](\bm{\xi})
&= \frac {e^{- \imag \langle\bm{\xi},\bm{x}\rangle} - 1} {\sqrt{2\pi}}
\frac{ (2 \pi)^{ \frac{d} {2} } } { \| \bm{\xi} \|^{ \frac{d} {2} - 1 } }
\frac{ \left( \frac{\| \bm{\xi} \|} {2 \lambda} \right)^{ \frac{d-2} {2} }
\Gamma(\frac{d}{2} + h) }  {\lambda^{h+1} \Gamma(\frac{d}{2}) }
{}_2 F_1 \left( \frac{ \frac{d}{2} + h} {2}, \frac{ \frac{d}{2} + h + 1} {2};
\frac{d}{2}; - \frac{\| \bm{\xi} \|^2} {\lambda^2} \right) \\
&= C_{h,\lambda} \hspace{1.5mm}{}_2 F_1 \left( \frac{ \frac{d}{2} + h} {2}, \frac{ \frac{d}{2} + h + 1} {2};
\frac{d}{2}; - \frac{\| \bm{\xi} \|^2} {\lambda^2} \right),
\end{split}
\end{equation*}
where $C_{h,\lambda}= \big( (2 \pi)^{ \frac{d} {2} } \Gamma(\frac{d}{2} + h) \big)
\big( 2^{ \frac{d-2} {2} } \lambda^{ \frac{d}{2} + h } \Gamma(\frac{d}{2}) \big)^{-1}$. Note that the matrix-valued function
$$
\mathcal{F}\big[ g_{H, \lambda, \bm{x}}\big] (\bm{\xi}) := P\textnormal{diag}\Big(\mathcal{F}\big[ g_{h_1, \lambda, \bm{x}}\big](\bm{\xi}),\hdots,
\mathcal{F}\big[ g_{h_n, \lambda, \bm{x}}\big](\bm{\xi})\Big) P^{-1}
$$
is Hermitian (in the sense of functions). Therefore, by applying Parseval's identity,
\begin{equation*}
\begin{split}
B_{H, \lambda} (\bm{x}) & = \int_{\rr^d} g_{H, \lambda, \bm{x}} (\bm{y}) \bm{Z}(d\bm{y})\stackrel{f.d.}= \int_{\rr^d} \mathcal{F}\big[ g_{H, \lambda, \bm{x}}\big] (\bm{\xi}) W(d\bm{\xi})\\
& = \frac{1}{\sqrt{2\pi}} \hspace{1mm}C_{H,\lambda} \int_{\rr^d} (e^{- \imag \langle\bm{\xi},\bm{x}\rangle} - 1)\ {}_2 F_1 \left( \frac{ I\frac{d}{2} + H} {2}, \frac{ I\frac{d}{2} + H + I} {2};
\frac{d}{2}; - \frac{\| \bm{\xi} \|^2} {\lambda^2} \right) W(d\bm{\xi}).
\end{split}
\end{equation*}
In other words, \eqref{TFGFI harmo} holds. $\Box$\\

\noindent {\sc Proof of Proposition \ref{prop2}}: By the harmonizable representation \eqref{TFGFI harmo} of $\widetilde{B}_{H, \lambda} ( \bm{x} )$,
\begin{equation}\label{eq:covariance_exponential_harmo}
\begin{split}
{\rm Cov}[ B_{H, \lambda} (\bm{x}), B_{H, \lambda} (\bm{x^{\prime}}) ]
&= \mathbb{E} [ B_{H, \lambda} (\bm{x})\overline{ B_{H, \lambda} (\bm{x^{\prime}}) } ] \\
& = \frac{C^2_{H,\lambda}}{2\pi}\int_{\mathbb{R}^{d}}
(e^{- \imag \langle\bm{\xi},\bm{x} - \bm{x^{\prime}}\rangle} - e^{- \imag \langle\bm{\xi},\bm{x}\rangle} - e^{\imag \langle\bm{\xi},\bm{x^{\prime}}\rangle} + 1)\\
&\qquad\qquad\qquad\qquad\times {}_2 {\mathbf F}_1 \left( \frac{I \frac{d}{2} + H} {2}, \frac{ \frac{d}{2} + H + 1} {2};
\frac{d}{2}; - \frac{\| \bm{\xi} \|^2} {\lambda^2} \right) d\bm{\xi}.
\end{split}
\end{equation}
In order to simplify expression \eqref{eq:covariance_exponential_harmo}, we consider the integrals
\begin{eqnarray*}
&&I_{1} = \int_{\mathbb{R}^{d}} e^{\imag \langle\bm{\xi},\bm{x}\rangle}
{}_2 {\mathbf F}_1 \left( \frac{ I \frac{d}{2} + H} {2}, \frac{I \frac{d}{2} + H + I} {2};
\frac{d}{2}; - \frac{\| \bm{\xi} \|^2} {\lambda^2} \right) d\bm{\xi}, \\
&& I_{2} = \int_{\mathbb{R}^{d}}
{}_2 {\mathbf F}_1 \left( \frac{ I \frac{d}{2} + H} {2}, \frac{ I \frac{d}{2} + H + I} {2};
\frac{d}{2}; - \frac{\| \bm{\xi} \|^2} {\lambda^2} \right) d\bm{\xi}.
\end{eqnarray*}
In regard to $I_{1}$, in light of expression \eqref{e:Bessel_shows_up} and by changing Euclidean coordinates into polar coordinates for $I_1$, we arrive at
\begin{equation}\label{eq:I1_integral}
\begin{split}
I_{1}&= \int_{\mathbb{R}^{d}} e^{\imag \langle\bm{\xi},\bm{x}\rangle}
{}_2 {\mathbf F}_1 \left( \frac{ \frac{d}{2}I + H} {2}, \frac{ I \frac{d}{2} + H + I} {2};
\frac{d}{2}; - \frac{\| \bm{\xi} \|^2} {\lambda^2} \right) d\bm{\xi} \\
&= \int_{0}^{\infty} {}_2 {\mathbf F}_1 \left( \frac{ \frac{d}{2}I + H} {2}, \frac{ I \frac{d}{2} + H + I} {2};
\frac{d}{2}; - \frac{r^2} {\lambda^2} \right)
\left[ \int_{S_{r}} e^{\imag \langle\bm{\xi},\bm{x}\rangle} dS \right] dr \\
&= \int_{0}^{\infty} {}_2 {\mathbf F}_1 \left( \frac{ \frac{d}{2}I + H} {2}, \frac{ I \frac{d}{2} + H + I} {2};
\frac{d}{2}; - \frac{r^2} {\lambda^2} \right)
\left( \frac{ 2 \pi r } { \| \bm{x} \| } \right)^{ \frac{d} {2} } \| \bm{x} \|\ J_{ \frac{d-2} {2} } (\| \bm{x} \|r)\ dr \\
&=  \left( 2\pi \right)^{ \frac{d} {2} } \| \bm{x} \|^{ 1 - \frac{d} {2} }
\int_{0}^{\infty} r^{ \frac{d}{2} } J_{ \frac{d-2} {2} } (\| \bm{x} \|r)
\ {}_2 {\mathbf F}_1 \left( \frac{ I\frac{d}{2} + H} {2}, \frac{ \frac{d}{2}I + H + I} {2};
\frac{d}{2}; - \frac{r^2} {\lambda^2} \right) dr.
\end{split}
\end{equation}
We now turn to $I_2$. Consider the following fact: in $\bbR^d$, for $0 < r < \infty$, $0 \leq \theta_i \leq \pi$, $i = 1,\hdots,d$, $0 \leq \theta_{d-1} \leq 2\pi$, the Jacobian for a change of variables from Euclidean to polar coordinates is given by
\begin{eqnarray*}
\left| \frac { \partial(k_1, \hdots, k_n) } { \partial(r, \theta_1, \hdots, \theta_{n-1}) } \right|
= r^{n-1} \prod_{i=1}^{n-1} (\sin \theta_i)^{n-i-1}.
\end{eqnarray*}
Moreover,
\begin{equation}\label{eq:polar_cordinate_calculus}
\begin{split}
\int_{[0, \pi]^{d-2} \times [0, 2\pi]} \prod_{i=1}^{d-1} (\sin \theta_i)^{d-i-1}
d\theta_1, \hdots, d\theta_{d-1}
&= \prod_{i=1}^{d-1} \int_{[0, \pi]}  (\sin \theta_i)^{d-i-1} d\theta_i
\times \int_{[0, 2\pi]} d\theta_{d-1} \\
&= \begin{cases}
\displaystyle{ \frac{2 \cdot (2\pi)^m} {(2m+1)!!} \times 2\pi }, & \textnormal{if }\ \ d = 2m+1; \\
\displaystyle{ \frac{(2\pi)^m} {(2m)!!} \times 2\pi }, & \textnormal{if} \ \ d = 2m.
\end{cases}
\end{split}
\end{equation}
As a consequence of \eqref{eq:polar_cordinate_calculus}, 
\begin{equation}\label{eq:I2_integral}
{\mathbf D}^2_{H,\lambda} = I_2.
\end{equation}
Expression \eqref{tmp} now follows from \eqref{eq:covariance_exponential_harmo}, \eqref{eq:I1_integral} and \eqref{eq:I2_integral}. $\Box$\\

\noindent {\sc Proof of Proposition \ref{p:I-B-TOFBF harmo}}: We establish \eqref{I-B-TOFBF harmo} entry-wise by computing the Fourier transform of the generic kernel function component
$$
\bbR \ni h_{h, \lambda, \bm{x}} (\bm{\xi}) =  \| \bm{x}-\bm{y} \| ^ { h - \frac{d}{2} } K_{ h - \frac{d}{2} }
\Big( \lambda \| \bm{x}-\bm{y} \| \Big) - \| -\bm{y} \| ^ { h - \frac{d}{2} } K_{ h - \frac{d}{2} } \Big( \lambda \| -\bm{y} \| \Big)
\Big].
$$
In fact, recall the notation \eqref{e:F(f)(xi)} for the Fourier transform. Then,
\begin{equation*}
\begin{split}
\mathcal{F}\Big[h_{h, \lambda, \bm{x}} (\bm{\xi})\Big]
&= \frac{1}{(2\pi)^{d/2}}\int_{\mathbb{R}^{d}} e^{\imag \langle\bm{\xi},\bm{y}\rangle}
\Big[ \| \bm{x}-\bm{y} \| ^ { h - \frac{d}{2} } K_{ h - \frac{d}{2} }
\Big( \lambda \| \bm{x}-\bm{y} \| \Big) - \| -\bm{y} \| ^ { h - \frac{d}{2} } K_{ h - \frac{d}{2} } \Big( \lambda \| -\bm{y} \| \Big)
\Big] d\bm{y} \\
=& \frac{1}{(2\pi)^{d/2}} \int_{\mathbb{R}^{d}} e^{\imag \langle\bm{\xi},\bm{y}\rangle}
\| \bm{x}-\bm{y} \| ^ { h - \frac{d}{2} } K_{ h - \frac{d}{2} }
\left( \lambda \| \bm{x}-\bm{y} \| \right) d\bm{y}\\
&- \int_{\mathbb{R}^{d}} e^{\imag \langle\bm{\xi},\bm{y}\rangle} \| -\bm{y} \| ^ { h - \frac{d}{2} }
K_{ h - \frac{d}{2} } \left( \lambda \| -\bm{y} \| \right) d\bm{y}=: I_{1} - I_{2}.
\end{split}
\end{equation*}
The change of variable $-\bm{y^{\prime}} = \bm{x}-\bm{y}$ yields
\begin{equation*}
\begin{split}
I_{1}&= \int_{\mathbb{R}^{d}} e^{\imag \langle\bm{\xi},\bm{x} + \bm{y^{\prime}}\rangle}
\| -\bm{y^{\prime}} \| ^ { h - \frac{d}{2} } K_{ h - \frac{d}{2} }
\left( \lambda \| -\bm{y^{\prime}} \| \right) d\bm{y^{\prime}} \\
&= e^{\imag \langle\bm{\xi},\bm{x}\rangle} \int_{\mathbb{R}^{d}} e^{\imag \langle\bm{\xi},\bm{y^{\prime}}\rangle}
\| -\bm{y^{\prime}} \| ^ { h - \frac{d}{2} } K_{ h - \frac{d}{2} }
\left( \lambda \| -\bm{y^{\prime}} \| \right) d\bm{y^{\prime}}
\end{split}
\end{equation*}
Therefore,
\begin{eqnarray}
 \mathcal{F}\Big[h_{h, \lambda, \bm{x}} (\bm{\xi})\Big]
&=& \frac {e^{\imag \langle\bm{\xi},\bm{x}\rangle} - 1} { (2\pi)^{d/2} } \int_{\mathbb{R}^{d}} e^{\imag \langle\bm{\xi},\bm{y}\rangle}
\| -\bm{y} \| ^ { h - \frac{d}{2} } K_{ h - \frac{d}{2} }
\left( \lambda \| -\bm{y} \| \right) d\bm{y} \nonumber \\
&=& \frac {e^{\imag \langle\bm{\xi},\bm{x}\rangle} - 1} {  (2\pi)^{d/2} } \int_{\mathbb{R}^{d}} e^{\imag \langle\bm{\xi},\bm{y}\rangle}
\| \bm{y} \| ^ { h - \frac{d}{2} } K_{ h - \frac{d}{2} }
\left( \lambda \| \bm{y} \| \right) d\bm{y} \nonumber \\
&=& (e^{\imag \langle\bm{\xi},\bm{x}\rangle} - 1)\ \mathcal{F} \Big[ \| \bm{y} \| ^ { h - \frac{d}{2} } K_{ h - \frac{d}{2} }
\left( \lambda \| \bm{y} \| \right) \Big] ( \bm{\xi} ). \label{eq7}
\end{eqnarray}
On the other hand,
\begin{equation*}
\begin{split}
\mathcal{F} \Big[ \| \bm{y} \| ^ { h - \frac{d}{2} } K_{ h - \frac{d}{2} } \left( \lambda \| \bm{y} \| \right) \Big](\bm{\xi})
&= \frac {1} { (2\pi) ^ { \frac{d}{2}} } \int_{\rr^{d}}
e^{\imag \langle\bm{\xi},\bm{y}\rangle} \Big[ \| \bm{y} \| ^ { h - \frac{d}{2} } K_{ h - \frac{d}{2} }
\left( \lambda \| \bm{y} \| \right) \Big] d\bm{y}\\
&= \frac{\lambda^{-d/2-h}}{(2\pi)^{d/2}} \int_{\rr^{d}}
e^{ \imag \langle\bm{\xi}/\lambda,\bm{y^\prime}\rangle } \Big[ \| \bm{ y^\prime } \| ^ { h - \frac{d}{2} } K_{ h - \frac{d}{2} }
\left( \| \bm{y^\prime} \| \right) \Big] d\bm{y^\prime}\\
&=\lambda^{h-d/2}2^{h-1}\Gamma(h) (\lambda^2 + \|{\bm{\xi}}\|^2 )^{-h},
\end{split}
\end{equation*}
where we make the change of variable $\lambda\bm{y} = \bm{y'}$ and apply the formula
\begin{equation*}
\mathcal{F} \Big[ \| \bm{y} \| ^ { h - \frac{d}{2} } K_{ h - \frac{d}{2} } \left( \| \bm{y} \| \right) \Big]({\bm \xi})
=2^{h-1} \Gamma(h) (1+ {\|\bm{\xi}\|}^2 )^{-h}, \quad h > 0
\end{equation*}
(see Lim and Teo \cite{Lim:Teo:2008}, p.\ 013509-3). Note that the matrix-valued function
$$
\mathcal{F}\big[ h_{H, \lambda, \bm{x}}\big]({\bm \xi})  = P \textnormal{diag}\Big(\mathcal{F}\big[ h_{h_1, \lambda, \bm{x}}\big]({\bm \xi}),\hdots,
\mathcal{F}\big[ h_{h_n, \lambda, \bm{x}}\big]({\bm \xi})\Big)P^{-1}
$$
is Hermitian (in the sense of functions). Therefore, by applying Parseval's identity, we arrive at
\begin{equation*}
\begin{split}
B^{Bes}_{H, \lambda}(\bm{x}) & =  \int_{\rr^d} h_{H, \lambda, \bm{x}} (\bm{y}) Z(d\bm{y}) \stackrel{f.d.}=
\int_{\rr^d} \mathcal{F}\big[ h_{H, \lambda, \bm{x}}\big]  W(d\bm{\xi})\\
& = C^*_{H,\lambda} \int_{\rr^d} (e^{- \imag \langle\bm{\xi},\bm{x}\rangle} - 1)  ( \lambda^2 + \|{\bm{\xi}}\|^2 )^{-H}  W(d\bm{\xi})
\end{split}
\end{equation*}
for a matrix constant $C^*_{H,\lambda}$. In other words, \eqref{I-B-TOFBF harmo} holds. $\Box$\\

\noindent {\sc Proof of Proposition \ref{p:cov_low d}}: The proof is similar to that of Proposition \ref{prop2}. By applying expression \eqref{I-B-TOFBF harmo}, we can write
\begin{eqnarray*}
{\rm Cov}[ B^{Bes}_{H, \lambda} (\bm{x}), B^{Bes}_{H, \lambda} (\bm{x^{\prime}}) ]
&=& \mathbb{E} [ B^{Bes}_{H, \lambda} (\bm{x}) B^{Bes}_{H, \lambda}(\bm{x^{\prime}})^t  ]
\end{eqnarray*}
$$
= \frac { \lambda^{2H - d I} } {2\pi} \int_{\mathbb{R}^{d}}
(e^{- \imag \langle\bm{\xi},\bm{x} - \bm{x^{\prime}}\rangle} - e^{- \imag \langle\bm{\xi},\bm{x}\rangle} - e^{\imag \langle\bm{\xi},\bm{x^{\prime}}\rangle} + 1)
 (\lambda^2 + \| \bm{\xi} \|^2)^{-H}(\lambda^2 + \| \bm{\xi} \|^2)^{-H^*}  d\bm{\xi}
$$
To establish \eqref{tmp}, we consider the integrals
\begin{eqnarray*}
&& I_{1}:=\int_{\mathbb{R}^{d}}
e^{\imag \langle\bm{\xi},\bm{x}\rangle} (\lambda^2 + \| \bm{\xi} \|^2)^{-H}(\lambda^2 + \| \bm{\xi} \|^2)^{-H^*}d\bm{\xi}, \\
&& I_{2}:=\int_{\mathbb{R}^{d}} (\lambda^2 + \| \bm{\xi} \|^2)^{-H}(\lambda^2 + \| \bm{\xi} \|^2)^{-H^*} d\bm{\xi}
\end{eqnarray*}
In regard to $I_1$, by switching from Euclidean to polar coordinates, we obtain
\begin{eqnarray}
I_1 &=& \int_{\mathbb{R}^{d}}
e^{\imag \langle\bm{\xi},\bm{x}\rangle}   (\lambda^2 + \| \bm{\xi} \|^2)^{-H}(\lambda^2 + \| \bm{\xi} \|^2)^{-H^*}d\bm{\xi} \nonumber\\
&=& \int_{0}^{\infty} (\lambda^2 + r^2)^{H}(\lambda^2 + r^2)^{-H^*}
\left[ \int_{S_{r}} e^{\imag \langle\bm{\xi},\bm{x}\rangle} dS \right] dr \nonumber \\
&=& \int_{0}^{\infty} (\lambda^2 + r^2)^{H}(\lambda^2 + r^2)^{-H^*}\left( \frac{ 2\pi r } { \| \bm{x} \| } \right)^{ \frac{d}{2} } \| \bm{x} \|\ J_{ \frac{d-2}{2} } (\| \bm{x} \|r)\ dr \nonumber\\
&=& \left( 2\pi \right)^{ \frac{d} {2} } \| \bm{x} \|^{ 1 - \frac{d} {2} }
\int_{0}^{\infty} (\lambda^2 + r^2)^{H}(\lambda^2 + r^2)^{-H^*} r^{ \frac{d}{2} } J_{ \frac{d-2}{2} } (\| \bm{x} \|r) dr. \label{e:typeII_integral_I1}
\end{eqnarray}
However, for $\displaystyle{ -1 < \Re \nu < \Re \left( 2\mu + \frac{3}{2} \right), a > 0, b > 0}$,
\begin{eqnarray*}
\int_{0}^{\infty} \frac {x^{ \nu+1}\ J_{ \nu } (bx)  } { (x^2 + a^2)^{\mu + 1} }\ dx
= \frac { a^{\nu - \mu} b^{ \mu } } { 2^{\mu} \Gamma (\mu + 1) } K_{\nu - \mu} (ab)
\end{eqnarray*}
(Gradshteyn and Ryzhik \cite{gradshteyn:ryzhik:1994}, p.686). So, recall that we can write $H = PJ_HP^{-1}$, where $J_H = \textnormal{diag}(h_1,\hdots,h_n)$. By setting $\displaystyle{ a = \lambda, b = \| \bm{x} \|,  \mu = h_\ell + h_{\ell'} - 1, \nu = \frac{d}{2} - 1}$, $\ell,\ell'= 1,\hdots,n$, we can rewrite the integral on the right-hand side of \eqref{e:typeII_integral_I1} as
$$
P \Big(\int^{\infty}_{0} (\lambda^2 + r^2)^{-J_H}(P^*P)^{-1}(\lambda^2 + r^2)^{-J_H}  r^{\frac{d}{2}}J_{ \frac{d-2}{2} } (\| \bm{x} \|r) dr \Big)P^*
$$
$$
= P  \Big( q_{\ell \ell'}\hspace{1mm}\frac { \lambda^{\frac{d}{2} - (h_\ell + h_{\ell'}-1)} {\|{\bm x}\|}^{h_\ell + h_{\ell'} - 1} } { 2^{h_\ell + h_{\ell'}-1} \Gamma (h_\ell + h_{\ell'}) } K_{\frac{d}{2} - (h_\ell + h_{\ell'})} (\lambda \|{\bm x}\|) \Big)_{\ell,\ell'=1,\hdots,n} P^*.
$$
Therefore, assuming $\displaystyle{ h_\ell + h_{\ell'} > \frac{d-1}{4} }$, $\ell, \ell'= 1,\hdots,n$,
$$
\int_{\mathbb{R}^{d}}
e^{\imag <\bm{\xi},\bm{x}>}   (\lambda^2 + \| \bm{\xi} \|^2)^{-H}(\lambda^2 + \| \bm{\xi} \|^2)^{-H^*} d\bm{\xi}
$$
$$
= \left( 2\pi \right)^{ \frac{d} {2} }
P  \Big( q_{\ell \ell'}\hspace{1mm}\frac { \lambda^{\frac{d}{2} - (h_\ell + h_{\ell'}-1)} {\|{\bm x}\|}^{h_\ell + h_{\ell'} - \frac{d} {2}} } { 2^{h_\ell + h_{\ell'}-1} \Gamma (h_\ell + h_{\ell'}) } K_{\frac{d}{2} - (h_\ell + h_{\ell'})} (\lambda \|{\bm x}\|) \Big)_{\ell,\ell'=1,\hdots,n} P^*.
$$
In regard to $I_2$, by following the same argument in the proof of Proposition \ref{prop2}, we arrive at
\begin{eqnarray*}
& & \int_{\mathbb{R}^{d}}  (\lambda^2 + \| \bm{\xi} \|^2)^{-H}(\lambda^2 + \| \bm{\xi} \|^2) ^ {-H^*}  d\bm{\xi} \\
& = &\int_{0}^{\infty} (\lambda^2 + r^2)^{-H}(\lambda^2 + r^2)^{-H^*} r^{d-1}   dr \times
\begin{cases}
\displaystyle{ \frac{2 \dot (2\pi)^{m+1} } {(2m+1)!!}}, & \textnormal{if}\ \ d = 2m+1; \\
\displaystyle{ \frac{(2\pi)^{m+1} } {(2m)!!} }, & \textnormal{if} \ \ d = 2m.
\end{cases}
\end{eqnarray*}
However, from Gradshteyn and Ryzhik \cite{gradshteyn:ryzhik:1994}, p.327, assuming $\displaystyle{ \Re \mu > 0, \Re \left( \nu + \frac{\mu}{2} \right) < 1 }$,
\begin{eqnarray*}
\displaystyle{ \int_{0}^{\infty} x^{\mu-1} (1 + x^2) ^ { \nu-1 } dx }
= \frac{1}{2} B \left( \frac{\mu}{2}, 1 - \nu - \frac{\mu}{2} \right).
\end{eqnarray*}
Therefore, by making the change of variable $r = \lambda r^{\prime}$ and $(q_{\ell \ell'})_{\ell,\ell'=1,\hdots,n}$ as in \eqref{e:(P*P)^(-1)=q_ell,ell'},
\begin{eqnarray*}
\displaystyle{ \int_{0}^{\infty}  (\lambda^2 + r^2)^{-H}(\lambda^2 + r^2)^{-H^*} r^{d-1} dr }
&=& \lambda^{d I -4H} \int_{0}^{\infty}  (1 + r^{\prime 2})^{-H}(1 + r^{\prime 2})^{-H^*} r^{\prime d-1} dr^{\prime} \\
&=& \frac{ \lambda^{d I-4H} }{2} P \left( q_{\ell \ell'}\hspace{1mm}B \Big( \frac{d}{2}, h_\ell + h_{\ell'} - \frac{d}{2} \Big) \right)_{\ell,\ell'=1,\hdots,n}P^*
\end{eqnarray*}
whenever $\displaystyle{ h_{\ell} + h_{\ell'} > \frac{d}{2} }$, $\ell, \ell'= 1,\hdots,n$. Thus,
\begin{eqnarray*}
& & \int_{\mathbb{R}^{d}}   (\lambda^2 + \| \bm{\xi} \|^2)^{-H}(\lambda^2 + \| \bm{\xi} \|^2)^{-H^*}  d\bm{\xi}\\
&=& \lambda^{d I-4H} P \left( q_{\ell \ell'}\hspace{1mm}B \Big( \frac{d}{2}, h_\ell + h_{\ell'} - \frac{d}{2} \Big) \right)_{\ell,\ell'=1,\hdots,n}P^*\begin{cases}
\displaystyle{ \frac{ (2\pi)^{m+1}  } {(2m+1)!! } }, & \textnormal{ if }\ \ d = 2m+1; \\
\displaystyle{ \frac{(2\pi)^{m+1}  } {2 (2m)!!}}, & \textnormal{ if }\ \ d = 2m.
\end{cases}
\end{eqnarray*}
Therefore, expression \eqref{e:cov_low d} holds under condition \eqref{e:h1>d/4}. $\Box$\\

\section{Section \ref{s:sample_path}: proofs}

The following classical proposition is used in proofs. We recap it here for the reader's convenience.
\begin{prop}(Adler \cite{adler:1981}, Theorems 3.3.2 and 8.3.2; Bierm\'{e} et al.\ \cite{bierme:meerschaert:scheffler:2007}, Proposition 5.2)\label{prop4.1}
Let $\{ X(\bm{x}) \}_{ \bm{x} \in \mathbb{R}^d }$ be a $\bbR$-valued Gaussian random field with stationary increments. Let $\eta \in (0, 1)$. If \begin{eqnarray*}
\eta = \sup \{ \beta > 0 : \mathbb{E} [ (X(\bm{x}) - X(\bm{0}))^2 ] = o_{ \| \bm{x} \| \rightarrow 0 } ( \| \bm{x} \|^{2\beta} ) \},
\end{eqnarray*}
then, for any $\beta \in (0, \eta)$, any continuous version of $\{ X(\bm{x}) \}_{ \bm{x} \in \mathbb{R}^d }$ satisfies property (a) in Definition \ref{CriticalHolder}. If, in addition,
\begin{eqnarray*}
\eta = \inf \{ \beta > 0 : \| \bm{x} \|^{2\beta} = o_{ \| \bm{x} \| \rightarrow 0 } ( \mathbb{E} [ (X(\bm{x}) - X(\bm{0}))^2 ] ) \},
\end{eqnarray*}
then any continuous version of $\{ X(\bm{x}) \}_{ \bm{x} \in \mathbb{R}^d }$ has the H\"older critical exponent $\eta$.
\end{prop}

In the following proofs, without loss of generality we assume $d \geq 2$.\\

\noindent {\sc Proof of Theorem \ref{thm6.1}}: We first show $(i)$. The proof follows from Proposition \ref{prop4.1} and the asymptotic behavior of second moment of $X_{\lambda} (\bm{x})$ around $\bm{x} = \bm{0}$. Let $\bm{x} \in \mathbb{R}^d$. We know that $X_{\lambda} (\bm{0}) = 0$. So, define
\begin{eqnarray*}
\Gamma_{\varphi, H, \lambda}^2 (\bm{x}) = \mathbb{E} [ X_{\lambda} (\bm{x})^2 ] = \int_{\mathbb{R}^d} \left| e^{-\lambda \varphi({\bm x} - {\bm y})} \varphi({\bm x} - {\bm y})^{H - \frac{q}{2}} - e^{-\lambda \varphi(-{\bm y})} \varphi(-{\bm y})^{H - \frac{q}{2}} \right|^2 d{\bm y}.
\end{eqnarray*}
By property \eqref{e:tempered_o.s.s.} and by using polar coordinates with respect to $E$ (Lemma \ref{l:polar_integration}), we obtain
\begin{equation}\label{tmp10}
\begin{split}
&\Gamma_{\varphi, H, \lambda}^2 (\bm{x}) = \tau(\bm{x})^{2H} \Gamma_{\varphi, H, \tau(\bm{x}) \lambda}^2 ( l(\bm{x}) )  \\
&= \tau(\bm{x})^{2H} \int_{\mathbb{R}^d} \left| e^{- \tau(\bm{x}) \lambda \varphi(l(\bm{x}) - {\bm y})} \varphi(l(\bm{x}) - {\bm y})^{H - \frac{q}{2}} - e^{- \tau(\bm{x}) \lambda \varphi(-{\bm y})} \varphi(-{\bm y})^{H - \frac{q}{2}} \right|^2 d{\bm y}  \\
&= \tau(\bm{x})^{2H} \int_{\mathbb{R}^d} \left| \Big[1 - O\Big(\tau(\bm{x}) \lambda \varphi(l(\bm{x}) - {\bm y})\Big)\Big] \varphi(l(\bm{x}) - {\bm y})^{H - \frac{q}{2}} - \Big[1 - O\Big(\tau(\bm{x}) \lambda \varphi(-{\bm y})\Big)\Big] \varphi(-{\bm y})^{H - \frac{q}{2}} \right|^2 d{\bm y} \\
&= O(\tau(\bm{x})^{2H}), \quad \textnormal{as } \| \bm{x} \| \rightarrow \bm{0}.
\end{split}
\end{equation}
In order to apply to Proposition \ref{prop4.1}, we construct bounds on $\tau(\bm{x})$ based on the conventional Euclidean norm $\| \bm{x} \|$ (cf.\ Bierm\'{e} et al.\ \cite{bierme:meerschaert:scheffler:2007}, Theorem 5.4). Since we are interested in the behavior of $\bm{x}$ around the origin, without loss of generality we can assume $\| \bm{x} \| \leq 1$. Fix $\bm{r} \in W_i \backslash W_{i-1}$ for any $i = 1, \hdots, p$. Using the space $W_i$ instead of $\mathbb{R}^d$ in Lemma \ref{lem1.1}, we obtain that, for any small $\delta > 0$, there exists a constant $C_2 > 0$ such that $\tau(t \bm{r}) \leq C_2 |t|^{1/a_i - \delta}$ for $|t| \leq 1$, because the eigenvalues of $E |_{W_i}$ are $a_1, \hdots, a_i$. Moreover, write $\bm{r} = \bm{r}_i + \overline{\bm{r}}_{i-1}$ with $\bm{r}_i \in V_i$ and $\overline{\bm{r}}_{i-1} \in W_{i-1}$. Then, we can decompose $t \bm{r} = \tau(t \bm{r})^E l(t \bm{r}) = \tau(t \bm{r})^E l_i(t \bm{r}) + \tau(t \bm{r})^E \overline{l}_{i-1}(t \bm{r})$, where $l_i(t \bm{r}) \in V_i$ and $\overline{l}_{i-1}(t \bm{r}) \in W_{i-1}$. Moreover, rewrite $E = E_1 \oplus \hdots \oplus E_p$, where every real part of the eigenvalues of the matrix $E_i$ equals $a_i$. We arrive at the bound
$$
|t| \ \| \bm{r}_i \| = \| \tau(t \bm{r})^E l_i(t \bm{r}) \| = \| \tau(t \bm{r})^{E_i} l_i(t \bm{r}) \|
$$
\begin{eqnarray}\label{e:|t||ri|=<Ctau^(ai-delta)}
\leq \| \tau(t \bm{r})^{E_i} \| \ \|l_i(t \bm{r}) \| \leq C_3 \tau(t \bm{r})^{a_i - \delta}, \quad |t| \leq 1.
\end{eqnarray}
The last inequality in \eqref{e:|t||ri|=<Ctau^(ai-delta)} is a consequence of the facts that $\|l_i(t \bm{r}) \| \leq C$ for any $|t| \leq 1$, and that every real part of the eigenvalues of $E_i$ equals $a_i$. Therefore, there exists a constant $C_1 > 0$ such that $\tau(t \bm{r}) \geq C_1 |t|^{1/a_i + \delta}$ for any $|t| \leq 1$. Hence, we conclude that for all directions $\bm{r} \in W_i \backslash W_{i-1}$, and any small $\delta > 0$, there exist constants $C_1, C_2 > 0$, such that
\begin{eqnarray}\label{e:tau(tr)_upper_lower_bound}
C_1 |t|^{1/a_i + \delta} \leq \tau(t \bm{r}) \leq C_2 |t|^{1/a_i - \delta}, \quad |t| \leq 1.
\end{eqnarray}
Therefore, in view of \eqref{tmp10} and \eqref{e:tau(tr)_upper_lower_bound},
\begin{eqnarray*}
C_1^{\prime} |t|^{2H/a_i + \delta} \leq  \Gamma_{\varphi, H, \lambda}^2 (t\bm{r}) \leq C_2^{\prime} |t|^{2H/a_i - \delta}, \quad |t| \leq 1.
\end{eqnarray*}
By Proposition \ref{prop4.1}, this shows that $X_{\lambda}$ has regularity $H/a_i$ in the direction $\bm{r}$, as claimed.
To establish the H\"older critical exponent, we can adapt the argument. Since $H/a_p$ is the H\"older critical exponent of $X_{\lambda}$ in any direction of $W_p \backslash W_{p-1}$, then for any $\beta \in (H/a_p, 1)$ the sample paths of $X_{\lambda}$ do not almost surely satisfy any uniform H\"older condition of order $\beta$. Moreover, from Lemma \ref{lem1.1}, we get that for any $\delta > 0$, there exists a constant $C > 0$ such that $ \Gamma_{\varphi, H, \lambda}^2 (\bm{x}) \leq C \| \bm{x} \|^{2H/a_p - \delta}$ for $\| \bm{x} \| \leq 1$. Proposition \ref{prop4.1} then implies that any continuous version of $X_{\lambda}$ almost surely satisfies a uniform H\"older condition of order $\beta < H/a_p$ on any compact set. This concludes the proof of $(i)$.

To show $(ii)$, in view of the proof of $(i)$, it suffices to establish the asymptotic behavior of the second moment of $\widetilde{X}_{\lambda} (\bm{x})$ around $\bm{x} = \bm{0}$. So, let $\bm{x} \in \mathbb{R}^d$. By property \eqref{e:tempered_o.s.s.} and by using polar coordinates with respect to $E$ (Lemma \ref{l:polar_integration}),
$$
\bbE \widetilde{X}^{2}_{\lambda}({\bm x}) = \tau({\bm x})^{2H} \bbE \widetilde{X}^{2}_{\tau({\bm x})\lambda}(l({\bm x}))
$$
$$
= \tau({\bm x})^{2H} \int_{\bbR^d} |e^{- \imag \langle l({\bm x}),{\boldsymbol \xi} \rangle}-1|^2 \frac{1}{(\tau({\bm x})\lambda + \varphi({\boldsymbol \xi}))^{2H+q}} d {\boldsymbol \xi}
$$
$$
= \tau({\bm x})^{2H} \int^{\infty}_{0}\int_{S_0} |e^{- \imag \langle l({\bm x}),r^{E}\theta\rangle}-1|^2 \frac{1}{(\tau({\bm x})\lambda + r^{2H+q}\varphi(\theta))^{2H+q}} r^{q-1}\sigma(d \theta) dr
$$
\begin{equation}\label{e:EB-tilde2_tau(x)(2H)}
\sim \tau({\bm x})^{2H} \int^{\infty}_{0}\int_{S_0} |e^{- \imag \langle l({\bm x}),r^{E}\theta\rangle}-1|^2 \frac{1}{r^{2H+1}\varphi(\theta)^{2H+q}} \sigma(d \theta) dr,
\end{equation}
where the integral on the right-hand side of \eqref{e:EB-tilde2_tau(x)(2H)} is finite by condition \eqref{e:UpsilonH<varphiE}. This establishes $(ii)$. $\Box$\\

\noindent {\sc Proof of Corollary \ref{cor6.3}}: The argument for showing \eqref{e:Haus_box_dim_MA-OFBF} resembles the proof of Theorem 5.6 in Bierm\'{e} et al.\ \cite{bierme:meerschaert:scheffler:2007}. For the reader's convenience, we provide the details.

Recall that, for any scalar-valued random field $X$,
\begin{eqnarray}
\label{Haus_Box_ineq}
dim_{Haus} \mathcal{G}( X ) \leq \underline{dim}_{Box} \mathcal{G}( X ) \leq \overline{dim}_{Box} \mathcal{G}( X ),
\end{eqnarray}
where $\underline{dim}_{Box}$ and $\overline{dim}_{Box}$ denote the lower and upper box-counting dimension, respectively. From the inequality (\ref{Haus_Box_ineq}), it suffices to show that, almost surely,
\begin{equation}\label{e:dimbox_dimHaus}
\overline{dim}_{ Box } \mathcal{G} (X_{\lambda}) \leq d + 1 - H/a_p, \quad d + 1 - H/a_p \leq dim_{ Haus } \mathcal{G} (X_{\lambda}).
\end{equation}
We first show the left inequality in \eqref{e:dimbox_dimHaus}. Consider a continuous version of $X_{\lambda}$. From Theorem \ref{thm6.1} and a $d$-dimensional version of Corollary 11.2 in Falconer \cite{falconer:1990}, we obtain that $\overline{dim_{ Box }} \mathcal{G} (X_{\lambda}) \leq d + 1 - \beta$ a.s.\ for any $\beta < H/a_p$. Therefore,
$$
\overline{dim_{ Box }} \mathcal{G} (X_{\lambda}) \leq d + 1 - H/a_p \quad \textnormal{a.s.}
$$
We now prove the right inequality in \eqref{e:dimbox_dimHaus}. Fix $\beta > 1$. If we show
\begin{eqnarray}\label{Ibeta1}
I_{\beta} := \int_{K \times K} \mathbb{E} [ ( X_\lambda(\bm{x}) - X_\lambda(\bm{y}))^{2} + (\| \bm{x} - \bm{y} \|^2 )^{-\beta/2} ] d\bm{x} d\bm{y} < \infty,
\end{eqnarray}
then the Frostman criterion (Falconer \cite{falconer:1990}, Theorem 4.13, (a)), implies that $dim_{Haus} \mathcal{G}(X_\lambda) \geq \beta$ a.s.
So, to prove \eqref{Ibeta1}, first note that $(x^2 + 1)^{-\beta/2} \in L^1(\mathbb{R})$ for $\beta > 1$. Thus, its Fourier transform $\hat{f}_{\beta} (\xi)$
\begin{eqnarray*}
(x^2 + 1)^{-\beta/2} = \frac{1}{2 \pi} \int_{\mathbb{R}} e^{i x \xi} \hat{f}_{\beta} (\xi) d\xi.
\end{eqnarray*}
is not only in $L^{\infty}$ but also it is in $L^{1}(\mathbb{R})$ as well. Using this fact and the Gaussian assumption on $X_{\lambda}$ we have
\begin{eqnarray*}
&& \mathbb{E} [ \{ (X_{\lambda}(\bm{x}) - X_{\lambda}(\bm{y}))^{2} + \| \bm{x} - \bm{y} \|^2 \}^{-\beta/2} ] \\
&& \qquad = \frac{1}{2\pi} \| \bm{x} - \bm{y} \|^{-\beta} \int_{\mathbb{R}} \mathbb{E} \left[ e^{i \xi \frac {X_{\lambda}(\bm{x}) - X_{ \lambda}(\bm{y})} {\| \bm{x} - \bm{y} \|} }
\right] \hat{f}_\beta(\xi) d\xi \\
&& \qquad = \frac{1}{2\pi} \| \bm{x} - \bm{y} \|^{-\beta} \int_{\mathbb{R}} e^{- \frac{\xi^2}{2} \frac {\mathbb{E} \left[ (X_{\lambda}(\bm{x}) - X_{\lambda}(\bm{y}) )^2 \right] } { \| \bm{x} - \bm{y} \|^2 } } \hat{f}_\beta(\xi) d\xi.
\end{eqnarray*}
Since $f_\beta \in L^{\infty} (\mathbb{R})$, then there exists $C > 0$ such that
\begin{eqnarray}
\label{key_ineq_dim}
&& \mathbb{E} [ ( (X_{\lambda}(\bm{x}) - X_{\lambda}(\bm{y}))^{2} + \| \bm{x} - \bm{y} \|^2 )^{-\beta/2} ] \nonumber \\
&& \qquad \leq C \| \bm{x} - \bm{y} \|^{1-\beta} (\mathbb{E} [ (X_{\lambda}(\bm{x}) - X_{\lambda}(\bm{y}))^2 ] )^{-1/2} \nonumber \\
&& \qquad \leq C^{\prime} \| \bm{x} - \bm{y} \|^{1-\beta} \tau(\bm{x} - \bm{y})^{-H},
\end{eqnarray}
where we use \eqref{tmp10} and the fact that $X_{\lambda}$ has stationary increments. Next, pick $a > 0$ satisfying $K \subset \{ \bm{x} \in \mathbb{R}^d ; \| \bm{x} \| \leq a/2 \}$. Then, there exists a constant $C > 0$ such that
\begin{equation}\label{Ibeta2}
I_{\beta} \leq C \int_{ \| \bm{x} \| \leq a} \| \bm{x} \|^{1-\beta} \tau( \bm{x} )^{-H} d\bm{x}
\end{equation}
is finite as long as the right-hand side of \eqref{Ibeta2} is bounded. For $p = 1$ (see \eqref{e:Wi}), Lemma \ref{lem1.1} implies that, for $\delta > 0$, there exists a $C > 0$ such that, for $\| \bm{x} \| \leq a$,
\begin{eqnarray*}
\tau( \bm{x} )^{-H} \leq \| \bm{x} \|^{-H/a_p - \delta}.
\end{eqnarray*}
Therefore, $I_{\beta}$ is finite if $1 - \beta - H/a_p - \delta > -d$, that is,  $\beta < d + 1 - H/a_p - \delta$. Now suppose $p \geq 2$. Decompose $\bm{x} = \bm{x}_p + \overline{\bm{x}}_{p-1} = \tau(\bm{x})^E l_p(t \bm{x}) + \tau(\bm{x})^E \overline{l}_{p-1}(\bm{x})$ with $\bm{x}_p, l_p(\bm{x}) \in V_p$ and $\overline{\bm{x}}_{p-1}, \overline{l}_{p-1}(\bm{x}) \in W_{p-1}$ as we did in the proof of Theorem \ref{thm6.1}. Again without loss of generality, we can choose an inner product $(\cdot, \cdot)$ on $\mathbb{R}^d$ that makes these spaces mutually orthogonal and set the norm $\| \bm{x} \| = (\bm{x}, \bm{x})^{1/2}$. Therefore, $V_p$ and $W_{p-1}$ are orthogonal, and $\| \bm{x_p} \| \leq a$ and $\| \overline{\bm{x}}_{p-1} \| \leq a$. Moreover, as discussed in the proof of Theorem \ref{thm6.1}, by Lemma \ref{lem1.1} restricted to the spaces $V_p$ and $W_{p-1}$, respectively, we obtain that for any $\delta > 0$ there exists a constant $c > 0$ such that, for $\| \bm{x} \| \leq a$,
\begin{eqnarray*}
\tau(\bm{x})^H \geq c \| \bm{x_p} \|^{H/a_p + \delta} \quad \textnormal{and} \quad \tau(\bm{x})^H \geq c \| \overline{\bm{x}}_{p-1} \|^{H/a_1 + \delta}.
\end{eqnarray*}
Hence, $\tau(\bm{x})^H \geq c/2 \left( \| \bm{x_p} \|^{H/a_p + \delta} + \| \overline{\bm{x}}_{p-1} \|^{H/a_1 + \delta} \right)$. Therefore, by the triangle inequality,
\begin{eqnarray*}
I_{\beta} \leq C \int_{ \| \bm{x_p} \| \leq a} \int_{ \| \bm{y} \| \leq a} \left( \| \bm{x_p} \|^2 + \| \overline{\bm{x}}_{p-1} \|^2 \right)^{1/2 - \beta/2} \left( \| \bm{x_p} \|^{H/a_p + \delta} + \| \overline{\bm{x}}_{p-1} \|^{H/a_1 + \delta} \right)^{-1} d\overline{\bm{x}}_{p-1}d\bm{x_p}.
\end{eqnarray*}
Let $k$ be the dimension of $V_p$. We know that $1 \leq k \leq d-1$. By using polar coordinates for both $\bm{x_p}$ and $\overline{\bm{x}}_{p-1}$, and we get that there exists a constant $C > 0$ such that $I_{\beta} \leq C J_{\beta}$, where
\begin{eqnarray*}
J_{\beta} = \int_0^a \int_0^a \left( r^2 + s^2 \right)^{1/2 - \beta/2} \left( r^{H/a_p + \delta} + s^{H/a_1 + \delta} \right)^{-1} r^{k-1} s^{d-1-k} dr ds.
\end{eqnarray*}
By a change of variable $r = ts$,
\begin{eqnarray*}
J_{\beta} &=& \int_0^a \int_0^{a/s} s^{d - \beta - H/a_p - \delta} (t^2 + 1)^{1/2 - \beta/2}
\left( t^{H/a_p + \delta} + s^{H/a_1 - H/a_p} \right)^{-1} t^{k-1} dt ds \\
&\leq& \left( \int_0^a s^{d - \beta - H/a_p - \delta} ds \right) \left( \int_0^{\infty} (t^2 + 1)^{1/2 - \beta/2} t^{- H/a_p - \delta + k - 1} dt \right).
\end{eqnarray*}
The first term is finite when $d - \beta - H/a_p - \delta > -1$, i.e. $\beta < d + 1 - H/a_p - \delta$, and the second term is bounded when $- \beta - H/a_p - \delta + k < -1$, that is, $\beta > k + 1 - H/a_p - \delta$. These together show \eqref{Ibeta1}. Hence, as anticipated, by the Frostman criterion we conclude that $dim_{Haus} \mathcal{G}(X_{ \lambda}) \geq d + 1 - H/a_p - \delta \ a.s.$ $\Box$\\

\section{The exponential of a matrix in Jordan canonical form} \label{s:Jordan_form}

Let $J_{\lambda} \in {\mathcal M}(n,\bbC)$ be a Jordan block of
size $n_\lambda$, whose expression is
\begin{equation} \label{e:Jordan_block}
J_{\vartheta} = \left( \begin{array}{ccccc}
\vartheta & 0 & 0 & \ldots & 0 \\
1   & \vartheta & 0 & \ldots & 0 \\
0   &    1   & \vartheta & \ldots & 0 \\
\vdots & \vdots & \vdots & \ddots & \vdots\\
0   &    0   & \ldots & 1 & \vartheta \\
\end{array} \right).
\end{equation}
Then, for $z > 0$,
\begin{equation} \label{e:z^Jlambda}
z^{J_{\vartheta}} = \left( \begin{array}{ccccc}
z^{\vartheta}      &    0    &    0    &  \ldots & 0\\
(\log z)z^{\vartheta} & z^{\vartheta} &    0     & \ldots & 0\\
\frac{(\log z)^{2}}{2!} z^{\vartheta} & (\log z)z^{\vartheta}   &    z^{\vartheta} & \ddots & 0 \\
\vdots   &  \vdots  &  \ddots   & \ddots  & 0 \\
\frac{(\log z)^{n_{\vartheta}-1}}{(n_{\vartheta}-1)!} z^{\vartheta} &
\frac{(\log z)^{n_{\vartheta}-2}}{(n_{\vartheta}-2)!} z^{\vartheta} &
\ldots &  (\log z)z^{\vartheta} & z^{\vartheta}\\
\end{array} \right).
\end{equation}
The expression for $z^{J}$, where $J$ is, more generally, a matrix
in Jordan canonical form (i.e., whose diagonal is made up of
Jordan blocks), follows promptly.

\section{Auxiliary technical results} \label{s:auxiliary}


Lemmas \ref{lem1.1}--\ref{lem1.4}, stated next, are used in Sections \ref{sec:time_domain} and \ref{sec:frequency_domain}. They are established in Bierm\'{e} et al.\ \cite{bierme:meerschaert:scheffler:2007}, and hence their proofs are omitted. Recall that, in the notation \eqref{e:a1<...<ai},
$$
a_1 < \hdots < a_p, \quad 1 \leq p \leq d,
$$
denote the $p$ distinct real parts of the eigenvalues of the matrix $E$. In particular,
$$
\varpi_E = a_1, \quad \Upsilon_E = a_p.
$$

\begin{lem}
\label{lem1.1}
Let $\|\cdot\|_0$ , $\tau_E(\cdot)$ be as in \eqref{e:S0} and \eqref{e:x=polar_coord}, respectively. For any (small) $\delta > 0$ there exist constants $C_1, C_2, C_3, C_4 > 0$ such that for all $\| {\bm x} \|_0 \leq 1$ or $\tau_E( {\bm x} ) \leq 1$
\begin{equation*}
C_1\| {\bm x} \|_0^{ \frac{1}{a_1} + \delta } \leq \tau_E( {\bm x} ) \leq C_2 \| {\bm x} \|_0^{ \frac{1}{a_p} - \delta }
\end{equation*}
and, for all $\| {\bm x} \|_0 \geq 1$ or $\tau_E( {\bm x} ) \geq 1$
\begin{equation*}
C_3 \| {\bm x} \|_0^{ \frac{1}{a_p} - \delta } \leq \tau_E( {\bm x} ) \leq C_4 \| {\bm x} \|_0^{ \frac{1}{a_1} + \delta }
\end{equation*}
\end{lem}

\begin{lem}
\label{l:polar_integration}
There exists a unique finite Radon measure $\sigma$ on $S_0$ such that for all $f \in L^1 (\mathbb{R}^d, d{\bm x})$ we have
\begin{equation}\label{e:polar_integration}
\int_{\mathbb{R}^d} f({\bm x}) d{\bm x} = \int_{0}^{\infty} \int_{S_0} f(r^E {\bm \theta}) \sigma(d {\bm \theta}) r^{q-1} dr
\end{equation}
where $q = \textnormal{tr}(E)$.
\end{lem}

\begin{rem}\label{r:converse_harmonizable_representation}
Let $h$ be a function that is continuous on $\bbR^d \backslash \{0\}$, and suppose we want to show that
\begin{equation}\label{e:exist_harmon_repres_integrability_condition}
\int_{\mathbb{R}^d} |h({\bm x})| d{\bm x} < \infty.
\end{equation}
It suffices to check the behavior of the kernel $h({\bm x})$ in \eqref{e:exist_harmon_repres_integrability_condition} for $x \in S_0$ and as $\| {\bm x}\| \rightarrow 0$ or $\infty$, where the latter limits can be equivalently expressed as $\tau_E({\bm x}) \rightarrow 0$ and $\infty$. This can be done conveniently by means of an entrywise application of the change-of-variables formula \eqref{e:polar_integration}. Since the latter formula assumes integrability of the original expression, one can build a truncation argument based on $h_A({\bm x}) = 1_{\{1/A \leq \tau_E({\bm x}) \leq A\}} h({\bm x})$, $A > 0$, and the dominated convergence theorem. It thus suffices to show that the condition
\begin{equation*}\label{e:int_polar_domain<infty}
\int^{\infty}_{0} \int_{S_0} |h(r^{E}\theta)| r^{\textnormal{tr}(E)-1} \sigma(d \theta) dr < \infty,
\end{equation*}
expressed in $E$-induced polar coordinates, holds.
\end{rem}

\begin{lem}
\label{lem1.2}
There exists a constant $K \geq 1$ such that for all ${\bm x}, {\bm y} \in \mathbb{R}^d$ we have
\begin{equation*}
\tau( {\bm x} + {\bm y} ) \leq K ( \tau( {\bm x} ) + \tau( {\bm y} ) )
\end{equation*}
\end{lem}
\begin{lem}
\label{lem1.4}
Let $\beta \in \mathbb{R}$, $q = \textnormal{tr}(E)$, and suppose $f \colon \mathbb{R}^d \rightarrow \mathbb{C}$ is measurable such that $| f({\bm x}) | = O(\tau( {\bm x} )^{\beta})$. If $\beta > -q$ then $f$ is integrable near 0, and if $\beta < -q$ then $f$ is integrable near infinity.
\end{lem}

The following proposition establishes an integral representation for the primary matrix function $K_{N}(x)$, and is used in the construction of $X^{Bes}_{\lambda}$ (see Definition \ref{defn:MA-B-TRF}).
\begin{prop}\label{p:Bessel}
For $N \in {\mathcal M}(n,\bbC)$, let $K_{N}(u)$ be a primary matrix function defined based on the modified Bessel function of the second kind. Then, expression \eqref{e:Bessel} holds, where $\cosh(N t) := (e^{-Nt}+e^{Nt})/2$, $t > 0$, is also a primary matrix function.
\end{prop}
\begin{proof}
For $u > 0$, it suffices to show that the function $K_{\nu}(u)$ is holomorphic (analytic) in $\nu$. So, fix $\nu_0 \in \bbC$ and for $\delta > 0$ consider any $\nu$ in the disc $D(\nu_0, \delta) \subseteq \bbC$. Let
$$
f(\nu,t) = e^{-u ( \frac{e^{-t}+e^{t}}{2})}\Big( \frac{e^{- \nu t}+e^{\nu t}}{2}\Big) \in \bbC.
$$
For any $t \geq 0$, $f(\nu,t)$ is holomorphic in $\nu$, where
$$
\frac{\partial}{\partial \nu}f(\nu,t) = e^{-u ( \frac{e^{-t}+e^{t}}{2})}t \hspace{0.5mm} \Big( \frac{e^{\nu t} - e^{-\nu t} }{2}\Big).
$$
In particular,
\begin{equation}\label{e:|deriv_f(nu,t)|}
\Big|\frac{\partial}{\partial \nu}f(\nu,t) \Big| \leq e^{-u ( \frac{e^{-t}+e^{t}}{2})} t \hspace{0.5mm} \Big( \frac{e^{|\nu_0 + \delta| t}+1 }{2}\Big),
\end{equation}
where the bounding function is integrable in $t$. So, rewrite
$$
\frac{1}{\nu - \nu_0}\Big\{\int^{\infty}_{0}f(\nu,t)dt - \int^{\infty}_{0}f(\nu_0,t) dt\Big\}
$$
\begin{equation}\label{e:(int-int)/nu-nu0}
=  \int^{\infty}_{0} \Re\Big( \frac{f(\nu,t) - f(\nu_0,t)}{\nu - \nu_0}\Big) + \imag \Im\Big( \frac{f(\nu,t) - f(\nu_0,t)}{\nu - \nu_0}\Big) dt.
\end{equation}
However, by the mean value theorem for $\bbC$-valued functions (Evard and Jafari \cite{evard:jafari:1992}, Theorem 2.2), there are $\nu_1(\nu_0,\nu,t), \nu_2(\nu_0,\nu,t) \in \bbC$ in the segment $(\nu_0,\nu) = \{z \in \bbC: z = \alpha \nu_0 + (1-\alpha), \alpha \in (0,1)\nu\}$ such that
\begin{equation}\label{e:derivRef(nu1)+derivImf(nu2)}
\Re \frac{\partial}{\partial \nu}f(\nu_1(\nu_0,\nu,t),t) + \imag \hspace{0.5mm}\Im \frac{\partial}{\partial \nu}f(\nu_2(\nu_0,\nu,t),t)
\end{equation}
Therefore, by expressions \eqref{e:derivRef(nu1)+derivImf(nu2)} and \eqref{e:(int-int)/nu-nu0}, as well as the bound \eqref{e:|deriv_f(nu,t)|}, the dominated convergence theorem implies that
$$
\frac{\partial}{\partial \nu} \int^{\infty}_0 f(\nu,t) dt= \int^{\infty}_0 \frac{\partial}{\partial \nu} f(\nu,t) dt.
$$
Therefore, $K_\nu(u)$ is holomorphic in $\nu$, as claimed.
\end{proof}

\bibliographystyle{agsm}

\bibliography{tempered}

\end{document}